\documentclass[12pt,reqno]{amsart}
\usepackage{hyperref}
\usepackage{amscd,amssymb,amsfonts,amsbsy,amsmath,verbatim,color, mathrsfs}
\usepackage{pst-node}
\usepackage{tikz-cd}
\usepackage{graphicx,cite}

\usepackage{float}
\usepackage{tikz,ifthen,bm}
\usepackage{verbatim,mathdots,caption,subcaption,epstopdf}
\usetikzlibrary{intersections}
\usetikzlibrary{calc}

\newtheorem{thm}{Theorem}[section]

\newtheorem{cor}[thm]{Corollary}

\newtheorem{lem}[thm]{Lemma}

\newtheorem{prop}[thm]{Proposition}
\newtheorem{exam}{Example}[section]
\theoremstyle{definition}
\newtheorem{defi}{Definition}[section]
\newtheorem{rema}[thm]{Remark}
\numberwithin{equation}{section}

\newcommand{\B}{{\mathcal B}}

\newcommand{\Real}{\mathbb{R}}

\newcommand{\R}{\Real}
\newcommand{\Z}{{\mathbb Z}}

\newcommand{\N}{{\mathbb N}}

\def\supp{{\textup{\textrm{supp}}}}

\makeatletter

\newcommand{\Rmnum}[1]{\expandafter\@slowromancap\romannumeral #1@}
\makeatother

\begin{document}
	\baselineskip=17pt
	\setcounter{figure}{0}
	
	\title[Kre\u{\i}n-Feller operators on Riemannian manifolds]
	{Kre\u{\i}n-Feller operators on Riemannian manifolds: Compactness of embedding and Hodge's theorem}
	\author[S.-M. Ngai]{Sze-Man Ngai}

	\address{Key Laboratory of High Performance Computing and Stochastic Information
		Processing (HPCSIP) (Ministry of Education of China), College of Mathematics and Statistics, Hunan Normal University,
		Changsha, Hunan 410081, China, and Department of Mathematical Sciences\\ Georgia Southern
		University\\ Statesboro, GA 30460-8093, USA.
		{\it Current address:} Beijing Institute of Mathematical Sciences and Applications, Huairou District, 101400, Beijing,  China.}
	
	\email{ngai@bimsa.cn}
	
	\author[L. Ouyang]{Lei Ouyang}
	\address{Key Laboratory of High Performance Computing and Stochastic Information
		Processing (HPCSIP) (Ministry of Education of China, College of Mathematics and Statistics, Hunan Normal University, Changsha, Hunan 410081, China.}
	\email{ouyanglei@hunnu.edu.cn}

	
	
	\thanks{The authors are supported in part by the National Natural Science Foundation of China, grants 12271156, and Construct Program of Key Discipline in Hunan Province.
		}

	\begin{abstract}
		For a bounded open set $\Omega$ in a complete oriented Riemannian $n$-manifold  and a positive finite Borel measure $\mu$ with support contained in $\overline{\Omega}$, we define an associated Kre\u{\i}n-Feller operators (or Laplacian) $\Delta_\mu$ by assuming the Poincar\'e inequalities for the measure $\mu$. We obtain sufficient conditions for the operator to have compact resolvent and in this case, we prove the Hodge theorem for functions, which states that there exists an orthonormal basis of $L^2(\Omega,\mu)$ consisting of eigenfunctions of $\Delta_\mu$, the eigenspaces are finite-dimensional, and the eigenvalues of $-\Delta_\mu$  are real, countable, and increasing to infinity. One of these sufficient conditions is that the lower $L^\infty$-dimension $\underline{\dim}_{\infty}(\mu)$ of $\mu$ is greater than $n-2$. 
		We prove that the compactness of embedding for functions also hold for measures without compact support, provided the manifold is of bounded geometry.  The main idea of our proof  is to use Toponogov's and Rauch's comparison theorems to extend a classical compact embedding theorem of Maz'ja to Riemannian manifolds. For a compact Riemannian manifold, using the above results, we also obtain sufficient conditions for Hodge Laplacian on $k$-forms, to have compact resolvent. Our result extends the classical Hodge theorem to Kre\u{\i}n-Feller operators. We study the condition $\underline{\dim}_{\infty}(\mu)>n-2$ for self-similar and self-conformal measures.  Results in this paper extend analogous ones by Hu \textit{et al.} in [J. Funct. Anal. \textbf{239} (2006), 542--565], which are established for measures on $\R^n$. 
		
	\end{abstract}
	
	\maketitle
	
	\tableofcontents

	\section{Introduction}\label{S:IN}
	\setcounter{equation}{0}
The Kre\u{\i}n-Feller operator plays an important role in studying analytic properties of measures, especially fractal measures, defined on an open subset $\Omega$ of $\R^n$ or a Riemannian  manifold. If the support of a measure is equal to $\overline{\Omega}$, then the Kre\u{\i}n-Feller operator is a natural generalization of the Laplace-Beltrami operator on $\Omega$ and encodes information of both the measure and the manifold. On  $\R^n$, suppose $\mu$ is a finite positive Borel measure on a bounded open subset $\Omega \subseteq\R^n$ with support contained in $\overline{\Omega}$.  It is well known that if  $\mu$ satisfies the Poincar\'e inequality on $\Omega$, then there exists a Dirichlet Laplacian $\Delta_\mu$ defined by $\mu$. $\Delta_\mu$ is also called a Kre\u{\i}n-Feller operator. Feller \cite{Feller_1957} studied such Laplacians in connection with diffusion processes. Kre\u{\i}n studied the eigenvalue problem for the vibrating string with finite mass (see \cite{Kac-Krein_1958, Krein_1952}). Spectral asymptotics of the Laplacian defined by the Cantor measure was studied by Mckean and Ray \cite{McKean-Ray_1962}. These Laplacians, as well as their generalizations, have been studied extensively in connection with fractal geometry, such as existence of an orthonormal basis of eigenfunctions, spectral dimension and spectral asymptotics, eigenvalues and eigenfunctions, eigenvalue estimates, differential equations,  nodal inverse problems, wave equations and wave speed, heat equation and heat kernel estimates, etc. (see \cite{Hu-Lau-Ngai_2006, Freiberg_2003, Freiberg_2005, Freiberg_2011, Freiberg-Zahle_2002, Fujita_1987, Gu-Hu-Ngai_2020, Naimark-Solomyak_1994, Naimark-Solomyak_1995, Ngai_2011, Ngai-Tang-Xie_2018, Ngai-Tang-Xie_2020, Ngai-Xie_2020, Ngai-Xie_2021, Bird-Ngai-Teplyaev_2003, Chen-Ngai_2010, Deng-Ngai_2015, Pinasco-Scarola_2019, Chan-Ngai-Teplyaev_2015, Pinasco-Scarola_2021, Kessebohmer-Niemann_2022-2,Tang-Ngai_2022,Kessebohmer-Niemann_2022-3} and references therein).  We remark
that Freiberg and her coauthors defined a class of Laplacians that are more general in that Lebesgue measure is replaced by a more general measure (see, e.g., \cite{Freiberg_2005,Freiberg-Zahle_2002}).  We also remark that Kigami \cite{Kigami_1993} defined Laplacians on a class of so-called post-critically finite fractals. These Laplacians are different from the Laplacians mentioned above. Kigami and Lapidus \cite{Kigami-Lapidus_1993} established an analogue of Weyl's classical theorem for the asymptotics of eigenvalues of Laplacians on finitely ramified self-similar fractals.

This paper is motivated by Yau's  conjecture on the nodal sets of eigenfunctions of the Laplacian on a compact Riemannian $n$-manifold $M$, which states that the nodal set of a $\lambda$-eigenfunction $u_\lambda$ on $M$ satisfy the following inequality:
	$$
	c_1\lambda^{1/2}\leq \mathcal{H}^{n-1}\big(\{x\in M:u_\lambda(x)=0\}\big)\leq c_2\lambda^{1/2},
	$$ 
	where $c_1$ and $c_2$ are positive constants and $\mathcal{H}^{n-1}$ is the $(n-1)$-dimensional Hausdorff measure. This conjecture has been studied by many authors including Donnelly and Fefferman \cite{Donnelly-Fefferman_1988,Donnelly-Fefferman_1990}, Hardt and Simon \cite{Hardt-Simon_1989}, Sogge and Zeldich \cite{Sogge-Zelditch_2012}, Logunov \cite{Logunov_2018_1,Logunov_2018_2}, and Logunov and Malinnikova \cite{Logunov-Malinnikova_2018}. 
	The lower estimate for Yau's conjecture was solved in 2018 by Logunov, using novel techniques developed by himself and Malinnikova (see \cite{Logunov_2018_1,Logunov_2018_2,Logunov-Malinnikova_2018}). 
	It turns out that for a Kre\u{\i}n-Feller operator on an interval $[a,b]$ defined by a measure $\mu$ with  $\supp(\mu)=[a,b]$,  Yau's conjecture takes the form
	\begin{align}\label{eq:Yau}
		c_1\lambda^{d_s/2}\leq \mathcal{H}^{0}(\{x\in [a,b]:u_\lambda(x)=0\})\leq c_2\lambda^{d_s/2},
	\end{align}
	where $d_s$ is the spectral dimension of the Kre\u{\i}n-Feller operator (see \cite{Bird-Ngai-Teplyaev_2003}).
The first author conjectured that for a Kre\u{\i}n-Feller operator on a compact Riemannian $n$-manifold $M$, if there exists a complete orthonormal basis consisting of eigenfunctions of $\Delta_\mu$ and if the eigenfunctions are continuous, then 
\begin{equation}
		\label{eq:Yau_general}
		c_1\lambda^{d_s/(2n)}\le\mathcal{H}^{n-1}(\{x\in M:u_\lambda(x)=0\})\le c_2\lambda^{d_s/(2n)}.
	\end{equation}
 To study this problem, we need to first make sure that the spectral dimension is well defined; this is the main motivation and objective of this paper. This problem is equivalent to whether Hodge's theorem for functions holds. Using Hodge's theorem in the present paper and assuming the existence of Green's function, Ngai and Zhao proved the nodal domain theorem and continuity of eigenfunctions for Kre\u{\i}n-Feller operators on a complete Riemannian manifold (see \cite{Ngai-Zhao_2024}).

Fractals on manifolds have been studied for over three decades. Strichartz\cite{Strichartz_1992} studied the spectral asymptotics of self-similar measures, the properties of Hausdorff measures and other fractal measures on the nilpotent Lie groups. 
Balogh and Tyson\cite{Balogh-Tyson_2005} studied  the Hausdorff dimensions of invariant sets associated to self-similar and self-affine iterated function systems in the Heisenberg group.  Patzschke\cite{Patzschke_1997} proved the multifractal formalism for self-conformal measures on Riemannian manifolds by assuming the open set condition. Ngai and Xu\cite{Ngai-Xu_2021,Ngai-Xu_2023} studied the existence of $L^q$-dimension and entropy dimension of self-conformal measures on Riemannian manifolds; they also obtained a formula for the Hausdorff dimension of a graph self-similar set generated by a graph-directed iterated function system satisfying the graph finite type condition.  Physicists Benedetti and Henson \cite{Benedetti-Henson_2009} observed that spacetime has fractal properties in that the spectral dimension function, which is defined by the Laplacian decreases from 4 to 2 near Planck scale. Calcagni studied fractal properties of spacetime by replacing Lebesgue measure by singular measures, and described physical phenomena with such measures (see, e.g., \cite{Calcagni_2010, Calcagni_2011}). These motivated our work to study Laplacians related to measures on Riemannian manifolds. General measures on manifolds have also been studied extensively. Chang {\em et al.}\cite{Chang-Gursky-Yang_2006} studied some conformal invariants of a Riemannian
manifold equipped with a smooth measure. Bueler\cite{Bueler_1999} studied the Hodge Laplacian on a weighted manifold with the heat-kernel being the weight. Grigor'yan~\cite{Grigor'yan_2006} studied the heat kernel on a weighted manifold. These weighted manifolds are endowed with a measure $\mu$ that has a smooth positive density with respect to Riemannian measure. However, as many fractal measures are singular with respect to  Riemannian measure, it is necessary to extend the framework of weight manifolds to allow more general measures; this is one of the objectives of this paper.

For the classical Laplacian $\Delta$ defined on Riemannian manifolds,  eigenfunctions, eigenvalues, eigenvalue estimates, and heat kernel estimates have been studied extensively by many authors and have played important roles in geometric analysis (see, e.g., \cite{Cheng-Peter-Yau_1981,Yau_1978,Minakshisundaram-Pleijel_1949,Schoen-Yau_1994,Berger-Gauduchon-Mazet_1971,Grigor'yan_2009,Peter_2012,Chavel_1984} and references therein). Hodge's theorem plays an important role in the spectral theory on manifolds and motivated this work. 

The classical Laplacian $\Delta$ extends from functions (i.e., $0$-forms) to differential forms of arbitrary degree, namely, the Laplace-Beltrami operator $\Delta^k:=dd^*+d^*d:\Gamma^\infty(\bigwedge^k M)\rightarrow \Gamma^\infty(\bigwedge^k M)$, where the definitions of $d$, $d^*$, and $\Gamma^\infty(\bigwedge^k M)$ are given in Section \ref{S:Hodge}. The spectrum of the Laplace-Beltrami operator $\Delta^k$ on differential forms has been investigated by many authors (see \cite{Chanillo-Treves_1997,Mantuano_2008,Gordon-Rossetti_2003} and references therein). In 1931, de Rham proved that de Rham cohomology groups are isomorphic to singular cohomology groups. This theorem gives the relationship between topology and smooth structures on manifolds. In 1933, Hodge established the Hodge theory while studying algebraic geometry. Hodge's theory took de Rham's work on de Rham's cohomology one step further. Hodge proved that there is a unique harmonic form in every de Rham cohomology class (see \cite{de_1955,Hodge_1941}). Since harmonic forms are solutions to elliptic partial differential equations on manifolds, this theorem establishes fundamental connections between analysis, geometry, and topology on manifolds. In 1964, Atiyah and Bott defined elliptic complex as a generalization of de Rham complex. The fundamental theorem of elliptic complexes is also a generalization of the Hodge theorem (see, e.g., \cite[Theorem 5.2]{Wells_2008}). The validity  of the Hodge theorem for forms is related to the Hodge theorem and the fundamental theorem concerning elliptic complexes. The Hodge theorem for forms (resp. functions) can be stated as follows.  Let $(M,g)$ be a compact connected oriented Riemannian manifold. Then there exists an orthonormal basis of $L^2(\bigwedge^kT^*M)$ (resp. $L^2(M)$) consisting of eigenforms (resp. eigenfunctions) of the Laplacian on $k$-forms $\Delta^k$ (resp. functions $\Delta$). All the eigenvalues are nonnegative, each eigenvalue has finite multiplicity, and the eigenvalues accumulate only at infinity (see, e.g., \cite[Theorem 1.30]{Rosenberg_1997}).

In this paper, we let $M$ be a complete oriented smooth Riemannian $n$-manifold supporting a measure $\mu$ with ${\rm supp}(\mu)\subseteq\Omega\subseteq\overline{\Omega}\subseteq M$,  where $\Omega$ is open in $M$. Our first goal is to define Kre\u{\i}n-Feller operators $\Delta_{\mu}^D$ and $\Delta_{\mu}^E$ on a Riemannian manifold related to a compactly supported measure $\mu$, in an effort to lay a foundation for studying spectral theory, spectral dimension function, heat kernel, etc., on Riemannian manifolds.  Our second goal is to prove an analog of Hodge's theorem for $\Delta_{\mu}^D$ (resp. $\Delta_{\mu}^E$) acting on functions, which states that under the assumption that $\underline{\dim}_{\infty}(\mu)>n-2$, there exists an orthonormal basis of $L^2(\Omega,\mu)$ (resp. $L^2(M,\mu)$) consisting of eigenfunctions of $\Delta_\mu^D$ (resp. $\Delta_{\mu}^E$), the eigenspaces are finite dimensional, and the eigenvalues of  $\Delta_{\mu}^D$ (resp. $\Delta_{\mu}^E$) are real, countable, and increasing to infinity. A main difficulty in the proof lies in  generalizing a compact embedding theorem of Maz'ja \cite{Maz'ja_1985} to Riemannian manifolds. To obtain such a generalization, we apply Toponogov's theorem and the Rauch comparison theorem. In addition, the proof also involves the continuity of the sectional curvature, the convergence of Maclaurin series, and properties of the exponential map. 	Our third goal is to generalize the results from the first two goals to measures without compact support.  A main difficulty in the proof lies again in  generalizing a compact embedding theorem of Maz'ja \cite{Maz'ja_1985} to non-compactly supported measures on Riemannian manifolds. We need to assume that the manifold is of bounded geometry. Our fourth goal is to prove an analog of Hodge's theorem for  $\Delta_{\mu}^{D,k}$ and $\Delta_{\mu}^{E,k}$ acting on $k$-forms. Moreover, when $\Omega=M$ and   $\mu$ is absolutely continuous with a positive density, we generalize the classical Hodge theorem. Our fifth goal is to prove several conditions equivalent to $\underline{\dim}_{\infty}(\mu)>n-2$, where $\mu$ is an invariant measure defined by an iterated function system of contractions on Riemannian manifolds.  We also remark that Kesseb\"ohmer and Niemann \cite{Kessebohmer-Niemann_2022-3} have recently shown that the condition $\underline{\dim}_{\infty}(\mu)>n-2$ is necessary for the validity of Hodge's theorem for functions.

This paper is organized as follows. In Section \ref{S:Pre*}, we summarize the construction of the Laplacian  defined by a measure $\mu$ with compact support and state some definitions that will be needed throughout the paper. Moreover, we state the main results in the paper. In Section \ref{S:L}, we define the Laplacian on a Riemannian manifold by assuming the Poincar\'e inequalities. We generalize the compact embedding theorem of Maz'ja\cite{Maz'ja_1985}. Moreover, we  prove Theorems \ref{thm:1.1} and \ref{thm:1.2}. 
In Section \ref{S:comp}, we prove Theorems \ref{thm:8.11}--\ref{thm:8.14}. In Section \ref{S:Hodge}, we define the Laplacian on space of $k$-forms. Moreover, we prove Theorems \ref{thm:1.1*} --\ref{thm:1.3*}. In Section \ref{S:self}, we prove Theorems \ref{thm:4.9} and \ref{thm:4.8}. 
	
	\section{Preliminaries and statement of main results}\label{S:Pre*}
	\setcounter{equation}{0}
	Throughout this paper, we assume that a Riemannian manifold is smooth, oriented, connected, and without boundary.  Let $(M,g)$ be a Riemannian $n$-manifold with Riemannian metric $g$. Let $\nu$ be the Riemannian volume on $M$, i.e.,
	$$\,d\nu=\sqrt{\det g_{ij}}\, dx,$$
	where the $g_{ij}$ are the components of $g$ in a coordinate chart, and $dx$ is the Lebesgue measure on $\R^n$. For any $F\subseteq M$, $\overline{F}$, $\partial F$, and $F^{\circ}$ denote, respectively, the closure, boundary, and interior of $F$. For an open set $\Omega\subseteq M$, $C_c(\Omega)$, $C^{\infty}(\Omega)$, and $C_c^{\infty}(\Omega)$ denote, respectively, the following spaces of functions on $\Omega$: continuous functions with compact support,  $C^{\infty}$ functions, and $C^{\infty}$ functions with compact support. For $u\in C^{\infty}(\Omega)$, the components of $\nabla u$ in local coordinates are given by $(\nabla u)_i=\partial_i u$ and let $|\nabla u|^2:=g^{ij}(\nabla u)_i(\nabla u)_j$, where $g^{ij}=(g_{ij})^{-1}$. Let $W^{1,2}(\Omega)$ be the Hilbert space equipped with the norm
	\begin{align}\label{eq:H}
		\|u\|_{W^{1,2}(\Omega)}=\Big(\int_{\Omega}|u|^2\,d\nu+\int_{\Omega}|\nabla u|^2\,d\nu\Big)^{\frac{1}{2}}.
	\end{align}
	The scalar product $\langle\cdot,\cdot\rangle_{W^{1,2}(\Omega)}$ associated to $\|\cdot\|_{W^{1,2}(\Omega)}$ is defined as
	$$ \left\langle u, v\right\rangle_{W^{1,2}(\Omega)} =\int_{\Omega}uv \,d\nu +\int_{\Omega}\langle\nabla u, \nabla v\rangle \,d\nu,$$
	where $\langle\cdot,\cdot\rangle= g(\cdot,\cdot)$.  Let $W^{1,2}_0(\Omega)$ denote the closure of $C^\infty_c(\Omega)$ in the $W^{1,2}(\Omega)$ norm. Note that if $\partial\Omega=\emptyset$, then  $\Omega=M$, as $M$ is assumed to be connected.
	
Now assume in addition that $M$ is complete. Let $\Omega\subseteq M$ be a bounded open set. Then $\overline{\Omega}$ is compact. If, in addition, $\partial\Omega=\emptyset$, then according to above, $\Omega=M$ is a compact set and thus $W_0^{1,2}(M)=W^{1,2}(M)$. 

	Let $\mu$ be a positive bounded regular Borel measure on $M$ with ${\rm supp}(\mu)\subseteq\overline{\Omega}$. Let $L^2(\Omega,\mu)$ be the space of all measurable functions $u$ with respect to $\mu$ on $\Omega$ satisfying $\int_{\Omega}|u|^2\,d\mu<\infty$, with the inner product $\langle\cdot,\cdot\rangle_{L^2(\Omega,\mu)}$ and associated norm $\|\cdot\|_{L^2(\Omega,\mu)}$ defined respectively as
		\begin{align*}
		\langle u, v\rangle_{L^2(\Omega,\mu)}:=\int_\Omega  uv\,d\mu\quad\text{and}\quad \|u\|_{L^2(\Omega,\mu)}:=\Big(\int_{\Omega}|u|^2\,d\mu\Big)^{1/2}.
	\end{align*}
Define
	\begin{align*}
		&\mathcal{C}(M):=\Big\{u\in C^{\infty}_c(M):\int_M u\,d\nu=0\Big\},\quad W(M):=\Big\{u\in W^{1,2}_0(M):\int_M u\,d\nu=0\Big\},\,\,\text{and}\\
			&\mathcal{C}:=\{c\cdot 1_M:c\in\R\},\,\,\text{where $1_M$ denotes the function identically equal to 1 on $M$}.
	\end{align*}
Note that $W(M)$ is a closed subspace of $ W^{1,2}_0(M)$.
	 Our method of defining  Kre\u{\i}n-Feller operators $\Delta_{\mu}^D$ and $\Delta_{\mu}^E$ on $L^2(\Omega, \mu)$ is similar to that in \cite{Hu-Lau-Ngai_2006}. First, we introduce the following {\em Poincar\'e inequalities for the measure $\mu$}, depending on the boundary conditions:
	 	
\noindent($\partial\Omega\neq\emptyset$)   If there exists a constant $C>0$ such that for all $u\in C^{\infty}_c(\Omega)$,
		\begin{align}\label{eq:PI}
			\int_\Omega |u|^2\,d\mu\leq C\int_\Omega |\nabla u|^2\,d\nu,
		\end{align}
	then we say that the {\em Poincar\'e inequality holds for the measure $\mu$ and the case $\partial\Omega\neq\emptyset$ (MPID)}.
	
\noindent ($\partial\Omega=\emptyset$) Note that in this case, $M$ is compact. If there exists a constant $C>0$ such that for all $u\in \mathcal{C}(M)$,
	\begin{align}\label{eq:PIN}
		\int_M |u|^2\,d\mu\leq C\int_M |\nabla u|^2\,d\nu,
	\end{align}
	then we say that the {\em Poincar\'e inequality holds for the measure $\mu$ and the case $\partial\Omega=\emptyset$ (MPIE)}.

Let $u\in W^{1,2}_0(\Omega)$. Then there exists a sequence $\{u_m\}$ in $C^{\infty}_c(\Omega)$ converging to $u$ in the $W^{1,2}_0(\Omega)$-norm. Hence $\{u_m\}$ has a subsequence $\{u_{m_k}\}$ that converges $\nu$-a.e. to $u$. Let $E\subseteq \Omega$ such that $\nu(E)=\nu(\Omega)$ and $\lim_{k\rightarrow\infty}u_{m_k}(x)=u(x)$ for all $x\in E$.
By $(\rm MPID)$,  $\{u_{m_k}\}$ is a Cauchy sequence in $L^2(\Omega, \mu)$ and thus converges to some $\widetilde{u}$ in the $L^2(\Omega, \mu)$-norm.  $\{u_{m_k}\}$ in turn has a subsequence $\{u_{m_{k_j}}\}$ that converges pointwise $\mu$-a.e. to $\widetilde{u}$. Hence there exists $\widetilde{E}\subseteq \Omega$, with $\mu (\widetilde{E})=\mu(\Omega)$, such that
$\lim_{j\rightarrow\infty}u_{m_{k_j}}(x)=\widetilde{u}(x)$ for all $x\in \widetilde{E}$.
Define $\overline{u}$ on $\Omega$ as
\[
\overline{u}(x):=\left\{
\begin{array}{ll}
	u(x)=\widetilde{u}(x),\quad &x\in E\bigcap \widetilde{E};\\
	u(x),\quad\quad\quad\quad &x\in E\backslash \widetilde{E};\\
	\widetilde{u}(x),\quad\quad\quad\quad &x\in\widetilde{E}\backslash E;\\
	0,\quad\quad\quad\quad\quad &x\in\Omega\backslash (E\bigcup \widetilde{E}) .
\end{array}
\right.
\]
Hence $\overline{u}=u$ $\nu$-a.e. and $\overline{u}=\widetilde{u}$ $\mu$-a.e. Thus $\overline{u}\in W^{1,2}_0(\Omega)\cap L^2(\Omega, \mu)$ exists. In other words,
(MPID) (resp. (MPIE)) implies that each equivalence class $u \in W^{1,2}_0(\Omega)$ (resp. $u \in W(M)$) contains a unique (in $L^{2}(\Omega, \mu)$ sense) member $\overline{u}$ that belongs to $L^{2}(\Omega, \mu)$ (resp. $L^{2}(M, \mu)$) and satisfies both conditions below:
	\begin{enumerate}
		\item[(1)] There exists a sequence $\left\{u_{k}\right\}$ in $C_{c}^{\infty}(\Omega)$ (resp. $\mathcal{C}(M)$) such that $u_{k} \rightarrow \overline{u}$ in $W^{1,2}_0(\Omega)$ (resp. $W(M)$) and $u_{k} \rightarrow \overline{u}$ in $L^{2}(\Omega, \mu)$ (resp. $L^{2}(M, \mu)$);
		\item[(2)] $\overline{u}$ satisfies the inequality in \eqref{eq:PI} (resp. \eqref{eq:PIN}).
	\end{enumerate}
	We call $\overline{u}$ the {\em $L^{2}(\Omega, \mu)$-representative} of $u$ (resp. {\em $L^{2}(M, \mu)$-representative} of $u$). Assume  $\mu$ satisfies (MPID) (resp. (MPIE)) and define a mapping $I_D: W^{1,2}_0(\Omega) \rightarrow L^{2}(\Omega, \mu)$ (resp. $I_W: W(M) \rightarrow L^{2}(M, \mu)$) by
	$$
	I_D(u)=\overline{u}\qquad(\text{resp.}\,\,I_W(u)=\overline{u}).
	$$
	Next, we notice that $I_D$ and $I_W$ are bounded linear operators but are not necessarily injective. Hence we consider a subspace $\mathcal{N}_D$ of $W^{1,2}_0(\Omega)$  defined as
	$$
	\mathcal{N}_D:=\left\{u \in W^{1,2}_0(\Omega):\|I_D(u)\|_{L^{2}(\Omega,\mu)}=0\right\}.$$
Similarly, we can define 
	$$	\mathcal{N}_W:=\left\{u \in W(M):\|I_W(u)\|_{L^{2}(M,\mu)}=0\right\}.$$
	Since $\mu$ satisfies (MPID) (resp. (MPIE)), $\mathcal{N}_D$ (resp. $\mathcal{N}_W$) is a closed subspace of $W^{1,2}_0(\Omega)$ (resp. $W(M)$). Let $(\mathcal{N}_D)^{\perp}$ (resp. $(\mathcal{N}_N)^{\perp}$) be the orthogonal complement of $\mathcal{N}_D$ in $W^{1,2}_0(\Omega)$ (resp. $\mathcal{N}_N$ in $W(M)$). Then $I_D: (\mathcal{N}_D)^{\perp} \rightarrow L^{2}(\Omega, \mu)$ (resp. $I_N: (\mathcal{N}_N)^{\perp} \rightarrow L^{2}(M, \mu)$) is injective. Throughout the rest of this paper, we will denote $\overline{u}$ simply by $u$.
	
	Now, we consider a nonnegative bilinear form $\mathcal{E}_D(\cdot, \cdot)$ (resp. $\mathcal{E}_W(\cdot, \cdot)$) on $L^{2}(\Omega, \mu)$ (resp. $L^{2}(M, \mu)$) defined as
	\begin{align}\label{eq(1.1)}
     \mathcal{E}_D(u, v):=\int_{\Omega} \langle\nabla u, \nabla v\rangle\,  d \nu\qquad\Big(\text{resp.} \,\,\mathcal{E}_W(u, v):=\int_{M} \langle\nabla u, \nabla v\rangle\,  d \nu\Big)	
\end{align}
	with $\operatorname{dom}(\mathcal{E}_D)=(\mathcal{N}_D)^{\perp}$ (resp. $\operatorname{dom}(\mathcal{E}_W)=(\mathcal{N}_W)^{\perp}$). Let $\mathcal{E}_E(\cdot, \cdot)$ be a nonnegative bilinear form on $L^{2}(M, \mu)$  defined as
		\begin{align}\label{eq:new}
		\mathcal{E}_E(u, v):=\int_{M} \langle\nabla u, \nabla v\rangle\,  d \nu
\end{align}	
	and let $\operatorname{dom}(\mathcal{E}_E):=\mathcal{C}\oplus \operatorname{dom}(\mathcal{E}_W)$. In Section \ref{S:L}, we show that $(\mathcal{E}_D, \operatorname{dom}(\mathcal{E}_D))$ (resp. $(\mathcal{E}_E, \operatorname{dom}(\mathcal{E}_E)))$ is closed quadratic form on $L^{2}(\Omega,\mu)$ (resp. $L^{2}(M,\mu)$) (see Proposition \ref{prop:2.3}). Hence there exists a nonnegative definite self-adjoint operator $H$ on $L^{2}(\Omega,\mu)$ such that $
\operatorname{dom}(\mathcal{E}_D)=\operatorname{dom}(H^{1/2})$ and
	$$\mathcal{E}_D(u,v)=\langle H^{1/2}u, H^{1/2}v \rangle_{L^{2}(\Omega,\mu)} \quad \text{for all}\,\, u, v\in \operatorname{dom}(\mathcal{E}_D).$$
	Finally, we write $\Delta_{\mu}^D=-H$, and call it the \textit{Dirichlet Laplacian (or Kre\u{\i}n-Feller operator) defined by $\mu$}. Similarly, we can define $\Delta_\mu^E$.  This completes the construction of the Laplacian. If $M=\R^n$ and $d\nu=dx$, then $\Delta_{\mu}^D$  is just the standard Laplacian on $\Omega\subseteq\R^n$ (see \cite{Hu-Lau-Ngai_2006}). 
	
	Let $(\mathfrak{X},d)$ be a metric space. We denote the {\em Euclidean distance} by $d_{E}(\cdot, \cdot)$. For a connected Riemannian $n$-manifold $M$, we denote the {\em Riemannian distance} by $d_M(\cdot,\cdot)$. Fix an arbitrary point $q_0\in M$. For $p\in M$, define $|p|_M:=d_M(q_0,p)$. For any $p\in M$, we let $T_pM$ be the {\em tangent space} of $M$ at $p$ and let  $TM:=\bigcup_{p\in M}T_p M$.   For $\epsilon>0$, let
	\begin{align*}
		&B^{\mathfrak{X}}(x,\epsilon):=\{y\in \mathfrak{X}:d(x,y)<\epsilon\}, \quad x\in \mathfrak{X},\\
		&B(x,\epsilon):=\{y\in \R^n: d_E(x,y)<\epsilon\}, \quad x\in \R^n,\\
		&B^M(p,\epsilon):=\{q\in M: d_M(p,q)<\epsilon\},\quad p\in M,\\
		&B^{T_pM}(0,\epsilon):=\{\xi\in T_pM: |\xi|<\epsilon\},\\
		&S^{T_pM}(0,\epsilon):=\{\xi\in T_pM: |\xi|=\epsilon\}.
	\end{align*}
 Let $\mu$ be a positive bounded regular Borel measure on $\mathfrak{X}$. Recall that the  {\em lower $L^{\infty}$-dimension of $\mu$}
	is defined as
	\begin{align*}
		\underline{\dim}_\infty(\mu):=\displaystyle{\varliminf_{\delta\to 0^+}}\frac{\ln (\sup_x \mu(B^{\mathfrak{X}}(x,\delta)))}{\ln \delta},
	\end{align*}
	where the supremum is taken over all $x\in{\rm supp}(\mu)$. Similarly, one can define $\overline{\dim}_\infty(\mu)$.
	If the limit exists, we denote the common value by $\dim_\infty(\mu)$. Let
	\begin{align*}
		\underline{\dim}_{\mu}(x):=\displaystyle{\varliminf_{\delta\to 0^+}}\frac{\ln\mu(B^{\mathfrak{X}}(x,\delta))}{\ln \delta}
	\end{align*}
	be the {\em lower local dimension} of $\mu$ at $x$.
	
 We have the following main theorems.
	\begin{thm}\label{thm:1.1}
	Let $n\geq 1$, $M$ be a complete smooth connected oriented Riemannian $n$-manifold, and $\Omega \subseteq M$ be a bounded open set. Let $\mu$ be a positive finite Borel measure on $M$ such that $\operatorname{supp}(\mu) \subseteq \overline{\Omega}$ and $\mu(\Omega)>0$. Assume that $\underline{\operatorname{dim}}_{\infty}(\mu)>n-2$. 
\begin{enumerate}
	\item[(a)]Assume  $\partial\Omega \neq \emptyset$. Then  {\rm (MPID)} holds, and the embedding $\operatorname{dom}(\mathcal{E}_D) \hookrightarrow L^{2}(\Omega,\mu)$ is compact.
	\item[(b)] Assume $\partial\Omega=\emptyset$. Then  {\rm (MPIE)} holds, and the embedding $\operatorname{dom}(\mathcal{E}_E) \hookrightarrow L^{2}(M,\mu)$ is compact, where $M$ is compact.
		\end{enumerate}
	\end{thm}
	
 Hu {\em et al.} \cite{Hu-Lau-Ngai_2006} proved that for a positive finite Borel measure $\mu$ defined on $\R^n$, Theorem \ref{thm:1.1}(a) holds. The main ingredient of their proof is a compact embedding theorem of Maz'ja (see Theorem \ref{thm:maz}). However, this compact embedding theorem cannot be applied directly to manifolds; we need to first generalize it to Riemannian manifolds.  By using the Toponogov theorem and the Rauch comparison theorem, along with  the continuity of the sectional curvature, the convergence of Maclaurin series, and properties of the exponential map, we can prove this compact embedding theorem (see Theorem \ref{thm:3.11}).  As a result, we have the following theorem.
	
	\begin{thm}\label{thm:1.2}
	Let $n$, $M$, $\Omega$, and $\mu$ be as in Theorem \ref{thm:1.1}.   Assume that $\underline{\operatorname{dim}}_{\infty}(\mu)>n-2$. 
	\begin{enumerate}
				\item[(a)] 	Assume $\partial\Omega \neq \emptyset$.  Then there exists an
				orthonormal basis $\left\{u_{k}\right\}_{k=1}^{\infty}$ of $L^{2}(\Omega, \mu)$ consisting of  eigenfunctions of $-\Delta_{\mu}^D.$ The eigenvalues $\left\{\lambda_{k}\right\}_{k=1}^{\infty}$ satisfy $0<\lambda_{1} \leq \lambda_{2} \leq \cdots$. 
				\item[(b)] 	Assume $\partial\Omega =\emptyset$, and hence $M$ is compact and $\Omega=M$. Then there exists an
				orthonormal basis $\left\{u_{k}\right\}_{k=1}^{\infty}$ of $L^{2}(M, \mu)$ consisting of  eigenfunctions of $-\Delta_{\mu}^E$. The eigenvalues $\left\{\lambda_{k}\right\}_{k=1}^{\infty}$ satisfy $0\leq\lambda_{1} \leq \lambda_{2} \leq \cdots$.
		\end{enumerate}
	Moreover, if $\dim(\operatorname{dom}(\mathcal{E}_D)) =\infty$ in (a) or $\dim(\operatorname{dom}(\mathcal{E}_E)) =\infty$ in (b), then  $\lim _{k \rightarrow \infty} \lambda_{k}=\infty$ and each eigenspace is finite-dimensional.
	\end{thm}

Let $M$ be a complete smooth connected Riemannian $n$-manifold. We now consider measures $\mu$ without compact support. The proof of Theorem \ref{thm:8.11}, analog of Theorem \ref{thm:3.11}, is more difficult in this case. In the proof of Theorem \ref{thm:3.11}, since the support of the measure $\mu$ is compact, there exists a finite family of normal coordinate charts covering ${\rm\supp}(\mu)$, allowing us to use a diagonal argument to extract a convergent subsequence in finitely many steps. We describe two main difficulties in the non-compact case and our ideas to overcome them. First, the number of coordinate charts covering  ${\rm\supp}(\mu)$ could be infinite. In proving the sufficiency of Theorem \ref{thm:8.11},  the method we use in the compact case to extract a convergent subsequence in finitely many steps no longer works.
To overcome this difficulty, we use the fact that for manifolds of bounded geometry, there exists a uniformly locally finite geodesic atlas on $M$ (see Lemma \ref{lem:8.2}). We use such an atlas, a partition of unity, and a second equivalent definition of Sobolev norm to extract a convergent subsequence. Second, in order to use the compact embedding theorem on $\R^n$ (see Theorem \ref{thm:maz}) to prove the necessity of Theorem \ref{thm:8.11}, our idea is to perform a partition of unity on each member of a sequence of functions in the unit ball of the space of the Sobolev of functions on $\R^n$, then use normal coordinate charts to pull the components to the manifold, and finally extract a convergent subsequence on $M$ and use normal coordinate charts to map it to a convergent subsequence of $\R^n$. However, there are countably many normal coordinate charts on $M$ with overlaps, making it impossible to use a similar argument as in the compact case. We deal with this problem by dividing the uniformly locally finite cover of $M$ into a finite number of layers so that the geodesic balls in each layer are disjoint, which allows us to construct a    convergent subsequence after a finite number of steps.

	Next we state some definitions. Let $M$ be a Riemannian $n$-manifold and let $X,Y$ be any basis for a $2$-plane $\Pi\in T_pM$. The {\em sectional curvature} of $M$ associated with $\Pi$ is defined as
		$$K_M:=K(X,Y)=\frac{{\rm Rm}(X,Y,Y,X)}{|X|^2|Y|^2-\langle X,Y\rangle^2},$$	
		where ${\rm Rm}:\Gamma(M)\times\Gamma(M)\times\Gamma(M)\times\Gamma(M)\rightarrow \R$ is the {\em Riemann curvature tensor} defined as  $${\rm Rm}(X,Y,Z,W)=\langle\nabla_X\nabla_YZ-\nabla_Y\nabla_XZ-\nabla_{[X,Y]}Z,W\rangle$$
		with $[X,Y]=XY-YX$ being the Poisson bracket. The {\em Ricci curvature} Ric is the trace of Rm.
		
	For the proof of Theorem \ref{thm:8.11}, because of the lack of compactness, we cannot use the method in Section \ref{S:L} to control the volume of a ball in $\R^n$ and that of its image in a manifold under a normal coordinate mapping. Hence we assume that the manifold is of bounded geometry (see Definition \ref{defi:2.3}), which ensures that the injective radius is positive and the sectional curvature is bounded, so we can use the techniques of Theorem \ref{thm:1.1} to show that the local density property of the $\widetilde{\mu}$ in $\R^n$ is equivalent to the local density property of the $\mu$ on the manifold, where $\widetilde{\mu}$ is the measure induced on $\R^n$ by the local normal coordinates. The following definition can be found in, for example, \cite{Shubin_1992}.
		\begin{defi}\label{defi:2.3}
			Let $M$ be a Riemannian $n$-manifold. We say that $M$ is of {\em bounded geometry} if the following two conditions are satisfied.
			\begin{enumerate}
				\item[(i)] The injectivity radius of $M$ is positive.
				\item[(ii)] Every covariant derivative of the Riemann curvature tensor ${\rm Rm}$ of $M$ is bounded, i.e., for all $k\in \N_0$, there exists a constant $C_k>0$ such that 
				$$|\nabla^k{\rm Rm}|\leq C_k.$$
			\end{enumerate}	
		\end{defi}
		
For examples of manifolds of bounded geometry, such as compact manifolds, Lie groups, homogeneous spaces, covering manifolds of compact manifolds, etc., we refer the reader to \cite{Cheeger-Gromov-Taylor_1982,Roe_1988, Shubin_1992}.
If the injectivity radius of $M$ is positive, then $M$ is complete (see, e.g.,\cite[Proposition 1.2a]{Eichhorn_2007}).

		\begin{thm}\label{thm:8.11}
		Let $n\geq 2$, $M$ be a smooth connected oriented Riemannian $n$-manifold of bounded geometry. Let $\mu$ be a positive finite Borel measure on $M$ (need not have compact support).  For $q>2$, the unit ball $B:=\{u \in C_{c}^{\infty}\left(M\right):\|u\|_{W^{1,2}_0\left(M\right)} \leq 1\}$
			is relatively compact in $L^{q}\left(M, \mu\right)$ if and only if
		\begin{align}
				&\lim _{\delta \rightarrow 0^{+}} \sup _{w \in M ; r \in(0, \delta)} r^{1-n/2} \mu(B^M(w,r))^{1/q}=0,  \label{eq:m1}\\
				&\lim _{|w|_M \rightarrow \infty} \sup _{r\in (0,1)} r^{1-n/2} \mu\left(B^M(w,r)\right)^{1 / q}=0   \qquad \text { \rm for } n>2, \label{eq:m2}
					\end{align}
				and
				\begin{align}
				&\lim _{\delta \rightarrow 0^{+}} \sup _{w \in M ; r \in(0, \delta)}|\ln r|^{1/2} \mu(B^M(w,r))^{1/q}=0, \label{eq:m3} \\ 
				&\lim _{|w|_M \rightarrow \infty} \sup _{r\in (0,1)} |\ln r|^{1/2} \mu(B^M(w,r))^{1/q}=0\qquad \text { \rm for } n=2. \label{eq:m4}
			\end{align}
		\end{thm}
			Theorem \ref{thm:8.11} generalizes a compact embedding theorem of Maz'ja (see Theorem \ref{thm:maz}) and plays a major role in proving Theorems \ref{thm:8.13} and \ref{thm:8.14}.

			The definitions of (PID), (PIE), $\operatorname{dom}(\mathcal{E}_D)$, $\operatorname{dom}(\mathcal{E}_E)$, $\Delta_{\mu}^D$, and $\Delta_{\mu}^E$ in the following theorems are given in Section \ref{S:comp}. (MPID) and (MPIE) in the following theorem are given in \eqref{eq:PI} and \eqref{eq:PIN}, respectively, where $\Omega\subseteq M$ is an open subset and $M$ is a complete Riemannian $n$-manifold.
			
	\begin{thm}\label{thm:8.13}	Let $n\geq 2$, $M$ be a smooth connected oriented Riemannian $n$-manifold of bounded geometry and let $\Omega\subseteq M$ be an open subset.  Let $\mu$ be a positive finite Borel measure on $M$ (need not have compact support) and assume that  $\supp(\mu)\subseteq \overline{\Omega}$. Assume $\underline{\operatorname{dim}}_{\infty}(\mu)>n-2$.
				\begin{enumerate}
			\item[(a)]	Assume that $\partial\Omega \neq \emptyset$ and {\rm (PID)} holds. Then {\rm (MPID)} holds, and the embedding $\operatorname{dom}(\mathcal{E}_D) \hookrightarrow L^{2}(\Omega,\mu)$ is compact.
			\item[(b)] Assume that $\partial\Omega =\emptyset$  and {\rm (PIE)} holds. Then {\rm (MPIE)} holds, and the embedding $\operatorname{dom}(\mathcal{E}_E) \hookrightarrow L^{2}(M,\mu)$ is compact.	
		\end{enumerate}
		\end{thm}
		
		Using the results of Theorem \ref{thm:8.13}, we obtain the following theorem.
\begin{thm}\label{thm:8.14}
				Let $n$, $M$, $\Omega$, and $\mu$ be as in Theorem \ref{thm:8.13}. Assume that $\underline{\operatorname{dim}}_{\infty}(\mu)>n-2$.
		\begin{enumerate}
		\item[(a)]	Assume that $\partial\Omega \neq \emptyset$ and {\rm (PID)} holds. Then there exists an
		orthonormal basis $\left\{u_{k}\right\}_{k=1}^{\infty}$ of $L^{2}(\Omega, \mu)$ consisting of  eigenfunctions of $-\Delta_{\mu}^D.$ The eigenvalues $\left\{\lambda_{k}\right\}_{k=1}^{\infty}$ satisfy $0<\lambda_{1} \leq \lambda_{2} \leq \cdots$.
		\item[(b)]	Assume that $\partial\Omega =\emptyset$  and {\rm (PIE)} holds.  Then there exists an
		orthonormal basis $\left\{u_{k}\right\}_{k=1}^{\infty}$ of $L^{2}(M, \mu)$ consisting of  eigenfunctions of $-\Delta_{\mu}^E.$ The eigenvalues $\left\{\lambda_{k}\right\}_{k=1}^{\infty}$ satisfy $0\leq\lambda_{1} \leq \lambda_{2} \leq \cdots$.
	\end{enumerate}	
Moreover, if $\dim(\operatorname{dom}(\mathcal{E}_D)) =\infty$ in (a) or $\dim(\operatorname{dom}(\mathcal{E}_E)) =\infty$ in (b), then  $\lim _{k \rightarrow \infty} \lambda_{k}=\infty$ and each eigenspace is finite-dimensional.				
		\end{thm}
	

	The definitions of (MPID*), (MPIE*), $\operatorname{dom}(\mathcal{E}_D^k)$,  $\operatorname{dom}(\mathcal{E}_E^k)$, $\Delta_{\mu}^{D,k}$, $\Delta_{\mu}^{E,k}$, and $\operatorname{dom}(\Delta_{\mu}^{E,k})$ in the following two theorems are given in Section \ref{S:Hodge}.
	\begin{thm}\label{thm:1.1*}
		Let $n\geq 1$, $M$ be a compact connected Riemannian $n$-manifold, and $\Omega \subseteq M$ be an open set. Let $\mu$ be a positive finite Borel measure on $M$ such that $\operatorname{supp}(\mu) \subseteq \overline{\Omega}$ and $\mu(\Omega)>0$. Assume that $\underline{\operatorname{dim}}_{\infty}(\mu)>n-2$. 
		\begin{enumerate}
			\item[(a)] Assume that  $\partial\Omega \neq \emptyset$. Then  {\rm (MPID*)} holds, and the embedding $\operatorname{dom}(\mathcal{E}_D^k) \hookrightarrow L^2\big(\bigwedge^kT^*\Omega,\mu\big)$ is compact.
			\item[(b)] Assume that $\partial\Omega=\emptyset$. Then  {\rm (MPIE*)} holds, and the embedding $\operatorname{dom}(\mathcal{E}_E^k) \hookrightarrow L^2\big(\bigwedge^kT^*M,\mu\big)$ is compact.
	\end{enumerate}
	\end{thm}
	As a result, we have the following theorem.
	\begin{thm}\label{thm:1.2*}
		Assume the same hypotheses of Theorem \ref{thm:1.1*}. 
	\begin{enumerate}
				\item[(a)]	Assume  $\partial\Omega \neq \emptyset$. Then there exists an
				orthonormal basis of $L^2\big(\bigwedge^kT^*\Omega,\mu\big)$ consisting of  eigenforms of $-\Delta_{\mu}^{D,k}$, where $0\leq k\leq n$. The eigenvalues $\left\{\lambda_{m}\right\}_{m=1}^{\infty}$ satisfy $0<\lambda_{1} \leq \lambda_{2} \leq \cdots$. 
				
				\item[(b)] Assume $\partial\Omega=\emptyset$. Then there exists an
				orthonormal basis of $L^2\big(\bigwedge^kT^*M,\mu\big)$ consisting of  eigenforms of $-\Delta_{\mu}^{E,k}$, where $0\leq k\leq n$. The eigenvalues $\left\{\lambda_{m}\right\}_{m=1}^{\infty}$ satisfy $0\leq\lambda_{1} \leq \lambda_{2} \leq \cdots$.
		\end{enumerate}
	Moreover, if $\dim(\operatorname{dom}(\mathcal{E}_D^k)) =\infty$ in (a) or $\dim(\operatorname{dom}(\mathcal{E}_E^k)) =\infty$ in (b), then  $\lim _{m \rightarrow \infty} \lambda_{m}=\infty$ and each eigenspace is finite-dimensional.
	\end{thm}

	\begin{defi}
		Let $M$ be a smooth Riemannian $n$-manifold and let $\mu$ be a positive finite Borel measure on $M$ such that $\operatorname{supp}(\mu) \subseteq M$. The {\em space of harmonic $k$-fields} is defined as
		\begin{align*}
			\widetilde{\mathcal{H}}^k(M):=\Big\{\omega\in W^{1,2}\big({\bigwedge}^kT^*M\big):d\omega=d^*\omega=0\Big\},
		\end{align*}
		where $d$ and $d^*$ are defined in \eqref{eq:d*} and \eqref{eq:ad*}, respectively. Let $(\widetilde{\mathcal{H}}^k(M))^{\perp}$ be the orthogonal complement space of $\widetilde{\mathcal{H}}^k(M)$.  Note that $	\widetilde{\mathcal{H}}^k(M)\subseteq \Gamma^\infty(\Lambda^kT^*M)$. The {\em space of harmonic $k$-forms}, denoted $\mathcal{H}^k(M)$,  is defined as
		\begin{align*}
			\mathcal{H}^k(M):=\Big\{\omega\in \operatorname{dom}(\Delta^k):\Delta^k\omega=0\Big\}.
		\end{align*}
		Let
		\begin{align*}
			\mathcal{H}^k_{\mu}(M):=\big\{\omega\in \operatorname{dom}(\Delta_{\mu}^{E,k}):\Delta_{\mu}^{E,k}\omega=0\big\},
		\end{align*}
		where $\Delta_{\mu}^{E,k}$ is defined in Section \ref{S:Hodge}. The element of $\mathcal{H}^k_{\mu}(M)$ are called {\em $\mu$-harmonic $k$-forms}.
		
	\end{defi}
	Assume that $\Omega=M$ and   $\mu$ is absolutely continuous with a positive density. The following theorem generalizes the classical Hodge theorem.

	\begin{thm}\label{thm:1.3*}
		Let $M$ be a compact Riemannian $n$-manifold. Let $\mu$ be a positive finite measure that is absolutely continuous with respect to the Riemannian volume form and has a positive and bounded  density. Then each cohomology class of  forms in $W^{1,2}(\bigwedge^kT^*M)$ contains a unique $\mu$-harmonic $k$-form.
	\end{thm}
	In Example \ref{exam:1*}, we show that the conclusion of Theorem \ref{thm:1.3*} need not hold if $\mu$ is not absolutely continuous with respect to the Riemannian volume form.	
	
	In Section \ref{S:self}, we consider two types of measures, self-similar measures and self-conformal measures. Let $\{S_i\}_{i=1}^{m}$ be a finite set of {\em contraction maps} on $M$, i.e., for each $i$, there exists $r_i$ with $0<r_i<1$ such that
	\begin{align}\label{eq:j1.1}
		d_M(S_i(x),S_j(y))\leq r_i d_M(x,y)\quad \text{for all}\,\, x, y\in M.
	\end{align}
	Then there exists a unique nonempty compact set $K$ such that $K=\bigcup^m_{i=1}S_i(K)$ (see Hutchinson\cite{Hutchinson_1981}). We call a family of contractions $\{S_i\}_{i=1}^{m}$ on $M$ an {\em iterated function system} (IFS), and call $K$ the {\em invariant set} or {\em attractor} of the IFS. If equality in (\ref{eq:j1.1}) holds,  then $S_i$ is called a {\em contractive similitude}. It is known that if a complete Riemannian manifold admits a proper similitude (i.e., a similitude with similarity ratio $r_i\neq1$), then it is locally isometric to Euclidean space (see, e.g., \cite[Lemma 2 of Theorem 3.6]{Kobayashi-Nomizu_1963}). Let $(p_1,\ldots,p_m)$ be a {\em probability vector}, i.e., $p_i>0$ and $\sum_{i=1}^mp_i=1$. Then there is a unique Borel probability measure $\mu$ with $\supp (\mu)=K$, called the {\em invariant measure}, that satisfies
	\begin{align*}
		\mu=\sum_{i=1}^m p_i\mu\circ S_i^{-1}.
	\end{align*}
	If $\{S_i\}_{i=1}^{m}$ is an IFS of contractive similitudes on $M$, then the attractor $K$ is called a {\em self-similar set} and $\mu$ a {\em self-similar measure}.
	
	Let $\{(U_{\alpha},\varphi_{\alpha})\}_{\alpha\in \mathcal{A}}$ be an atlas of $M$. We say that a map $S:M\rightarrow M$ is a {\em $C^{1+\gamma}$ diffeomorphism} at $x\in M$ if there exist $\alpha,\beta\in \mathcal{A}$ such that  $x\in U_{\alpha}$, $S(x)\in U_{\beta}$, and
	\begin{align*}
		f:=\varphi_{\beta}\circ S\circ \varphi_{\alpha}^{-1}: \R^n\rightarrow \R^n
	\end{align*}
	is $C^{1+\gamma}$, where $0<\gamma<1$.
	
	Let $W\subseteq M$ be open and connected, and let $F\subseteq W$ be a compact set with $\overline{F^{\circ}}=F$. Recall that a map $S: W \rightarrow W$ is {\em conformal} on $W$ if for each $x\in W$, $S'(x)$ is a similarity matrix, i.e., a scalar multiple of an orthogonal matrix.
	
	According to \cite{Ngai-Xu_2021} and \cite{Patzschke_1997}, given a positive integer $m\geq 2$, $S_1,\ldots,S_m:W \rightarrow W$,  we say that $\{S_i\}_{i=1}^m$ is a {\em conformal iterated function system} (CIFS) if conditions (a)--(c) below are satisfied:
	\begin{enumerate}
		\item[(a)] for all $i\in\{1,\ldots, m\}$, $S_i:W \rightarrow S_i(W)\subseteq W$ is a conformal $C^{1+\gamma}$ diffeomorphism, with $0<\gamma<1$ ;
		\item[(b)] $S_i(F^{\circ})\subseteq F^{\circ}$ for all $i\in\{1,\ldots, m\}$;
		\item[(c)] $0<|\det S'_{i}(x)|<1$ for all $i\in\{1,\ldots, m\}$ and all $x\in W$.
	\end{enumerate}
	If $S_1,\ldots,S_m$ are conformal mappings, the attractor $K$ is called the {\em self-conformal set} and the invariant measure $\mu$ is called the {\em self-conformal measure}.	If $S$ is a $C^{1+\gamma}$ diffeomorphisms, then $|\det S'(x)|$ satisfies Dini's condition.
Let
\begin{align*}
	\Sigma_{k}:=\{1, \ldots, m\}^{k}\quad \text { and } \quad \Sigma_{*}:=\bigcup_{k \geq 0} \Sigma_{k}
\end{align*}
(with $\left.\Sigma_{0}:=\{\emptyset\}\right)$. For $\tau=\left(i_{1}, \ldots, i_{k}\right) \in \Sigma_{k}$, let $\tau^-:=(i_1,\ldots,i_{k-1})$, $S_{\tau}:=S_{i_{1}} \circ \cdots \circ S_{i_{k}}$, and  $p_{\tau}:=p_{i_{1}}  \cdots  p_{i_{k}}.$
	As shown in \cite{Ngai-Xu_2021}, we can find an open and connected set $\Omega$ such that
	\begin{align}\label{eq:j1.7}
		\overline{\Omega} \,\,\text{is compact},\quad F\subseteq \Omega\subseteq\overline{\Omega}\subseteq W,
	\end{align}
and $\Omega$ satisfies the {\em bounded distortion property} (BDP), i.e., there exists a
constant $c \geq 1$ such that for any $x,y\in\Omega$ and all $\tau\in \Sigma_{*}$,
$$|{\rm det}S_\tau'(x)|\leq c\,|{\rm det}S_\tau'(y)|.$$
Define
\begin{align*}
	\|S'_{\tau}\|:=\sup_{x\in \Omega}|\det S'_{\tau}(x)|.
\end{align*}

	Recall that an IFS $\{S_i\}_{i=1}^{m}$ satisfies the {\em open set condition} (OSC) if there exists a nonempty bounded open set $U$ such that $\bigcup_{i=1}^mS_i(U)\subseteq U$ and $S_i(U)\cap S_j(U)=\emptyset$ for any $i\neq j$. Such a $U$ is called an {\em OSC-set}. If an OSC-set $U$ can be chosen so that $U\cap K\neq \emptyset$, we say that $\{S_i\}_{i=1}^{m}$ satisfies the {\em strong open set condition} (SOSC).
	
	Let $M$ be a complete Riemannian $n$-manifold. We state the following conditions:
	\begin{enumerate}
		\item[(C1)] (MPID) holds, and the embedding $\operatorname{dom}(\mathcal{E}_D) \hookrightarrow L^{2}(\Omega, \mu)$ is compact;
		
		\item[(C2)] $\overline{A}:=\max _{1 \leq i \leq m}\left\{p_{i} \|S'_{i}\|^{-(n-2)}\right\}<1$. If $\left\{S_{i}\right\}_{i=1}^{m}$ is an IFS of contractive similitudes on $M$, we use $r_i$ instead of $\|S'_{i}\|$;
		
		\item[(C3)] $\underline{\operatorname{dim}}_{\infty}(\mu)>n-2$;
		
		\item[(C4)] $\mu$ is upper $s$-regular for some $s>n-2$, where upper $s$-regularity will be defined in Section \ref{S:L}.
	\end{enumerate}

	\begin{thm}\label{thm:4.9}
	Let $n\geq 1$ and $M$ be a complete Riemannian $n$-manifold with  ${\rm Ric}\geq (n-1)\xi$ for some $\xi\in\mathbb{R}$. Let $\left\{S_{i}\right\}_{i=1}^{m}$ be a {\rm CIFS} on $M$ satisfying {\rm (OSC)}, and let $\Omega$ be defined as in \eqref{eq:j1.7}. Let $\mu$ be an associated self-conformal measure such that ${\rm supp}(\mu)\subseteq \overline{\Omega}$  and $\mu(\Omega)>0$.  Then
	\begin{enumerate}
		\item[(a)]  {\rm (C3)} and {\rm (C4)} are equivalent;
		\item[(b)]   {\rm (C2)} implies {\rm (C3)};
		\item[(c)]  {\rm(C3)} implies {\rm(C1)}.
	\end{enumerate}
\end{thm}

	In Theorem \ref{thm:4.9}, the implication $({\rm C3})\Rightarrow({\rm C1})$ and the equivalence  $({\rm C3}) \Leftrightarrow ({\rm C4})$  do not  require the Ricci curvature of $M$ to have a lower bound.

	\begin{thm}\label{thm:4.8}
		Let $n\geq 1$ and $M$ be a complete Riemannian $n$-manifold with  ${\rm Ric}\geq (n-1)\xi$ for some $\xi\in\mathbb{R}$. Let $\left\{S_{i}\right\}_{i=1}^{m}$ be an {\rm IFS} of contractive similitudes on $M$ satisfying {\rm(OSC)}, and let $\Omega\subseteq M$ be a bounded open set. Let $\mu$ be an associated self-similar measure such that ${\rm supp}(\mu)\subseteq \overline{\Omega}$  and $\mu(\Omega)>0$.  Then conditions {\rm(C1) }--{\rm (C4)} are equivalent.
	\end{thm}

	In Theorem \ref{thm:4.8}, the implications $({\rm C3})\Rightarrow({\rm C1})\Rightarrow({\rm C2})$, as well as the equivalence  $({\rm C3}) \Leftrightarrow ({\rm C4})$,  do not require the Ricci curvature of $M$ to have a lower bound.

\section{Hodge theorem for functions on complete Riemannian manifolds} \label{S:L}
	\setcounter{equation}{0}
	
	\subsection{Kre\u{\i}n-Feller operators for functions} \label{S:frac}
	Let $M$ be a complete Riemannian $n$-manifold. Let $\Omega\subseteq M$ be a bounded open set and let $\mu$ be a positive finite Borel measure on $M$ such that ${\rm supp}(\mu)\subseteq \overline{\Omega}$. In this section, we assume $\mu$ satisfies  (MPID) and (MPIE). Under this condition, we will define the Kre\u{\i}n-Feller operators $\Delta_{\mu}^D$ and $\Delta_{\mu}^E$, and study its basic properties.
	
	 Let 
	\begin{align}\label{eq(2.1)}
		\mathcal{E}_D^*(u, v):=&\int_{\Omega} \langle\nabla u, \nabla v\rangle  \,d\nu+\int_{\Omega} u v \,d \mu\quad\text{and}\nonumber\\
\mathcal{E}_E^*(u, v):=&\int_{M} \langle\nabla u, \nabla v\rangle  \,d\nu+\int_{M} u v \,d \mu
	\end{align}
 be  nonnegative bilinear forms on $L^2(\Omega,\mu)$ and $L^2(M,\mu)$, respectively.
	Then $\mathcal{E}_D^*(\cdot, \cdot)$ (resp. $\mathcal{E}_E^*(\cdot, \cdot)$) is a inner product on $\operatorname{dom}(\mathcal{E}_D)$ (resp. 	$\operatorname{dom}(\mathcal{E}_E)$).

	To prove Proposition \ref{prop:2.3}, we need the following lemmas. The following lemma is well-known.
	
	\begin{lem}\label{lem:2.2p} Let $M$ be a complete Riemannian $n$-manifold and let $\Omega\subseteq M$ be a bounded open set. Then there exists a positive constant $C$ such that for any $u\in C_c^{\infty}(\Omega)$,
			\begin{align*}
				\int_\Omega|u|^2d\nu\leq C\int_\Omega|\nabla u|^2d\nu.
			\end{align*}
	\end{lem}
\begin{proof}It can be proved by using coordinate charts, the fact that the results hold for $\Omega\subseteq \R^n$, and partition of unity. We omit the proof.
\end{proof}

	\begin{lem}\label{lem:2.2} Let $M$ be a complete Riemannian $n$-manifold. Then $C^{\infty}_c(M)$ is dense in $C_c(M)$ in the supremum norm.
	\end{lem}
	\begin{proof}It can be proved by using coordinate charts, the fact that the results hold for $M=\R^n$, and partition of unity. We omit the proof.
	\end{proof}
	\begin{prop}\label{prop:2.3}
		Let $M$ be a complete Riemannian $n$-manifold and let $\Omega\subseteq M$ be a bounded open set. Let $\mu$ be a positive finite Borel measure on $M$ such that ${\rm supp}(\mu) \subseteq \overline{\Omega}$ and $\mu(\Omega)>0$. Let $\mathcal{E}_D$, $\mathcal{E}_E$, $\mathcal{E}_D^*$, and $\mathcal{E}_E^*$ be the quadratic forms defined as in (\ref{eq(1.1)}), \eqref{eq:new}, and (\ref{eq(2.1)}).
		\begin{enumerate}
			\item[(a)] Assume  that $\partial\Omega\neq \emptyset$ and $\mu$ satisfies (MPID). Then $\operatorname{dom}(\mathcal{E}_D)$ is dense in $L^{2}(\Omega, \mu)$. Moreover,  $\left(\mathcal{E}_D^*, \operatorname{dom}(\mathcal{E}_D)\right)$ is a Hilbert space.
			\item[(b)]  Assume that  $\partial\Omega= \emptyset$ and $\mu$ satisfies (MPIE). Then $\operatorname{dom}(\mathcal{E}_E)$ is dense in $L^{2}(M, \mu)$. Moreover,  $\left(\mathcal{E}_E^*, \operatorname{dom}(\mathcal{E}_E)\right)$ is a Hilbert space.
		\end{enumerate}
	\end{prop}
	\begin{proof}(a) This follows by combining Lemmas \ref{lem:2.2p} and  \ref{lem:2.2} . The proof is similar to that of \cite[Proposition 2.1]{Hu-Lau-Ngai_2006}.
		
		(b)	By using Lemma \ref{lem:2.2} and the definition of $\operatorname{dom}(\mathcal{E}_E)$, we can show that $\operatorname{dom}(\mathcal{E}_E)$ is dense in $L^{2}(M, \mu)$. Let $\{u_m\}\subseteq \operatorname{dom}(\mathcal{E}_E)$ be a Cauchy sequence.  By the definition of $\operatorname{dom}(\mathcal{E}_E)$, we write $u_m=u_m^c+u_m^w$, where $u_m^c\in \mathcal{C}$ and  $u_m^w\in \operatorname{dom}(\mathcal{E}_W)$.
		By using the Poincar\'e inequality for $d\nu$ (\cite[Lemma 3.8]{Hebey_1996}) and (MPIE), we obtain that  $\{u_m^w\}$ is a Cauchy sequence in $\operatorname{dom}(\mathcal{E}_W)$ and thus converges to some $u^w$ in the norm induced by $\mathcal{E}_E^*$. Moreover, we show that $\{u_m^c\}$ is a Cauchy sequence in $\R$ and thus converges to some $u^c$ in the norm induced by $\mathcal{E}_E^*$. Therefore,  $\{u_m\}$ converges to  $u^w+u^c$ in the norm induced by $\mathcal{E}_E^*$. Thus, $\left(\mathcal{E}_E^*, \operatorname{dom}(\mathcal{E}_E)\right)$ is a Hilbert space.
	\end{proof}
	
	Proposition \ref{prop:2.3} implies that if $\mu$ satisfies (MPID), then the quadratic form $\left(\mathcal{E}^{*}_D, \operatorname{dom}(\mathcal{E}_D)\right)$ is closed on $L^{2}(\Omega, \mu)$.  Hence it follows from standard theory that there exists a nonnegative self-adjoint operator $H$ on $L^{2}(\Omega, \mu)$ such that  $\operatorname{dom}(H) \subseteq \operatorname{dom}\left(H^{1 / 2}\right)=\operatorname{dom}(\mathcal{E}_D)$ and
	$$
	\mathcal{E}_D(u,v)=\left\langle H^{1 / 2} u, H^{1 / 2} v\right\rangle_{L^{2}(\Omega, \mu)} \quad \text { for all}\,\,  u, v \in \operatorname{dom}(\mathcal{E}_D).
	$$
	Moreover, $u \in \operatorname{dom}(H)$ if and only if $u \in \operatorname{dom}(\mathcal{E}_D)$ and there exists $f \in L^{2}(\Omega, \mu)$ such that $\mathcal{E}_D(u, v)=\langle f, v\rangle_{L^{2}(\Omega, \mu)}$ for all $v \in \operatorname{dom}(\mathcal{E}_D)$ (see, e.g., \cite{Davies_1995,Kigami_2001,Reed-Simon_1972}). Note that for all $u \in \operatorname{dom}(H)$ and $v \in \operatorname{dom}(\mathcal{E}_D)$,
	\begin{align}\label{eq(2.2)}
		\int_{\Omega} \langle\nabla u, \nabla v\rangle \, d \nu=\mathcal{E}_D(u, v)=\langle H u, v\rangle_{L^{2}(\Omega, \mu)}.
	\end{align}
Similarly, $\mathcal{E}_E(u,v)$ satisfies the above properties. We let $\Delta_{\mu}^D:=-H$ and call it the  \textit{Dirichlet Laplacian (or Kre\u{\i}n-Feller operator) with respect to} $\mu$. Similarly, we can define $\Delta_\mu^E$. If no confusion is possible, we will denote  $\Delta_{\mu}^D$ and $\Delta_{\mu}^E$ simply by $\Delta_{\mu}$.

	Let $\mathcal{D}(\Omega)$ denote the space of {\em test functions} consisting of $C_{c}^{\infty}(\Omega)$ equipped with the following topology: a sequence $\left\{u_{n}\right\}$ converges to a function $u$ in $\mathcal{D}(\Omega)$ if there exists a compact $K \subseteq \Omega$ such that $\operatorname{supp}\left(u_{n}\right) \subseteq K$ for all $n$, and for any covariant derivative $\nabla^{s}$ of order $s$, the sequence $\left\{\nabla^{s} u_{n}\right\}$ converges to $\nabla^{s} u$ uniformly on $K$ (see, e.g., \cite[Section 2.1]{Hebey_1996}). Denote by $\mathcal{D}'(\Omega)$ the space of distributions, the dual space of $\mathcal{D}(\Omega)$.

	The proofs of Proposition \ref{prop:2.4} and Theorem \ref{thm:2.5} are similar to those of \cite[Proposition 2.2]{Hu-Lau-Ngai_2006} and \cite[Theorem 2.3]{Hu-Lau-Ngai_2006}, respectively, and are omitted.
	
\begin{prop}\label{prop:2.4}
		Let $n$, $M$, $\Omega$, and $\mu$ be defined as in Proposition \ref{prop:2.3}.
		\begin{enumerate}
			\item[(a)] 	Assume that  $\partial\Omega\neq \emptyset$ and $\mu$ satisfies {\rm(MPID)}. For $u \in \operatorname{dom}(\mathcal{E}_D)$ and $f \in L^{2}(\Omega, \mu)$, the following conditions are equivalent:
			\begin{enumerate}
				\item[(i)] $u \in \operatorname{dom}(-\Delta_\mu^D)$ and $-\Delta_\mu^D u=f$;
				\item[(ii)] $-\Delta u=f d \mu$ in the sense of distribution; that is, for any $v \in \mathcal{D}(\Omega)$,
				\begin{align*}
					\int_{\Omega} \langle\nabla u, \nabla v\rangle  \,d\nu=\int_{\Omega} v f\, d \mu.
				\end{align*}
			\end{enumerate}
			\item[(b)]	Assume that  $\partial\Omega= \emptyset$ and  $\mu$ satisfies {\rm(MPIE)}. For $u \in \operatorname{dom}(\mathcal{E}_E)$ and $f \in L^{2}(M, \mu)$, the following conditions are equivalent:
			\begin{enumerate}
				\item[(i)] $u \in \operatorname{dom}(-\Delta_\mu^E)$ and $-\Delta_\mu^E u=f$;
				\item[(ii)] $-\Delta u=f d \mu$ in the sense of distribution; that is, for any $v \in \mathcal{D}(M)$,
				\begin{align*}
					\int_{M} \langle\nabla u, \nabla v\rangle  \,d\nu=\int_{M} v f\, d \mu.
				\end{align*}
			\end{enumerate}
		\end{enumerate}
	\end{prop}
 For any $u \in \operatorname{dom}\left(\Delta_{\mu}^D\right)$, we have by Proposition \ref{prop:2.3} that $\Delta u=\Delta_{\mu}^D u\, d \mu$ in the sense of distribution. We rewrite (\ref{eq(2.2)}) as
	$$
	\int_{\Omega} \langle\nabla u, \nabla v\rangle \, d \nu=\mathcal{E}_D(u, v)=\left\langle-\Delta_{\mu}^D u, v\right\rangle_{L^{2}(\Omega, \mu)}
	$$
	for $u \in \operatorname{dom}\left(\Delta_{\mu}^D\right)$ and $v \in \operatorname{dom}(\mathcal{E}_D)$.
	
	Similarly, 	we have $$
		\int_{M} \langle\nabla u, \nabla v\rangle \, d \nu=\left\langle-\Delta_{\mu}^E u, v\right\rangle_{L^{2}(M, \mu)}
	$$
	for $u \in \operatorname{dom}\left(\Delta_{\mu}^E\right)$ and $v \in \operatorname{dom}(\mathcal{E}_E)$.
	
	The following theorem shows that for any $f \in L^{2}(\Omega, \mu)$, the equation
$$\Delta_{\mu}^D u=f,\qquad u|_{\partial\Omega}=0$$
	has a unique solution in $L^{2}(\Omega, \mu)$;
	the equation
	$$\Delta_{\mu}^E u=f$$
	has a unique solution in $L^{2}(M, \mu)$.
	
	\begin{thm}\label{thm:2.5}
		Let $n$, $M$, $\Omega$, and $\mu$ be defined as in Proposition \ref{prop:2.3}.
		\begin{enumerate}
			\item[(a)] Assume that  $\partial\Omega\neq \emptyset$ and $\mu$ satisfies {\rm(MPID)}.
			Then for any $f \in L^{2}(\Omega, \mu)$, there exists a
			unique $u \in \operatorname{dom}\left(\Delta_{\mu}^D\right)$  such that $\Delta_{\mu}^D u=f.$ The operator
			$$(\Delta_{\mu}^D)^{-1}: L^{2}(\Omega, \mu) \rightarrow \operatorname{dom}(\Delta_{\mu}^D),\quad f\mapsto u $$
			is bounded and has norm at most $C$, the constant in \eqref{eq:PI} .
			\item[(b)] Assume that  $\partial\Omega= \emptyset$ and $\mu$ satisfies {\rm(MPIE)}.
			Then for any $f \in L^{2}(M, \mu)$, there exists a
			unique $u \in \operatorname{dom}\left(\Delta_{\mu}^E\right)$  such that $\Delta_{\mu}^E u=f.$ The operator
			$$(\Delta_{\mu}^E)^{-1}: L^{2}(M, \mu) \rightarrow \operatorname{dom}(\Delta_{\mu}^E),\quad f\mapsto u $$
			is bounded and has norm at most $C$, the constant in  \eqref{eq:PIN}.
		\end{enumerate}
\end{thm}

\subsection{$L^\infty$-dimension and proofs of Theorems \ref{thm:1.1} and \ref{thm:1.2}}
	
	Let $\mathfrak{X}$ be a complete metric space, and $\mu$ be a positive finite regular Borel measure on $\mathfrak{X}$ with compact support.  We say that $\mu$ is \textit{ upper $s$-regular} for $s>0$, if there exists some $c>0$ such that, for all $x \in {\rm supp}(\mu)$ and all $0 \leq r \leq \operatorname{diam}({\rm supp}(\mu))$,
	\begin{align*}
		\mu\left(B^{\mathfrak{X}}(x,r)\right) \leq c\, r^{s}.
	\end{align*}
	 \textit{Lower $s$-regularity} is defined by reversing the inequality.
	
	The proof of Theorem \ref{thm:1.1} is quite long; we give an outline here, focusing on the case $n\geq 2$. In Step 1, we prove a relation between the upper (resp. lower) regularity of $\mu$ and $\underline{\operatorname{dim}}_{\infty}(\mu)$ (resp. $\overline{\operatorname{dim}}_{\infty}(\mu)$) (see Lemma \ref{lem:3.1}). This relation will be used in Step 4. In Step 2, we prove that under a normal coordinate map $\varphi$, a point $w\in {\rm supp}(\mu)$ and its image $\varphi(w)$ have the same local density property specified by Lemma \ref{lem:3.10}. To this end, we first prove that each point in the image of $\varphi$ has a neighborhood $\widetilde{V_{\delta}}(x)$ which contains a ball $B(x,\delta)$ and is contained
	in a ball $B(x,c\delta)$, where $c$ is some constant; moreover, the inverse image $V_\delta(\varphi^{-1}(x))$ has a similar property (see Proposition \ref{prop:3.8}).  Then we prove that $\mu$ induces a measure $\widetilde{\mu}$ with compact support on $\R^n$ (see Remark \ref{rema:3.9}). In Step 3, we show that for $q>2$, the unit ball $B:=\{u \in C_{c}^{\infty}\left(M\right):\|u\|_{W^{1,2}_0\left(M\right)} \leq 1\}$ is relatively compact in $L^{q}\left(M, \mu\right)$ if and only if $\mu$ satisfies certain local density property (see Theorem \ref{thm:3.11}). This is a generalization of an analogous result on $\R^n$ (see Theorem \ref{thm:maz}). Lemma \ref{lem:3.10} is needed in Step 3. In Step 4, we use the result in Step 3 to prove that if $\underline{\operatorname{dim}}_{\infty}(\mu)>q(n-2) / 2$, then the unit ball $B$ is relatively compact in $L^{q}(M, \mu)$ (see Theorem \ref{thm:3.12}).  Finally, Theorem \ref{thm:1.1} follows by combining Step 4 and a result for the $1$-dimensional case (see Theorem \ref{lem:3.13}).
	
	We state a relation between the upper (resp. lower) regularity and lower (resp. upper) $L^\infty$-dimension of $\mu$.
	\begin{lem}\label{lem:3.1}
		Let $\mu$ be a compactly supported finite positive Borel measure on a complete metric space $\mathfrak{X}$.
		\begin{enumerate}
			\item[(a)]  If $\mu$ is upper (resp. lower) $s$-regular for some $s>0$, then $\underline{\operatorname{dim}}_{\infty}(\mu) \geq s$ (resp. $\overline{\operatorname{dim}}_{\infty}(\mu) \leq s$).
			\item[(b)]  Conversely, if $\underline{\operatorname{dim}}_{\infty}(\mu) \geq s$ (resp. $\overline{\operatorname{dim}}_{\infty}(\mu) \leq s$) for some $s>0$, then $\mu$ is upper (resp. lower)  $\alpha$-regular for any $0<\alpha<s$ (resp. $\alpha>s$).
		\end{enumerate}
	\end{lem}
	The proof of Lemma \ref{lem:3.1} is similar to that of\cite[Lemma 3.1]{Hu-Lau-Ngai_2006}; we omit the details.
	

	The following technical lemma is used in the proof of Proposition \ref{prop:3.8}; we postpone its proof to the appendix.

	\begin{lem}\label{lem:3.4}
	Let $\rho$ and $R_i$ be  positive constants, where $i=1,2$.  
	\begin{enumerate}
		\item[(a)]  If $\delta\in[ 0,\rho\pi/4]$ and $\lambda \in [0,2\delta]$, then
		\begin{align*}
			C_1\delta\leq{\rm arccosh}\Big[\cosh^2\Big(\frac{\rho+\lambda}{R_1}\Big)-\sinh^2\Big(\frac{\rho+\lambda}{R_1}\Big)\cos\Big(\frac{4\delta}{\rho}\Big)\Big]	\leq C_2\delta,
		\end{align*}
		where $C_1$ and $C_2$ are positive constants depending only on $\rho$ and $R_1$.
		\item[(b)] 
		Let $\pi R_2-\rho>0$. If $\delta\in[0,\rho\pi/4]$  and  $\lambda \in [0,2\delta]\cap [0,\pi R_2-\rho]$, then
		\begin{align*}
			C_3\delta\leq{\rm arccos}\Big[\cos^2\Big(\frac{\rho+\lambda}{R_2}\Big)+\sin^2\Big(\frac{\rho+\lambda}{R_2}\Big)\cos\Big(\frac{4\delta}{\rho}\Big)\Big]	\leq C_4\delta,
		\end{align*}
		where $C_3$ and $C_4$ are positive constants depending only on $\rho$ and $R_2$.
		\item[(c)]
		For $\delta\in [0,\pi\rho/4]$ and $\lambda \in [0,2\delta]$,
		\begin{align*}
			2\delta\leq2(\rho+\lambda)\sin\Big(\frac{2\delta}{\rho}\Big)	\leq 8\delta.
		\end{align*}
	\end{enumerate}
	\end{lem}

	We refer the reader to \cite{Cheeger-Ebin_2008,Karcher_1987} for Definitions \ref{def:3.5} and \ref{def:j3.7} and Theorems \ref{thm:3.6} and \ref{thm:3.7} (see, e.g.,\cite[Theorem 2.2]{Cheeger-Ebin_2008} or \cite[Theorems 4.1 and 4.2]{Karcher_1987}), which are stated below.

	\begin{defi}\label{def:3.5}
		A set of three geodesic segments $\left(\gamma_{1}, \gamma_{2}, \gamma_{3}\right)$ parameterized by arc length with lengths $l_{1}, l_{2}, l_{3}$ respectively is said to be a {\em (geodesic) triangle} in a Riemannian manifold $M$ if $\gamma_{i}\left(l_{i}\right)=\gamma_{i+1}(0)$ and $l_{i}+l_{i+1} \geq l_{i+2}$. Let
		$$
		\alpha_{i}=\measuredangle\left(-\gamma_{i+1}'\left(l_{i+1}\right), \gamma_{i+2}'(0)\right)\in [0,\pi]
		$$
		be the angle between $-\gamma_{i+1}'\left(l_{i+1}\right)$ and $\gamma_{i+2}'(0)$.
	\end{defi}
	Let $M^{\underline{\kappa}}$ (resp. $M^{\overline{\kappa}}$) be a simply connected $2$-dimensional space of constant sectional curvature $\underline{\kappa}$ (resp. $\overline{\kappa}$).
	
	
	\begin{defi}\label{def:j3.7}
		Let $T:=(\gamma_{1}, \gamma_{2}, \gamma_{3})$ be a (geodesic) triangle in a Riemannian manifold $M$. Let $\gamma_{1}, \gamma_{2}$ be geodesic segments in $M$ such that $\gamma_{1}\left(l_{1}\right)=\gamma_{2}(0)$ and $\measuredangle\left(-\gamma_{1}'\left(l_{1}\right), \gamma_{2}'(0)\right)=\alpha$. We call such a configuration a {\em hinge} and denote it by $\left(\gamma_{1}, \gamma_{2}; \alpha\right)$. A {\em Rauch hinge} in $M^{\underline{\kappa}}$ is a hinge that has the same edge length and angle as that at one vertex of $T$. A triangle in $M^{\underline{\kappa}}$ with the same edgelengths as $T$, denoted $T_{\underline{\kappa}}:=(\overline{\gamma}_{1}, \overline{\gamma}_{2}, \overline{\gamma}_{3})$, is called an {\em Aleksandrow triangle}. Let $\overline{\alpha}_i$, $i=1,2,3$, denote the angles in an Aleksandrow triangle $T_{\underline{\kappa}}$.
	\end{defi}
	
	Theorems \ref{thm:3.6} and \ref{thm:3.7} below use the notation in Definitions \ref{def:3.5} and \ref{def:j3.7}.
	\begin{thm}[Toponogov theorem]\label{thm:3.6}
		Let $M$ be a complete Riemannian $n$-manifold with $\underline{\kappa} \leq K_M$. Let $T$ be a (geodesic) triangle in $M$. Assume the circumference of $T$ satisfies $\sum_{i=1}^3l_i<2 \pi/\sqrt{\underline{\kappa}}$. Then an Aleksandrow triangle $T_{\underline{\kappa}}$ exists and the angles of $T$ and $T_{\underline{\kappa}}$ satisfy
		$$
		\overline{\alpha}_1\leq \alpha_1, \quad \overline{\alpha}_2\leq \alpha_2, \quad \overline{\alpha}_3\leq \alpha_3.
		$$
		The third edge closing a Rauch hinge $(\overline{\gamma}_{1}, \overline{\gamma}_{2}; \overline{\alpha})$ in  $M^{\underline{\kappa}}$ satisfies
		$$
		d_M\left(\gamma_{1}(0), \gamma_{2}\left(l_{2}\right)\right) \leq d_{M^{\underline{\kappa}}}\left(\overline{\gamma}_{1}(0), \overline{\gamma}_{2}\left(l_{2}\right)\right).
		$$
	\end{thm}
The following is a consequence of Rauch's comparison theorem.
	\begin{thm}[Rauch comparison theorem]\label{thm:3.7}
		Let $M$ be a complete Riemannian $n$-manifold with $K_{M}\leq \overline{\kappa}$. Let $T$ be a (geodesic) triangle.  Assume $T$ does not meet the cut locus of its vertices, and the circumference of $T$ satisfies $\sum_{i=1}^3l_i<2 \pi /\sqrt{\overline{\kappa}}$ (ignore this if $\overline{\kappa} \leq 0)$. Then an Aleksandrow triangle $T_{\overline{\kappa}}$ exists and the angles of $T$ and $T_{\overline{\kappa}}$ satisfy
		$$
		\alpha_1 \leq \overline{\alpha}_1, \quad \alpha_2 \leq \overline{\alpha}_2, \quad \alpha_3 \leq \overline{\alpha}_3.
		$$
		The third edge closing a Rauch hinge $(\overline{\gamma}_{1}, \overline{\gamma}_{2}; \overline{\alpha})$ in $M^{\overline{\kappa}}$ satisfies
			$$
		d_{M^{\overline{\kappa}}}\left(\overline{\gamma}_{1}(0), \overline{\gamma}_{2}\left(l_{2}\right)\right) \leq d_M\left(\gamma_{1}(0), \gamma_{2}\left(l_{2}\right)\right).
		$$
	\end{thm}
For any $p\in M$, let
\begin{align*}
	&\underline{K}_M(p):=\inf\{K(X(p),Y(p)): X(p),Y(p)\in T_pM \text{ are linearly independent}\},\,\,\text{and}\\
	&\overline{K}_M(p):=\sup\{K(X(p),Y(p)): X(p),Y(p)\in T_pM \text{ are linearly independent}\}.	
\end{align*}		
For $A\subseteq M$, let  
 	\begin{align}\label{cur}
 	\underline{\kappa}(A):=\min\{\underline{K}_M(p):p\in A\}\quad \text{and}\quad \overline{\kappa}(A):=\max\{\overline{K}_M(p):p\in A\}.
 \end{align}
 We have the following proposition.

\begin{prop}\label{prop:3.0}
	Let $M$ be a Riemannian $n$-manifold. 
Let $\Omega\subseteq M$ be a bounded open set.
Then for any $p\in \overline{\Omega}$, 
\begin{align}\label{sec}	
	-\infty<\underline{\kappa}(\overline{\Omega})\leq K_M(p)\leq \overline{\kappa}(\overline{\Omega})<\infty.
\end{align}
\end{prop}
\noindent The proof of Proposition \ref{prop:3.0} is given in the appendix.

	Let $M$ be a complete Riemannian $n$-manifold. For each $p\in M$, the {\em injectivity radius of $M$ at $p$, denoted by ${\rm inj}(p)$ is} the supremum of all $r>0$ such that $\exp_p:B^{T_pM}(0, r)\rightarrow B^M(p, r)$ is a diffeomorphism. Let
		\begin{align*}
			{\rm inj}(M):=\inf\{ {\rm inj}(p),\,\, p\in M\}.
\end{align*}
Then ${\rm inj}(p)$ is positive for every $p\in M$ (see, e.g., \cite[Corollary I.6.3]{Chavel_2006}). Let $\Omega\subseteq M$ be a bounded open set. Let $\mu$ be a positive finite Borel measure on $M$ such that $\supp(\mu) \subseteq \overline{\Omega}$ and $\mu(\Omega)>0$. The continuity of ${\rm inj}(p)$ implies that ${\rm inj(\supp(\mu))}:=\inf\{ {\rm inj}(p),\,\, p\in \supp(\mu) \}>0$.

Let $\epsilon\in(0,{\rm inj}(p))$. Then the exponential map $\exp_p:B^{T_pM}(0,\epsilon)\to B^M(p,\epsilon)$ is a diffeomorphism.  Every orthonormal basis $(b_i)$ for $T_pM$ determines a basis isomorphism  $E_p:T_pM  \rightarrow \R^n$.  Moreover, $E_p$ is an isometry. Let $B(0,\epsilon):=E_p(B^{T_pM}(0,\epsilon))$ and $S: B(0,\epsilon) \rightarrow B(z,\epsilon)$ be a similitude. We define a normal coordinate map
\begin{align}\label{eq:j3.8}
	\varphi:=S \circ E \circ \exp _{p}^{-1}: B^M(p,\epsilon) \rightarrow B(z,\epsilon)\subseteq\R^n,
\end{align}
where $\varphi(p)=z$.

 Since $\supp(\mu)$ is compact, there exists  a finite open cover $\{B^M(p_i,\epsilon_i)\}_{i=1}^N$ of $\supp(\mu)$, where  each $B^M(p_i,\epsilon_i)$ is a geodesic ball and $p_i\in \supp(\mu)$. Let $W:=\overline{\cup_{i=1}^N B^M(p_i,\epsilon_i)}$. Proposition \ref{prop:3.0} implies that for any $p\in W$, 
\begin{align}\label{eq:22.0}
-\infty<\underline{\kappa}(W)\leq K_M(p)\leq \overline{\kappa}(W)<\infty.
\end{align}
Let $\epsilon_i\in(0,{\rm inj(\supp(\mu))})$. Then the mapping $\exp_{p_i}:B^{T_{p_i}M}(0,\epsilon_i)\rightarrow B^M(p_i,\epsilon_i)$ is a diffeomorphism.  For each $i=1,\ldots,N$, let $E_{p_i}: T_{p_i}M\rightarrow \R^n$ be defined as above, and let $B(0,\epsilon_i):=E_{p_i}(B^{T_{p_i}M}(0,\epsilon_i))$. Also, let $S_i: B(0,\epsilon_i) \rightarrow B(z_i,\epsilon_i)$ be similitudes such that  $B(z_i,\epsilon_i)$ are disjoint. Therefore we have a family of normal coordinate maps
\begin{align}\label{eq:jj3.8}	
	\varphi_i:=S_i \circ E_{p_i }\circ \exp _{p_i}^{-1}: B^M(p_i,\epsilon_i) \rightarrow B(z_i,\epsilon_i), \,\,i=1,\ldots,N,
\end{align}
where  $\varphi_i(p_i)=z_i$ and the set $\{\varphi_i(B^M(p_i,\epsilon_i))\}_{i=1}^N$ are disjoint.

Given a circle in $\R^n$ and two points $x,y$ on the circle, we let $\overset{\frown}{xy}$ be the {\em shorter arc} connecting $x,y$, and $|\overset{\frown}{xy}|$ be the {\em arc length}. Let $\overline{xy}$ be the {\em line segment} from $x$ to $y$. For each $p,q\in M$, let $\gamma_{pq}$ be an {\em admissible curve} (i.e., piecewise smooth and has non-zero velocity) in $M$  from $p$ to $q$. Denote the {\em length of an admissible curve} $\gamma$ by $L[\gamma]$.  Let $\angle(a,b)$ be {\em the angle} between two vectors $a$ and $b$ in $\R^n$.  Let $\Lambda_{\omega}(z,l)$ be an {\em open infinite cone}, where $z$ is the vertex,  $l$ is the axis of symmetry, and $\omega<\pi^{n/2}/\Gamma(n/2)$ is the solid angle. Here $\Gamma$ is the gamma function. Note that $2\pi^{n/2}/\Gamma(n/2)$ is the surface area of a unit sphere in $\R^n$, and therefore the condition imposed on $\omega$ ensures that it is less than a straight angle. We define the \textit{polar angle} $\theta$ of a cone as the angle between the axis of symmetry and any generatrix.  Let $\beta=2\theta$. For $u\in [0,1]$, let  $I_{u}(m,j):=D_u(m,j)/D(m,j)$, where for $ m>0$ and $j>0$, 
$$D_u(m,j):=\int_0^uv^{m-1}(1-v)^{j-1}\,dv\quad \text{and} \quad D(m,j):=\int_0^1v^{m-1}(1-v)^{j-1}\,dv.$$
It can be shown that $\omega=\pi^{n/2}I_{\sin^2\theta}((n-1)/2,1/2)/\Gamma(n/2)$ (see \cite{Li_2011}). Let $ \rho$ and $\delta$ be positive numbers. Define $$\widetilde{V}_{\delta}:=\widetilde{V}_{\delta,\rho,\beta,z,l}=(\Lambda_{\omega}(z,l)\bigcap B(z,\rho+2\delta))\backslash\overline{(\Lambda_{\omega}(z,l)\bigcap B(z,\rho))}$$  (see Figure \ref{fig:j1}). Let $x\in\widetilde{V}_{\delta}$ be the {\em center} of $\widetilde{V}_{\delta}$ defined as the mid-point of the intersections of $l$ with $\partial B(z,\rho)$ and $\partial B(z,\rho+2\delta)$. We also write $\widetilde{V}_{\delta}$ as $\widetilde{V}_{\delta}(x)$.

Our next goal is to prove that under a normal coordinate map $\varphi$, a point $w\in {\rm supp}(\mu)$ and its image $\varphi(w)$ have the same local density property stated in Lemma \ref{lem:3.10}. To this end, we first establish the following proposition.

\begin{prop}\label{prop:3.8}
	Let $M$ be a complete Riemannian $n$-manifold. Let $\Omega\subseteq M$ be a bounded open set. Let $\mu$ be a positive finite Borel measure on $M$ such that $\operatorname{supp}(\mu) \subseteq \overline{\Omega}$ and $\mu(\Omega)>0$.  Let $\{(U_{i},\varphi_{i})\}_{i=1}^{N}$ be a family of normal coordinate charts such that $\{U_i\}_{i=1}^N$ covers $\supp(\mu)$, where $U_i:=B^M(p_i,\epsilon_i)$, $p_i\in \supp(\mu)$, $\epsilon_i$, and $\varphi_{i}$ are defined as in the paragraph leading to \eqref{eq:jj3.8}. Let $W:=\overline{\cup_{i=1}^N B^M(p_i,\epsilon_i)}$.
	\begin{enumerate}
		\item[(a)] For $i=1,\ldots,N$ and $x\in \varphi_i(U_i)\backslash \{\varphi_i(p_i)\}$,  there exist positive numbers $\beta_i$, $\rho_i$ and $\delta_i$  satisfying $\beta_i=4\delta_i/\rho_i$, $\rho_i+2\delta_i\le \epsilon_i$, and $0<\delta_i< \pi\rho_i/4$ such that $\widetilde{V}_{\delta_i}(x):=\widetilde{V}_{\delta_i,\rho_i,\beta_i,z_i,l_i}(x)$ satisfies
		\begin{align}\label{eq:3.3.1}
			B(x,\delta_i)\subseteq \widetilde{V}_{\delta_i}(x)\subseteq B(x,\delta_i'),
		\end{align}
		where  $\delta_i'$ can be taken to be $10\delta_i$.\\
		\item[(b)] For $i=1,\ldots,N$ and $x\in \varphi_i(U_i)\backslash \{\varphi_i(p_i)\}$, define $V_{\delta_i}(\varphi^{-1}_i(x)):=\varphi^{-1}_i(\widetilde{V}_{\delta_i}(x))$. Then  for  $0<\rho_i<\min\{\pi/(4\sqrt{\overline{\kappa}(W)}),\epsilon_i\}$, there exist positive constants $c'$ and $c$ such that for $\delta_i$ satisfying $0<\delta_i<\min\{\pi/(4\sqrt{\overline{\kappa}(W)})-\rho_i/2,\pi\rho_i/4\}$, we have
		\begin{align}\label{eq:3.3.2}
			B^M(\varphi^{-1}_i(x),c'\delta_i)\subseteq V_{\delta_i}(\varphi^{-1}_i(x))\subseteq B^M(\varphi^{-1}_i(x),c\delta_i).
		\end{align}
	\end{enumerate}
\end{prop}
	\begin{proof}\,(a) 
		Let $(U_i,\varphi_i)$ be a coordinate chart with  $p_i\in \supp(\mu)$.	For simplicity, we denote $(U_i,\varphi_i)$ by $(U,\varphi)$, $p_i$ by $p$, and $\epsilon_i$ by $\epsilon$. Let $B(z,\epsilon):=\varphi(B^M(p,\epsilon))\subseteq \R^n$.
		For each $ x\in B(z,\epsilon)\backslash \{z\}$, there exist positive constants $\rho$ and $\delta$ such that $B(x,\delta)\subseteq  B(z,\rho+2\delta)\setminus B(z,\rho)$ and $\rho+2\delta\leq \epsilon$.  Let $\overline{zx}$ be the axis of symmetry of $\Lambda_{\omega}(z,l)$. Let $\beta$ be the maximum of the angle between any two rays on the surface of $\Lambda_{\omega}(z,l)$ emanating from $z$. Then $\beta$ is twice the polar angle.
		Let $a$ (resp. $e$) and $b$ (resp. $c$) be the intersections of  two such rays with $\partial B(z,\rho+2\delta)$ (resp. $\partial B(z,\rho)$). Then $\beta=\angle azb$ and $a$, $z$, $b$ define a plane which contains $e,c$, and $x$. Hence $\overset{\frown}{ec}$ and $\overset{\frown}{ab}$ can be determined by intersecting this plane with $\partial B(z,\rho)$ and $\partial B(z,\rho+2\delta)$ (see Figure \ref{fig:j1}).
		
		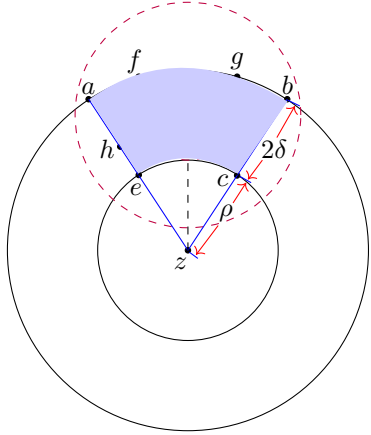
\begin{figure}[h!]
			\centerline{
				\begin{tikzpicture}[scale=0.6]
					\draw[fill=black](-2,0)circle(.06); 
					\draw[fill=black](-3.09,3.85)circle(.06);
					\draw[fill=black](-3.09,1.67)circle(.06);
					\draw[fill=black](-4.2,3.35)circle(.06);
					\draw[fill=black](0.2,3.35)circle(.06);
					\draw[fill=black](-3.5,2.3)circle(.06);
					\draw[fill=black](-2,3)circle(.06);
					\draw[fill=black](-0.91,1.66)circle(.06);
					\draw[fill=black](-0.91,3.85)circle(.06);
					\draw[black](0,0)arc(0:360:2);    
					\draw[black](2,0)arc(0:360:4);
					\draw[purple,dashed](-1,3)arc(0:360:1);
					\draw [purple,dashed] (-2,3) circle [radius=2.5];
					\draw[blue](-2,0)--(0.2,3.35);
					\draw[blue](-2,0)--(-4.2,3.35);
					\draw[black,dashed](-2,0)--(-2,4);  
					\draw[black,dotted](-2,3)--(-3.5,2.3);
					\draw[black,dashed](-3.09,1.65)--(-3.09,3.85);
					\draw[black,dashed](-0.91,1.66)--(-0.91,3.85);
					\draw[blue](-1.97,-0.04)--(-1.78,-0.18);
					\draw[blue](-0.85,1.64)--(-0.6,1.5);
					\draw[blue](0.21,3.33)--(0.48,3.2);
					\draw[red,->](-1,1.05)--(-0.71,1.5);
					\draw[red,<-](-1.82,-0.1)--(-1.33,0.55);
					\draw[red,<-](-0.67,1.57)--(-0.35,2);
					\draw[red,->](-0.03,2.5)--(0.37,3.2);
					\draw[black] (-2.55,-0.28) node[right]{\footnotesize{$z$}};
					\draw[black] (-3.53,1.35) node[right]{\footnotesize{$e$}};
					\draw[black] (-1.6,1.55) node[right]{\footnotesize{$c$}};
					\draw[black] (-2.13,2.9) node[right]{\footnotesize{$x$}};
					\draw[black] (-4.6,3.6) node[right]{\footnotesize{$a$}};
					\draw[black] (-3.6,4.2) node[right]{\footnotesize{$f$}};
					\draw[black] (-1.3,4.2) node[right]{\footnotesize{$g$}};
					\draw[black] (-0.16,3.67) node[right]{\footnotesize{$b$}};
					\draw[black] (-4.2,2.25) node[right]{\footnotesize{$h$}};
					\draw[black] (-1.55,0.82) node[right]{\footnotesize{$\rho$}};
					\draw[black] (-0.62,2.25) node[right]{\footnotesize{$2\delta$}};
					\path[fill=blue!20,opacity=.3] (0.205,3.3) to [out=148, in=36](-4.18,3.35) to (-3.1,1.7) to [out=34, in=150](-0.9,1.7) to (0.205,3.3);
			\end{tikzpicture}}
			\caption{The concentric balls $B(z,\rho)$ and $B(z,\rho+2\delta)$  inside the concentric ball $B(z,\epsilon)=\varphi(B^M(p,\epsilon))$ (not shown). The shaded part is the set $\widetilde{V}_{\delta}(x)$.}\label{fig:j1}
		\end{figure}
		
		Now we vary $\beta$ until $|\overset{\frown}{ec}|=4\delta$. Note that $d_{E}(z, c)=d_{E}(z, e)=\rho$, and $d_{E}(b, c)=d_{E}(a, e)=2\delta$. 
		Hence
		$$d_E(e,c)=2\rho\sin\frac{2\delta}{\rho}\quad \text{and} \quad d_{E}(a, b)=2\rho\sin\frac{2\delta}{\rho}+4\delta\sin\frac{2\delta}{\rho}.$$ For $\delta\in (0,\pi\rho/4)$, we have $2\rho \sin (2\delta/\rho)>2\delta$, and thus $d_{E}(e, c)>2\delta$.
		Let $f,g \in \overset{\frown}{ab}$ such that  $\overline{ef}$, $\overline{zx}$, and $\overline{cg}$ are parallel. Since
		$d_E(e, c)>2\delta$, it follows that for any point $s \in \overline{fe}$, we have
		$d_E(s, x)>\delta.$
		Thus for any point $ m\in \overline{ae}$, $d_E(m, x)>\delta$. Consequently, there exist $\rho$ and $\delta$   satisfying $\rho+2\delta\le \epsilon$ and $0<\delta< \pi\rho/4$ such that 
		\begin{align}\label{eq:9}	
			B(x,\delta)\subseteq \widetilde{V_\delta}(x).
		\end{align}
		
		Next, we choose a point $h\in \overline{ae}$ such that $\overline{ah} \perp \overline{hx}$. Then
		$$d_E(h, x)<d_E(a, b)=2 \rho \sin \frac{2\delta}{\rho}+4 \delta \sin \frac{2\delta}{\rho}.$$
		For any point $t \in \overline{ae}$ and $0<\delta<\pi\rho/4$,
		\begin{align*}
			d_E(t, x) &\leq d_E(t, h)+d_E(h, x)<d_E(a, e)+d_E(h, x)<10\delta.
		\end{align*}
		Hence there exist $\rho$ and $\delta$  satisfying $\rho+2\delta\le \epsilon$ and $0<\delta< \pi\rho/4$ such that 
		\begin{align}\label{eq:99}	
			\widetilde{V}_{\delta}(x)\subseteq B(x,\delta'),
		\end{align}
		where $\delta'$ can be taken to be $10\delta$.  Combining \eqref{eq:9} and \eqref{eq:99}, we see that for $x\in \varphi(U)\backslash\{\varphi(p)\}$, there exist $\beta$, $\rho$, and $\delta$   satisfying $\beta=4\delta/\rho$, $\rho+2\delta\le \epsilon$, and $0<\delta< \pi\rho/4$ such that
			$$B(x,\delta)\subseteq\widetilde{V}_{\delta}(x)\subseteq B(x,\delta').$$
	Hence, for $i=1,\ldots,N$ and $x\in \varphi_i(U_i)\backslash \{\varphi_i(p_i)\}$,  there exist positive numbers $\beta_i$, $\rho_i$ and $\delta_i$  satisfying $\beta_i=4\delta_i/\rho_i$, $\rho_i+2\delta_i\le \epsilon_i$, and $0<\delta_i< \pi\rho_i/4$ such that \eqref{eq:3.3.1} holds.
		
		(b) Let $S$, $E_p$, and $\exp_p$ be maps defined as in the paragraph containing \eqref{eq:j3.8}. Since $S$ is a similitude  and  $E_p$ is an isometry, it follows that the composition $E_p^{-1}\circ S^{-1}$ is an isometry. Let $a$, $b$, and $z$ be defined as in (a). Let $v_1:=a-z$ and $u_1:=b-z$. Then $v:=E_p^{-1}\circ S^{-1}(v_1)\in T_pM$ and $u:=E_p^{-1}\circ S^{-1}(u_1)\in T_pM$, and therefore $|v_1|=|v|$, $|u_1|=|u|$, and $\angle(v_1, u_1)=\angle(v, u)$. Hence $\exp_p(v)$ and $\exp_p(u)$ are geodesics in $B^M(p,\epsilon)$, and $\angle(v, u)=\measuredangle(\exp_p(v), \exp_p(u))$. Thus, there exist $a',b'\in \partial B^M(p,\rho+2\delta)$ such that
			$$|v|=L[\gamma_{pa'}]=\rho+2\delta,\,\,\,\,|u|=L\left[\gamma_{pb'}\right]=\rho+2\delta,\,\,\,\,\text{and}\,\, \,\,\angle(v, u)=\measuredangle(\gamma_{pa'},\gamma_{pb'})=\frac{4\delta}{\rho},$$
			where $\gamma_{pa'}$ and $\gamma_{pb'}$ are minimal geodesics in $ M$ and  $0<\delta<\pi\rho/4$ (see Figure \ref{fig.2}(a)).
			\begin{figure}[htbp]
				\centering
				\begin{tikzpicture}[scale=0.6]
					\draw[fill=black](-2,0)circle(.06);
					\draw[fill=black](-3.47,2.4)circle(.06);
					\draw[fill=black](-0.52,2.4)circle(.06);
					\draw[fill=black](-2,3)circle(.06); 
					\draw[fill=black](-3.7,3.65)circle(.06); 
					\draw[fill=black](-0.3,3.65)circle(.06); 
					\draw[fill=black](-3.1,1.66)circle(.06); 
					\draw[fill=black](-0.9,1.67)circle(.06); 
					\draw[black](0,0)arc(0:360:2);    
					\draw[black](2,0)arc(0:360:4);    
					\draw [purple,dashed, thin] (-2,3) circle [radius=0.98];
					\draw [purple,dashed,semithick] (-2,3) circle [radius=2.4];
					\draw[blue] (-2,0) ..controls (-2,0) and (-3.8,2.1) .. (-3.7,3.65);
					\draw [blue](-2,0) ..controls (-2,0) and (-0.2,2.1) .. (-0.3,3.65);
					\draw [blue,dashed](-3.47,2.4) ..controls (-3.47,2.4) and (-2.24,2.8) .. (-0.52,2.4);
					\draw[black] (-2.4,-0.3) node[right]{\footnotesize{$p$}};
					\draw[black] (-3.9,1.67) node[right]{\footnotesize{$e'$}};
					\draw[black] (-0.9,1.8) node[right]{\footnotesize{$c'$}};
					\draw[black] (-2.1,2.9) node[right]{\footnotesize{$w$}};
					\draw[black] (-4,4) node[right]{\footnotesize{$a'$}};
					\draw[black] (-0.7,3.95) node[right]{\footnotesize{$b'$}};
					\draw[black] (-4.3,2.3) node[right]{\footnotesize{$x'$}};
					\draw[black] (-0.6,2.5) node[right]{\footnotesize{$y'$}};
					\path[fill=blue!20,opacity=.3] (-0.4,3.6) to [out=158, in=31](-3.6,3.62) to[out=250, in=290] (-3.1,1.7) to [out=34, in=150](-0.9,1.7) to[out=70, in=-70] (-0.4,3.6);
					(3.1,7.375);
					\foreach \i/\tex in {0/(a)}
					\draw(-2,-4.2)node[below]{\tex};
					\foreach \i/\tex in {0/\text{}}
					\draw(0,-2)node[below]{\tex};
				\end{tikzpicture}\qquad\qquad\qquad
				\begin{tikzpicture}[scale=0.6]
					\draw[fill=black](-2,1.36)circle(.06);
					\draw[fill=black](-3.3,0.1)circle(.06);
					\draw[fill=black](-0.74,0.14)circle(.06);
					\draw[black](2,0)arc(0:360:4);    
					\draw[blue] (-1,3.87) ..controls (-1,3.87) and (-1,0.5) .. (-5.82,-1.2);
					\draw [blue](-3,3.87) ..controls (-3,3.87) and (-3,0.5) .. (1.82,-1.2);
					\draw [blue](-5.93,-0.8) ..controls (-5.93,-0.8) and (-2,1.5) .. (1.93,-0.8);	
					\path[fill=blue!20,opacity=.3] (-2,1.36) to [out=50, in=-120](-3.3,0.1) to[out=15, in=-15] (-0.74,0.14) to[out=135, in=-55](-2,1.36) ;
					\draw[black] (-2.3,1.85) node[right]{\footnotesize{$p'$}};
					\draw[black] (-3.45,-0.1) node[right]{\footnotesize{$h'$}};
					\draw[black] (-1,0.42) node[right]{\footnotesize{$s'$}};
					(3.1,7.375);
					\foreach \i/\tex in {0/(b)}
					\draw(-2,-4.2)node[below]{\tex};
					\foreach \i/\tex in {0/\text{}}
					\draw(0,-2)node[below]{\tex};
				\end{tikzpicture}
				\begin{tikzpicture}[scale=0.36,domain=0:180,>=stealth]
					\draw[fill=blue](8.5,15)circle(.12);
					\draw[fill=black](13.5,6)circle(.12);
					\draw[fill=black](3.5,6)circle(.12);
					\coordinate (org) at (8.5,8);
					\draw[blue](8.5,15)--(13.5,6);
					\draw[blue](8.5,15)--(3.5,6);
					\draw[blue](13.5,6)--(3.5,6);
					\draw[black] (7.9,15.7) node[right]{\footnotesize{$p'$}};
					\draw[black] (3,5.5) node[right]{\footnotesize{$h'$}};
					\draw[black] (13,5.5) node[right]{\footnotesize{$s'$}};
					\path[fill=blue!10,opacity=.3] (8.5,15) to [out=65, in=57](3.5,6) to[out=0, in=0] (13.5,6) to[out=-68, in=-55](8.5,15) ;
					\foreach \i/\tex in {0/(c)}
					\draw(8.5,1)node[below]{\tex};
					\foreach \i/\tex in {0/\text{}}
					\draw(0,0)node[below]{\tex};
				\end{tikzpicture}\qquad\,
				\begin{tikzpicture}[scale=0.36,domain=0:180,>=stealth]
					\draw[fill=blue](8.5,8)circle(.12);
					\draw[fill=black](8.5,15)circle(.12);
					\draw[fill=black](10,9.19)circle(.12);
					\draw[fill=black](12.4,9.59)circle(.12);
					\coordinate (org) at (8.5,8);
					\draw (8.5,8) circle[radius=7];
					\draw (org) ellipse (7cm and 2.2cm);
					\draw[white,dashed] plot ({7*cos(\x)+10},{2.2*sin(\x)+8});
					\draw (8.5,11.3) ellipse (6cm and 2.2cm);
					\draw[white,dashed] plot ({6*cos(\x)+10},{2.2*sin(\x)+11.3});
					\foreach\i/\text in{{3,1}/x,{19,8}/y,{10,18}/{}};
					\draw[black] (7.7,7.4) node[right]{\footnotesize{$z'$}};
					\draw[black] (7.9,15.7) node[right]{\footnotesize{$p'$}};
					\draw[black] (9.7,9.65) node[right]{\footnotesize{$h'$}};
					\draw[black] (11.8,9.25) node[right]{\footnotesize{$s'$}};
					\draw[blue] (8.5,15) ..controls (8.5,15) and (9.5,14) .. (10,9.19);
					\draw[blue] (8.5,15) ..controls (8.5,15) and (11.5,14) .. (12.4,9.59);
					\draw[blue] (10,9.19) ..controls (10,9.19) and (11.5,9.4) .. (12.4,9.59);
					\draw[black,dashed](8.5,8)--(15.5,8);
					\draw[black,dashed](8.5,8)--(8.5,15);
					\path[fill=blue!20,opacity=.3] (8.5,15) to [out=300, in=100](10,9.19) to[out=10, in=15] (12.4,9.59) to[out=100, in=-30](8.5,15) ;
					\foreach \i/\tex in {0/(d)}
					\draw(8.5,1)node[below]{\tex};
					\foreach \i/\tex in {0/\text{}}
					\draw(0,0)node[below]{\tex};
				\end{tikzpicture}
				\caption{ Figure (a) shows the concentric balls $B^M(p,\rho)$ and $B^M(p,\rho+2\delta)$ in the domain $B^M(p,\epsilon)$ of $\varphi$. The shaded part is the set $V_{\delta}(\varphi^{-1}(x))$. Figure (b) shows the set $M^{\underline{\kappa}(W)}$ (or $M^{\overline{\kappa}(W)}$), where $\underline{\kappa}(W)<0$ (or $\overline{\kappa}(W)<0$). Figure (c) shows the set $M^{\underline{\kappa}_W}$ (or $M^{\overline{\kappa}(W)}$), where $\underline{\kappa}(W)=0$ (or $\overline{\kappa}(W)=0$). Figure (d) shows the set $M^{\underline{\kappa}(W)}$ (or $M^{\overline{\kappa}(W)}$), where $\underline{\kappa}(W)>0$ (or $\overline{\kappa}(W)>0$). The shaded part of each of the figures (b), (c) and (d) is an Aleksandrow triangle in $M^{\underline{\kappa}(W)}$ (or $M^{\overline{\kappa}(W)}$).}\label{fig.2}
			\end{figure}
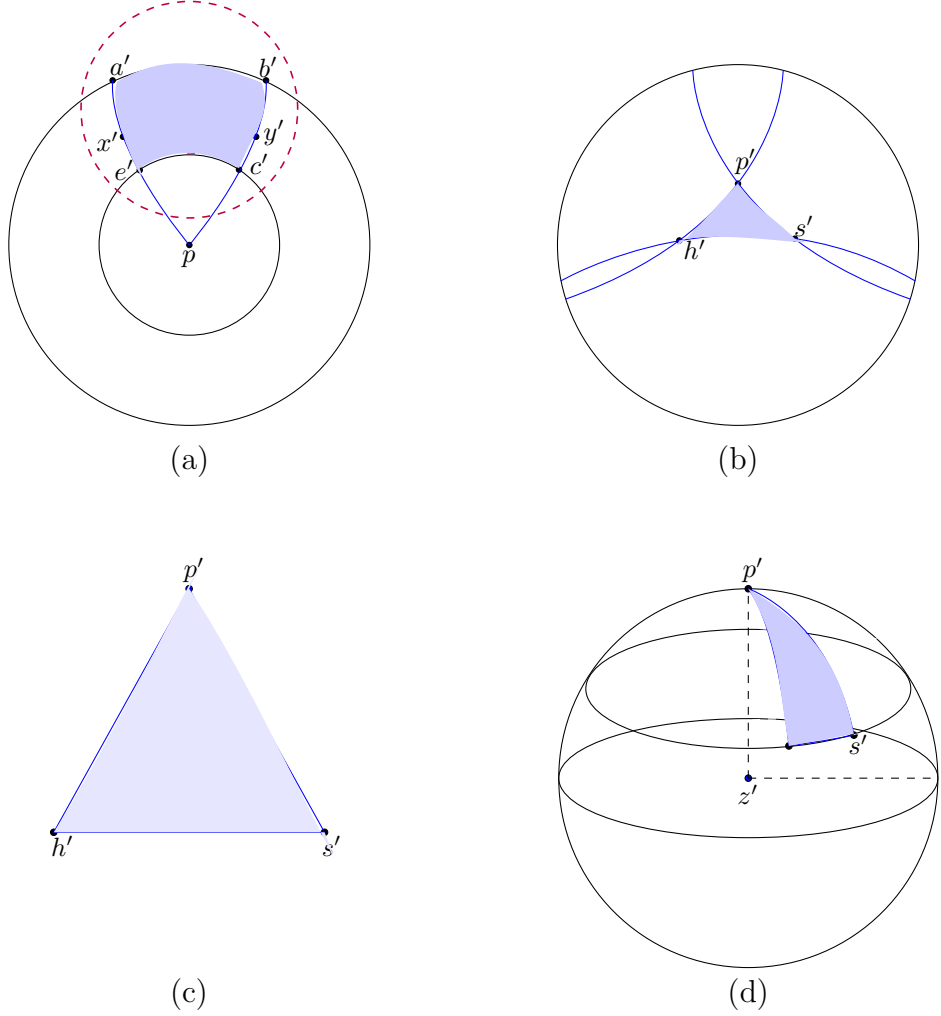
			
			Let $x'\in\gamma_{pa'}$ and $y'\in \gamma_{pb'}$ satisfy $L[\gamma_{px'}]=L[\gamma_{py'}]=\rho+\lambda$ for $\lambda\in[0,2\delta]$, and let $\gamma_{x'y'}$ be a minimal geodesic. Then $(\gamma_{px'},\gamma_{py'},\gamma_{x'y'})$ is a geodesic triangle on $M$. 
		 We claim that there exist positive constants $c_5$ and  $c_6$ such that 	
			\begin{align}\label{eq:d}
			c_5\delta\leq d_M(x', y')\leq c_6\delta.
			\end{align}
			Using $\eqref{eq:22.0}$,  we have
			\begin{align}\label{eq:22.1}
				-\infty<\underline{\kappa}(W)\leq K_M(w)\leq \overline{\kappa}(W)<\infty\quad \text{for all}\,\, w\in B^M(p,\epsilon),
			\end{align}
			where $\underline{\kappa}(W)$ and $\overline{\kappa}(W)$ are given in \eqref{cur}. To prove this claim, we consider the following six cases:
			
			 \noindent1. $\underline{\kappa}(W)<0$ and $\overline{\kappa}(W)<0$.\,\,\,\, 2. $\underline{\kappa}(W)=0<\overline{\kappa}(W)$. \,\, \,\,3. $\underline{\kappa}(W)<0<\overline{\kappa}(W)$.\\
		 4. $\underline{\kappa}(W)=0=\overline{\kappa}(W)$.\,\,\,\,\,\,\qquad\quad
			5. $\underline{\kappa}(W)<0=\overline{\kappa}(W)$.\,\,\,\,\,\, 6. $\underline{\kappa}(W)>0$ and $\overline{\kappa}(W)>0$.
			
			\noindent{\em Case 1}.	By \eqref{eq:22.1}, we have $\underline{\kappa}(W)\le K_M(w)$ for all  $w\in B^M(p,\epsilon)$. Theorem \ref{thm:3.6} implies that there exists an Aleksandrow triangle in $M^{\underline{\kappa}(W)}$. Since $\underline{\kappa}(W)<0$, we let $M^{\underline{\kappa}(W)}:=\{x=(x^1,x^2)\in \mathbb{R}^2|(x^1)^2+(x^2)^2<R_1^2\}$ be the Poincar\'e ball with radius $R_1$, where $\underline{\kappa}(W)=-1/R_{1}^{2}$ (see Figure \ref{fig.2}(b)). Let $\gamma_{p'h'}, \gamma_{p's'}\in M^{\underline{\kappa}(W)}$ such that $L[\gamma_{p'h'}]=L[\gamma_{px'}]=\rho+\lambda$, $L[\gamma_{p's'}]=L[\gamma_{py'}]=\rho+\lambda$ and $\measuredangle(\gamma_{p'h'}, \gamma_{p's'})=\measuredangle(\gamma_{px'}, \gamma_{py'})$. Then $(\gamma_{p'h'},\gamma_{p's'},\gamma_{h's'})$ is an Aleksandrow triangle in $M^{\underline{\kappa}(W)}$. By the hyperbolic law of cosines, we have 
			\begin{align*}
				\cosh\ \Big(\frac{L[\gamma_{h's'}]}{R_1}\Big)=\cosh^2\Big(\frac{\rho+\lambda }{R_1}\Big)-\sinh^2\Big(\frac{\rho+\lambda }{R_1}\Big)\cos\Big(\frac{4\delta}{\rho}\Big).
			\end{align*}
			Hence
			\begin{align*}
				L[\gamma_{h's'}]=R_1\cdot{\rm arccosh} \Big(\cosh^2\Big(\frac{\rho+\lambda }{R_1}\Big)-\sinh^2\Big(\frac{\rho+\lambda }{R_1}\Big)cos\Big(\frac{4\delta}{\rho}\Big)\Big).
			\end{align*}
			By Lemma \ref{lem:3.4}(a), for $\delta\in(0,\rho\pi/4]$ and  $\lambda \in [0,2\delta]$, we have 
			\begin{align*}
				L[\gamma_{h's'}]\leq R_1\cdot C_1\delta.
			\end{align*}
			Hence
			$$d_{M^{\underline{\kappa}(W)}}(h', s')= L[\gamma_{h's'}]\leq R_1\cdot C_1\delta=:c_1\delta.$$
			Using Theorem \ref{thm:3.6}, we have
			$$d_M(x', y')\leq d_{M^{\underline{\kappa}(W)}}(h', s')\leq c_1\delta.$$
			By \eqref{eq:22.1}, we have $\overline{\kappa}(W)\geq K_M(w)$ for all  $w\in B^M(p,\epsilon)$. Theorem \ref{thm:3.7} implies that there exists an Aleksandrow triangle in $M^{\overline{\kappa}(W)}$. Since $\overline{\kappa}(W)<0$, we let $M^{\overline{\kappa}(W)}:=\{x=(x^1,x^2)\in \mathbb{R}^2|(x^1)^2+(x^2)^2<R_2^2\}$ be a Poincar\'e ball, where $\overline{\kappa}(W)=-1/R_{2}^{2}$. Let $\gamma_{p'h'}, \gamma_{p's'}\in M^{\overline{\kappa}(W)}$ such that $L[\gamma_{p'h'}]=L[\gamma_{px'}]=\rho+\lambda$, $L[\gamma_{p's'}]=L[\gamma_{py'}]=\rho+\lambda$ and $\measuredangle(\gamma_{p'h'}, \gamma_{p's'})=\measuredangle(\gamma_{px'}, \gamma_{py'})$. Then $(\gamma_{p'h'},\gamma_{p's'},\gamma_{h's'})$ is an Aleksandrow triangle in $M^{\overline{\kappa}(W)}$. By the hyperbolic law of cosines, we have 
			\begin{align*}
				\cosh\Big(\frac{L[\gamma_{h's'}]}{R_2}\Big)=\cosh^2\Big(\frac{\rho+\lambda }{R_2}\Big)-\sinh^2\Big(\frac{\rho+\lambda }{R_2}\Big)\cos\Big(\frac{4\delta}{\rho}\Big).
			\end{align*}
			Hence
			\begin{align*}
				L[\gamma_{h's'}]=R_2\cdot{\rm arccosh} \Big(\cosh^2\Big(\frac{\rho+\lambda }{R_2}\Big)-\sinh^2\Big(\frac{\rho+\lambda }{R_2}\Big)\cos\Big(\frac{4\delta}{\rho}\Big)\Big).
			\end{align*}
			By Lemma \ref{lem:3.4} (a), for $\delta\in(0,\rho\pi/4]$  and $\lambda \in [0,2\delta]$, we have 
		$L[\gamma_{h's'}]\geq R_2\cdot C_2\delta.$
			Hence
			$$d_{M^{\overline{\kappa}(W)}}(h', s')=L[\gamma_{h's'}]\geq R_2\cdot C_2\delta=:c_2\delta.$$
			Using Theorem \ref{thm:3.7}, we have
			$$d_M(x', y')\geq d_{M^{\overline{\kappa}(W)}}(h', s')\geq c_2\delta.$$

		\noindent{\em Case 2}. By \eqref{eq:22.1}, we have $\underline{\kappa}(W)\le K_M(w)$ for all  $w\in B^M(p,\epsilon)$. Theorem \ref{thm:3.6} implies that there exists an Aleksandrow triangle in $M^{\underline{\kappa}(W)}$. Since $\underline{\kappa}(W)=0$, we let $M^{\underline{\kappa}(W)}:=\{x=(x^1,x^2)\in \mathbb{R}^2\}$ be the Euclidean plane. Let $\gamma_{p'h'}, \gamma_{p's'}\in M^{\underline{\kappa}(W)}$ such that $L[\gamma_{p'h'}]=L[\gamma_{px'}]=\rho+\lambda$, $L[\gamma_{p's'}]=L[\gamma_{py'}]=\rho+\lambda$ and $\measuredangle(\gamma_{p'h'}, \gamma_{p's'})=\measuredangle(\gamma_{px'}, \gamma_{py'})$. Then $(\gamma_{p'h'},\gamma_{p's'},\gamma_{h's'})$ is an Aleksandrow triangle in $M^{\underline{\kappa}(W)}$. By the law of cosines, we have 
			\begin{align*}
				L[\gamma_{h's'}]=2(\rho+\lambda)\sin\Big(\frac{2\delta}{\rho}\Big).
			\end{align*}
			By Lemma \ref{lem:3.4}, for $\delta\in(0,\rho\pi/4]$ and  $\lambda \in [0,2\delta]$, we have 
		$L[\gamma_{h's'}]\leq  c_3\delta.$
			Hence
			$$d_{M^{\underline{\kappa}(W)}}(h', s')= L[\gamma_{h's'}]\leq c_3\delta.$$
			Using Theorem \ref{thm:3.6}, we have
			$$d_M(x', y')\leq d_{M^{\underline{\kappa}(W)}}(h', s')\leq c_3\delta.$$
			By \eqref{eq:22.1}, we have $\overline{\kappa}(W)\geq K_M(w)$ for all  $w\in B^M(p,\epsilon)$. To apply Theorem \ref{thm:3.7}, we assume $L[\gamma_{px'}]$, $L[\gamma_{py'}]$, and $L[\gamma_{x'y'}]$ are all less than ${\rm inj}(\supp(\mu))$. By the definition of injectivity radius, $(\gamma_{px'},\gamma_{py'},\gamma_{x'y'})$ does not meet the cut locus of its vertices. For $2\rho+4\delta<\pi/\sqrt{\overline{\kappa}(W)}$, the  circumference of triangle $(\gamma_{px'},\gamma_{py'},\gamma_{x'y'})$ is less than $2\pi/\sqrt{\overline{\kappa}(W)}$. Hence for $\delta<\pi/(4\sqrt{\overline{\kappa}(W)})-\rho/2$, Theorem \ref{thm:3.7} implies that there exists an Aleksandrow triangle in $M^{\overline{\kappa}(W)}$. Since $\overline{\kappa}(W)>0$, we let $M^{\overline{\kappa}(W)}:=\{x=(x^1,x^2)\in \mathbb{R}^2|(x^1)^2+(x^2)^2=R_3^2\}$ be a sphere, where $\underline{\kappa}(W)=1/R_{3}^{2}$. Let $\gamma_{p'h'}, \gamma_{p's'}\in M^{\overline{\kappa}(W)}$ such that $L[\gamma_{p'h'}]=L[\gamma_{px'}]=\rho+\lambda$, $L[\gamma_{p's'}]=L[\gamma_{py'}]=\rho+\lambda$ and $\measuredangle(\gamma_{p'h'}, \gamma_{p's'})=\measuredangle(\gamma_{px'}, \gamma_{py'})$. Then $(\gamma_{p'h'},\gamma_{p's'},\gamma_{h's'})$ is an Aleksandrow triangle in $M^{\overline{\kappa}(W)}$. By the spherical law of cosines, we have 
			\begin{align*}
				\cos\Big(\frac{L[\gamma_{h's'}]}{R_3}\Big)=\cos^2\Big(\frac{\rho+\lambda }{R_3}\Big)-\sin^2\Big(\frac{\rho+\lambda }{R_3}\Big)\cos\Big(\frac{4\delta}{\rho}\Big).
			\end{align*}
			Hence
			\begin{align*}
				L[\gamma_{h's'}]=R_3\cdot{\rm arccos} \Big(\cos^2\Big(\frac{\rho+\lambda }{R_3}\Big)+\sin^2\Big(\frac{\rho+\lambda }{R_3}\Big)\cos\Big(\frac{4\delta}{\rho}\Big)\Big)
			\end{align*}
			By Lemma \ref{lem:3.4}, for $\delta\in(0,\rho\pi/4]$ and  $\lambda \in [0,2\delta]$, we have 
			\begin{align*}
				L[\gamma_{h's'}]\geq C_4\delta R_3.
			\end{align*}
			Hence
			$$d_{M^{\overline{\kappa}(W)}}(h', s')=L[\gamma_{h's'}]\geq C_4\delta R_3\delta=:c_4\delta.$$
			Using Theorem \ref{thm:3.7}, we have
			$$d_M(x', y')\geq d_{M^{\overline{\kappa}(W)}}(h', s')\geq c_4\delta.$$	

		For the other four cases, if $\underline{\kappa}(W)<0$ (or $\overline{\kappa}(W)<0$), then we let $M^{\underline{\kappa}(W)}$ (or $M^{\overline{\kappa}(W)}$) be a Poincar\'e ball; if $\underline{\kappa}(W)=0$ (or $\overline{\kappa}(W)=0$), then we let $M^{\underline{\kappa}(W)}$ (or $M^{\overline{\kappa}(W)}$) be a Euclidean plane; if $\overline{\kappa}(W)>0$ (or $\overline{\kappa}(W)>0$), then we let $M^{\overline{\kappa}(W)}$ (or $M^{\overline{\kappa}(W)}$) be a sphere. Therefore,  \eqref{eq:d} can be obtained by using a method similar to that in Cases 1 and 2.

			Combining the above six cases proves that for  $0<\rho<\min\{\pi/(4\sqrt{\overline{\kappa}(W)}),\epsilon\}$, there exist positive constants $c_7$ and $c_8$ such that for $\delta$ satisfying $0<\delta<\min\{\pi/(4\sqrt{\overline{\kappa}(W)})-\rho/2,\pi\rho/4\}$, we have	
			$$c_7\delta\leq d_M(x', y')\leq c_8\delta,$$
			where $c_7:=\min\{c_2,c_3,c_4\}$ and $c_8:=\max\{c_1,c_5,c_6\}$.
			Hence 
			$$c'\delta:=|(c_7/2)\delta-\delta|\leq d_M(x', w)\leq (c_8/2)\delta+\delta=:c\delta$$
			We see that for $x\in \varphi(U)\backslash\{\varphi(p)\}$ and $0<\rho<\min\{\pi/(4\sqrt{\overline{\kappa}(W)}),\epsilon\}$, there exist positive constants $c'$ and $c$ such that for $\delta$ satisfying $0<\delta<\min\{\pi/(4\sqrt{\overline{\kappa}(W)})-\rho/2,\pi\rho/4\}$, we have
			$$B^M(\varphi^{-1}(x),c'\delta)\subseteq V_{\delta}(\varphi^{-1}(x))\subseteq B^M(\varphi^{-1}(x),c\delta).$$
			Hence, for each $x\in \varphi_i(U_i)\backslash \{\varphi_i(p_i)\}$,  $i=1,\ldots,N$  and  $0<\rho_i<\min\{\pi/(4\sqrt{\overline{\kappa}(W)}),\epsilon_i\}$, there exist positive constants $c'$ and $c$ such that for $\delta_i$ satisfying $0<\delta_i<\min\{\pi/(4\sqrt{\overline{\kappa}(W)})\\-\rho_i/2,\pi\rho_i/4\}$, \eqref{eq:3.3.1} holds.
			\end{proof}
	
	\begin{rema}\label{rema:3.9}
		Let $M$ be a complete Riemannian $n$-manifold. Let $\mu$ be a positive finite Borel measure on $M$ with compact support. Then there exist coordinate charts $\{(U_i,\varphi_i)\}_{i=1}^N$, where each $\varphi_i$ is defined as in \eqref{eq:jj3.8}, such that  $\widetilde{\mu}:=\sum_{i=1}^{N}\mu\circ\varphi_i^{-1}$ is a positive finite Borel measure on $\R^n$ with compact support. In fact, since ${\rm supp}(\mu)$ is compact, it follows that ${\rm supp}(\mu)$ has a finite open cover $\{U_i\}_{i=1}^{N}$, where each $U_i:=B^M(p_i,\epsilon) $ is a geodesic ball. By definition, $\varphi_i$ is a diffeomorphism and $\{\varphi_i(U_i)\}_{i=1}^{N}$ are  disjoint. Hence each normal coordinate map $\varphi_i$ induces a measure $\widetilde{\mu}_i:=\mu \circ \varphi_i^{-1}$ on $\varphi_i(U_i)$. Therefore we can define a measure $\widetilde{\mu}:=\sum_{i=1}^{N}\mu\circ\varphi_i^{-1}$ with  ${\rm supp}(\widetilde{\mu})\subseteq \overline{\bigcup_{i=1}^N \varphi_i(U_i)}$. Moreover, ${\rm supp}(\widetilde{\mu})$ is compact.
	\end{rema}

	\begin{prop}\label{prop:3.9}
	Let $M$ be a complete Riemannian $n$-manifold. Let $\Omega \subseteq M$ be a bounded open set. Let $\mu$ be a positive finite Borel measure on $M$ such that $\operatorname{supp}(\mu) \subseteq \overline{\Omega}$ and $\mu(\Omega)>0$. Let $\widetilde{\mu}$ be defined as in Remark \ref{rema:3.9}. Let $\{(U_{i},\varphi_{i})\}_{i=1}^{N}$ be a family of coordinate charts such that $\supp(\mu)\subseteq \bigcup_{i=1}^NU_i$, where $U_i:=B^M(p_i,\epsilon_i)$, and $\varphi_i$ is defined with respect to $\epsilon_i$ as in  \eqref{eq:jj3.8}. Then for each $x\in \varphi_i(U_i)$,
		\begin{align}\label{eq:3.7.1}
		\underline{\dim}_{\widetilde{\mu}}(x)=\underline{\dim}_{\mu}(\varphi_i^{-1}(x)). 
		\end{align}
	\end{prop}
	\begin{proof}\,Fix $i\in\{1,\ldots,N\}$ and let $x\in \varphi_i(U_i)$. We consider the following two cases.
		
			\noindent{\em Case 1. $x\in \varphi_i(U_i)\backslash\{\varphi_i(p_i)\}$.}
	Define $$\underline{\operatorname{dim}}_{\widetilde{\mu}}^*(x):=\varliminf\limits_{\delta_i\rightarrow 0^{+}}\frac{\ln\widetilde{\mu}(\widetilde{V}_{\delta_i}(x))}{\ln \delta_i},$$
		where $\widetilde{V}_{\delta_i}(x)$ is defined as in Proposition \ref{prop:3.8}. Applying Proposition \ref{prop:3.8}(a), we have
		\begin{align*}
			\varliminf\limits_{{\delta_i}\rightarrow0^+}\frac{\ln\widetilde{\mu}(B(x,{\delta_i}))}{\ln{\delta_i}}\geq \varliminf\limits_{{\delta_i}\rightarrow0^+}\frac{\ln\widetilde{\mu}(\widetilde{V}_{\delta_i}(x))}{\ln{\delta_i}}\geq \varliminf\limits_{{\delta_i}'\rightarrow0^+}\frac{ \ln\widetilde{\mu}(B(x,{\delta_i}'))}{\ln{\delta_i}'},\\
		\end{align*}
		where ${\delta_i}'$ can be taken to be $10{\delta_i}$.
		Therefore
	$
			\underline{\operatorname{dim}}_{\widetilde{\mu}}(x)\geq\underline{\operatorname{dim}}_{\widetilde{\mu}}^*(x)\geq
			\underline{\operatorname{dim}}_{\widetilde{\mu}}(x).
	$
		It follows that
		$$
		\underline{\operatorname{dim}}_{\widetilde{\mu}}^*(x)=\underline{\operatorname{dim}}_{\widetilde{\mu}}(x)=\varliminf\limits_{{\delta_i}\rightarrow 0^{+}}\frac{\ln\widetilde{\mu}(B(x,{\delta_i}))}{\ln {\delta_i}}.$$
	Define $$\underline{\operatorname{dim}}_{\mu}^*(\varphi^{-1}_i(x)):=\varliminf\limits_{{\delta_i}\rightarrow 0^{+}}\frac{\ln\mu(V_{\delta_i}(\varphi^{-1}_i(x))}{\ln {\delta_i}},$$
		where $V_{\delta_i}(\varphi^{-1}_i(x)$ is defined as in Proposition \ref{prop:3.8}. Using Proposition \ref{prop:3.8}(b), we have
		\begin{align*}
			\varliminf\limits_{{\delta_i}\rightarrow0^+}\frac{\ln\mu(B^M(\varphi^{-1}_i(x),c_2{\delta_i}))}{\ln(c_2{\delta_i})\big(1-\ln c_2/\ln(c_2{\delta_i})\big)}&\geq \varliminf\limits_{{\delta_i}\rightarrow0^+}\frac{\ln\mu( V_{\delta_i}(\varphi^{-1}_i(x)))}{\ln{\delta_i}}\\&\geq \varliminf\limits_{c_3{\delta_i}\rightarrow0^+}\frac{ \ln\mu(B^M(\varphi^{-1}_i(x),c_3{\delta_i}))}{\ln(c_3{\delta_i})\big(1-\ln c_3/\ln(c_3{\delta_i})\big)}.
		\end{align*}
		Hence
	$
			\underline{\operatorname{dim}}_{\mu}(\varphi^{-1}_i(x))\geq\underline{\operatorname{dim}}_{\mu}^*(\varphi^{-1}_i(x))\geq
			\underline{\operatorname{dim}}_{\mu}(\varphi^{-1}_i(x)).
	$
		Consequently,
		$$
		\underline{\operatorname{dim}}_{\mu}^*(\varphi^{-1}_i(x))=\underline{\operatorname{dim}}_{\mu}(\varphi^{-1}_i(x))=\varliminf\limits_{{\delta_i}\rightarrow 0^{+}}\frac{\ln\mu(V_{\delta_i}(\varphi^{-1}_i(x)))}{\ln {\delta_i}}.$$
	Since $\mu(V_{\delta_i}(\varphi^{-1}_i(x)))=\mu(\varphi_i^{-1}(\widetilde{V}_{{\delta_i}}(x)))=\widetilde{\mu}(\widetilde{V}_{{\delta_i}}(x))$, it follows that \eqref{eq:3.7.1} holds.
	
\noindent{\em Case 2. $x=\varphi_i(p_i)$.}	Let $0<\delta_i<\epsilon_i$. By the definitions of $\widetilde{\mu}$ and $\varphi_i$, we have
	$$\mu(B^M(\varphi_i^{-1}(x),\delta_i))=\mu(\varphi_i^{-1}(B(x,\delta_i)))=\widetilde{\mu}(B(x,\delta_i)).$$
	Hence \eqref{eq:3.7.1} holds and this completes the proof.
\end{proof}
	
	The following lemma will be used in the proof of Theorem \ref{thm:3.11}.
	\begin{lem}\label{lem:3.10}
		Assume the hypotheses of Proposition \ref{prop:3.9}. Let $n\geq 2$ and $2<q<\infty$. The following conditions are equivalent:
		\begin{enumerate}
			\item[(a)]
			\begin{align}
				&\lim _{\delta \rightarrow 0^{+}} \sup _{w \in M ; r \in(0, \delta)} r^{1-n/2} \mu(B^M(w,r))^{1/q}=0 \quad \text { \rm for } n>2, \quad \text { \rm and } \label{eq:3.7}\\
				&\lim _{\delta \rightarrow 0^{+}} \sup _{w \in M ; r \in(0, \delta)}|\ln r|^{1/2} \mu(B^M(w,r))^{1/q}=0 \quad \text {\rm  for } n=2.\label{eq:3.8}
			\end{align}
			\item[(b)]
			\begin{align}
				&\lim _{\delta \rightarrow 0^{+}} \sup _{x \in \R^n ; r \in(0, \delta)} r^{1-n/2}\, \widetilde{\mu}\left(B(x,r)\right)^{1 / q}=0 \quad \text {\rm for } \,n>2, \quad \text {\rm and }\label{eq:3.9} \\
				&\lim _{\delta \rightarrow 0^{+}} \sup _{x \in \R^n ; r \in(0, \delta)}|\ln r|^{1 / 2}\, \widetilde{\mu}\left(B(x,r)\right)^{1 / q}=0 \quad \text {\rm for } \, n=2.\label{eq:3.10}
			\end{align}
		\end{enumerate}
	\end{lem}
	\begin{proof}\, Assume that $(a)$ holds. We first consider the case $n>2$.  Fix $i\in\{1,\ldots,N\}$. For any $w\in U_i$, let $x:=\varphi_i(w)\subseteq \varphi_i(U_i)$.  We consider the following two cases.
			
			\noindent{\em Case 1.  $x\in \varphi_i(U_i)\backslash\{\varphi_i(p_i)\}$.} Let $\beta_i/2$ and $\beta_i'/2$ be the polar angles of the open infinite cones $\Lambda_{\beta_i}(z_i,l_i)$ and $\Lambda_{\beta_i'}(z_i,l_i)$, respectively. Let $\delta_i$ and $r_i$ be positive numbers. For  $r_i \in(0, \delta_i)$ and $\beta_i'<\beta_i$, let $\widetilde{V}_{\delta_i}(x):=\widetilde{V}_{\delta_i,\rho_i,\beta_i,z_i,l_i}(x)$ be defined as in Proposition \ref{prop:3.8}, and let $V_{r_i}(w)=V_{r_i}(\varphi_i^{-1}(x)):=\varphi_i^{-1}(\widetilde{V}_{r_i}(x))=\varphi_i^{-1}(\widetilde{V}_{r_i,\rho_i,\beta_i,z_i,l_i}(x))$. For  $0<\rho_i<\min\{\pi/(4\sqrt{\overline{\kappa}(W)}),\epsilon_i\}$, Proposition  \ref{prop:3.8}(b) implies that there exist positive constants $c_2$ and $c_3$ such that for $r_i \in(0, \delta_i)$ and $\delta_i\in(0,\min\{\pi/(4\sqrt{\overline{\kappa}(W)})-\rho_i/2,\pi\rho_i/4\})$, we have
	$B^M(\varphi^{-1}_i(x),c_2r_i)\subseteq V_{r_i}(\varphi^{-1}_i(x))\subseteq B^M(\varphi^{-1}_i(x),c_3r_i).$
	Hence
		\begin{align*}
			&\lim _{c_2{\delta_i} \rightarrow 0^{+}}\sup _{w \in U_i\backslash\{p_i\} ; c_2{r_i} \in(0, c_2{\delta_i})}\frac{\mu(B^M(w,c_2{r_i}))^{1/q}}{(c_2{r_i})^{-1+n/2}c_2^{1-n/2}}\nonumber\\
			\leq& \lim _{{\delta_i} \rightarrow 0^{+}} \sup _{w \in U_i\backslash\{p_i\}  ; {r_i} \in(0, {\delta_i})}\frac{\mu( V_{r_i}(w))^{1/q}}{{r_i}^{-1+n/2}}\nonumber\\
			\leq&\lim _{c_3{\delta_i} \rightarrow 0^{+}}\sup _{w \in U_i\backslash\{p_i\}  ;c_3 {r_i} \in(0, c_3{\delta_i})}\frac{\mu( B^M(w,c_3{r_i}))^{1/q}}{(c_3{r_i})^{-1+n/2}c_3^{1-n/2}}.
		\end{align*}
 By using (\ref{eq:3.7}), we have
		$$\lim _{{\delta_i} \rightarrow 0^{+}} \sup _{w \in U_i\backslash\{p_i\}  ; {r_i} \in(0, {\delta_i})}\frac{\mu( V_{r_i}(w))^{1/q}}{{r_i}^{-1+n/2}}=0.$$
	Since $\mu( V_{r_i}(w))=\mu(\varphi_i^{-1}(\widetilde{V}_{r_i}(x)))=\widetilde{\mu}(\widetilde{V}_{r_i}(x))$, it follows that
		\begin{align}\label{eq:3.13}
			\lim _{{\delta_i} \rightarrow 0^{+}} \sup _{x \in \varphi_i(U_i)\backslash\{\varphi_i(p_i)\} ; {r_i} \in(0, {\delta_i})}\frac{\widetilde{\mu}(\widetilde{V}_{r_i}(x))^{1/q}}{{r_i}^{-1+n/2}}=0.
		\end{align}
		Using Proposition \ref{prop:3.8}(a), we have
		\begin{align*}
			&\lim _{{\delta_i} \rightarrow 0^{+}}\sup _{x \in  \varphi_i(U_i)\backslash\{\varphi_i(p_i)\}; {r_i} \in(0, {\delta_i})}\frac{\widetilde{\mu}(B(x,{r_i}))^{1/q}}{{r_i}^{-1+n/2}}\nonumber\\\leq& \lim _{{\delta_i} \rightarrow 0^{+}}\sup _{x \in  \varphi_i(U_i)\backslash\{\varphi_i(p_i)\} ; {r_i} \in(0, {\delta_i})}\frac{\widetilde{\mu}( \widetilde{V}_{r_i}(x))^{1/q}}{r_i^{-1+n/2}}\nonumber\\
			\leq&\lim _{10{\delta_i} \rightarrow 0^{+}}\sup _{x \in  \varphi_i(U_i)\backslash\{\varphi_i(p_i)\} ; {r_i}' \in(0, 10{\delta_i})}\frac{\widetilde{\mu}( B(x,{r_i}'))^{1/q}}{({r_i}')^{-1+n/2}10^{1-n/2}},
		\end{align*}
		where ${r_i}'$ can be taken to be $10{r_i}$. By using (\ref{eq:3.13}), we have
		\begin{align}\label{eq:3.7.2}
			\lim _{{\delta_i} \rightarrow 0^{+}}\sup _{x \in \varphi_i(U_i)\backslash\{\varphi_i(p_i)\} ; {r_i} \in(0, {\delta_i})}\frac{\widetilde{\mu}(B(x,{r_i}))^{1/q}}{{r_i}^{-1+n/2}}= 0.
		\end{align}
	
\noindent{\em Case 2. $x=\varphi_i(p_i)$.} Let $0<r_i<\delta_i<\epsilon_i$.  By the definitions of $\widetilde{\mu}$ and $\varphi_i$, we have
	$$\mu(B^M(w,r_i))=\mu(B^M(\varphi_i^{-1}(x),r_i))=\mu(\varphi_i^{-1}(B(x,r_i)))=\widetilde{\mu}(B(x,r_i)).$$
	Using \eqref{eq:3.7}, we have 	
	\begin{align}\label{eq:3.7.3}
		\lim _{\delta_i \rightarrow 0^{+}} \sup _{ r_i \in(0, \delta_i)} {r_i}^{1-n/2}\, \widetilde{\mu}\left(B(x,r_i)\right)^{1 / q}=0.
		\end{align}
	
	Combining \eqref{eq:3.7.2} and \eqref{eq:3.7.3} proves \eqref{eq:3.9}.
	
	For the case $n=2$, by using a similar proof we can show that \eqref{eq:3.10} holds. This proves $(b)$. The proof of the converse is similar; we omit the details.
	\end{proof}

	In Theorem \ref{thm:3.11}, we generalize the compact embedding theorem of Maz'ja\cite{Maz'ja_1985} to Riemannian manifolds. This generalization plays a major role in the proof of Theorem \ref{thm:1.1}.
	We first state the theorem of Maz'ja.
	\begin{thm}(Maz'ja \cite[Section 8.8, Theorems 3 and 4]{Maz'ja_1985})\label{thm:maz}
		Let $\widetilde{\mu}$ be a finite measure on $\R^n$. For $q>2$, the unit ball $\widetilde{B}:=\{u \in C_{c}^{\infty}\left(\R^n\right):\|u\|_{W^{1,2}_0\left(\R^n\right)} \leq 1\}$
		is relatively compact in $L^{q}\left(\R^n,  \widetilde{\mu}\right)$ if and only if
		\begin{align} 
		&\lim _{\delta \rightarrow 0^{+}} \sup _{x \in \R^n ; r \in(0, \delta)} r^{1-n/2}\, \widetilde{\mu}\left(B(x,r)\right)^{1 / q}=0, \label{maz1} \\
		\qquad	&\lim _{|x| \rightarrow \infty} \sup _{r\in (0,1)} r^{1-n/2} \widetilde{\mu} \left(B(x,r)\right)^{1 / q}=0\qquad \text { \rm for } n>2,\label{maz2}
		\end{align}
	and
		\begin{align}
		&\lim _{\delta \rightarrow 0^{+}} \sup _{x \in \R^n ; r \in(0, \delta)}|\ln r|^{1 / 2}\, \widetilde{\mu}\left(B(x,r)\right)^{1 / q}=0, \label{maz3}\\
		\qquad	&\lim _{|x| \rightarrow \infty} \sup _{r\in (0,1)} |\ln r|^{1/2} \widetilde{\mu}(B(x,r))^{1/q}=0\qquad \text {\rm  for} \,\, n=2. \label{maz4}
		\end{align}
	\end{thm}
	
	\begin{thm}\label{thm:3.11}
	Let $M$ be a complete Riemannian $n$-manifold. Let $\mu$ be a positive finite Borel measure on $M$ with
	bounded support.  For $q>2$, the unit ball $B:=\{u \in C_{c}^{\infty}\left(M\right):\|u\|_{W^{1,2}_0\left(M\right)} \leq 1\}$
		is relatively compact in $L^{q}\left(M, \mu\right)$ if and only if
	\begin{align}
		&\lim _{\delta \rightarrow 0^{+}} \sup _{w\in M ; r \in(0, \delta)} r^{1-n/2} \mu\left(B^M(w,r)\right)^{1 / q}=0 \quad \text {\rm for } n>2,  \label{eq:3.18} 
		\end{align}
	and
		\begin{align}
		&\lim _{\delta \rightarrow 0^{+}} \sup _{w \in M ; r \in(0, \delta)}|\ln r|^{1 / 2} \mu\left(B^M(w,r)\right)^{1 / q}=0 \quad \text {\rm for } n=2\label{eq:3.19}.
	\end{align}
	\end{thm}

	\begin{proof}\,Assume that (\ref{eq:3.18}) and (\ref{eq:3.19}) holds. We will prove that for $q>2$, the unit ball $B$
		is relatively compact in $L^{q}\left(M, \mu\right)$. We divide our proof into five steps.
		
		\noindent{\em Step 1}. Let $\{(U_{\alpha},\varphi_{\alpha})\}_{\alpha=1}^{N}$ be a family of normal coordinate charts such that $\supp(\mu)\subseteq \cup_{\alpha=1}^NU_\alpha$, where  $U_\alpha:=B^M(p_\alpha,\epsilon_\alpha)$, and $\varphi_\alpha$ is defined with respect to $\epsilon_\alpha$ as in  \eqref{eq:jj3.8}. Let $\widetilde{\mu}$ be as in Remark \ref{rema:3.9}.  By using Lemma \ref{lem:3.10}, we see that \eqref{eq:3.9}--\eqref{eq:3.10} hold.
Let $(U,\varphi)$ be a coordinate chart on $\{(U_{\alpha},\varphi_{\alpha})\}_{\alpha=1}^{N}$ and let $\widetilde{U}:=\varphi(U)$. Then
		\begin{align*}
			&\lim _{\delta \rightarrow 0^{+}} \sup _{x \in \widetilde{U} ; r \in(0, \delta)} r^{1-n/2} \widetilde{\mu}\left(B(x,r)\right)^{1 / q}=0 \quad \text { for } n>2, \quad \text { and } \\
			&\lim _{\delta \rightarrow 0^{+}} \sup _{x \in \widetilde{U} ; r \in(0, \delta)}|\ln r|^{1 / 2} \widetilde{\mu}\left(B(x,r)\right)^{1 / q}=0 \quad \text { for } n=2.
		\end{align*}
		Using Theorem \ref{thm:maz}, we see that the unit ball $\widetilde{B}_0:=\{u \in C_{c}^{\infty}(\widetilde{U}):\|u\|_{W^{1,2}_0(\widetilde{U})} \leq 1\}$
		is relatively compact in $L^{q}(\widetilde{U}, \widetilde{\mu})$.
		
		\noindent{\em Step 2}. It is straightforward to show that $f\in C_c^{\infty}(U)$ if and only if $\widetilde{f}:=f\circ\varphi^{-1} \in C_c^{\infty}(\widetilde{U})$. Hence we can define a bijection $I_1: C^{\infty}_c(\widetilde{U})\rightarrow C^{\infty}_c(U)$ by $I_1(\widetilde{f}):=\widetilde{f}\circ \varphi$. Next, we prove the following claims.
		
		\noindent{\em Claim 1. For any $f\in C^{\infty}_c(U)$, we have $\|f\|_{W^{1,2}_0(U)}=\|f\circ\varphi^{-1}\|_{W^{1,2}_0(\widetilde{U})}$}. In fact, for any $f\in C_c^{\infty}(U)$, we have
		\begin{align*}
			\|f\|_{W^{1,2}_0(U)}^2&=\int_U|f|^2\,d\nu+\int_U|\nabla f|^2\,d\nu=\int_U|f|^2\,d\nu+\int_U\sum_j(\partial_jf)^2\,d\nu\nonumber\\
			&=\int_{\widetilde{U}}|f\circ \varphi^{-1}|^2 dx+\int_{\widetilde{U}}|\nabla(f\circ \varphi^{-1})|^2 dx
			=\|f\circ\varphi^{-1}\|^2_{W^{1,2}_0(\widetilde{U})}.
		\end{align*}
		\noindent{\em Claim 2. For any $f\in C^{\infty}_c(U)$, we have $\|f\|_{L^q(U, \mu)}=\|f\circ\varphi^{-1}\|_{L^q(\widetilde{U},\,\widetilde{\mu})}$.} In fact, by a change of variables, we have
		\begin{align*}
			\|f\|_{L^{q}(U,\mu)}^{q} &=\int_{U}|f|^{q} \,d \mu=\int_{\widetilde{U}}\left|f \circ \varphi^{-1}\right|^{q} \,d \mu\circ \varphi^{-1} \\
			&=\int_{\widetilde{U}} | f \circ \varphi^{-1}|^{q} \,d \widetilde{\mu}
			=\left\|f \circ \varphi^{-1}\right\|_{L^{q}(\widetilde{U},\widetilde{\mu})}^{q}.
		\end{align*}
		
		\noindent{\em Step 3}. Let $[u]_{\mu}$ (resp. $[\widetilde{v}]_{\widetilde{\mu}}$) be the equivalence class of $u$ in $L^q(U,\mu)$ (resp. $\widetilde{v}$ in $L^q(\widetilde{U},\widetilde{\mu})$), and let $B_0:=\{u \in C_{c}^{\infty}(U):\|u\|_{W^{1,2}_0\left(U\right)} \leq 1\}$. Define a map $I_2:\widetilde{B}_0\rightarrow B_0$ by $I_2(\widetilde{u}):=\widetilde{u}\circ \varphi$, which is well-defined and bijective by Step $2$. Define maps $I_3: B_0\rightarrow L^q(U, \mu)$ by $I_3(u):=[u]_{\mu}$, and  $I_4: \widetilde{B}_0\rightarrow L^q(\widetilde{U}, \widetilde{\mu})$ by $I_4(\widetilde{u}):=[\widetilde{u}]_{\widetilde{\mu}}$. It is straightforward to show that the maps $I_3$ and $I_4$ are well-defined. Given $[\widetilde{u}]_{\widetilde{\mu}}\in L^q(\widetilde{U},\widetilde{\mu})$, we can define 
	\begin{align*}
	[\widetilde{u}]_{\widetilde{\mu}}\circ \varphi:=[\widetilde{u}\circ \varphi]_{\widetilde{\mu}\circ \varphi}=[\widetilde{u}\circ \varphi]_{\mu}\in L^q(U,\mu).
		\end{align*}	
	To see that this is well-defined, we notice that Claim 2 implies that
$
	\int_{U}| \widetilde{\mu}\circ \varphi|^{q} \,d \mu= \int_{\widetilde{U}} |\widetilde{\mu}|^q \,d \widetilde{\mu}<\infty.
$
	 Thus $\widetilde{u}\circ \varphi\in L^q(U,\mu)$. Let $\widetilde{u}$ and $\widetilde{u}'$ be representatives of $[\widetilde{u}]_{\widetilde{\mu}}$. Then
	$$\|\widetilde{u}\circ \varphi-\widetilde{u}'\circ \varphi\|_{L^q(U, \mu)}=\|\widetilde{u}-\widetilde{u}'\|_{L^q(\widetilde{U}, \widetilde{\mu})}= 0.$$
Hence $\widetilde{u}\circ \varphi=\widetilde{u}'\circ \varphi$  in $L^q(U,\mu)$. 

Now we define a map $I_5: L^q(\widetilde{U}, \widetilde{\mu})\rightarrow L^q(U, \mu)$ by $I_5([\widetilde{u}]_{\widetilde{\mu}}):=[\widetilde{u}]_{\widetilde{\mu}}\circ\varphi$.  This is also well-defined.  In fact, we have $[\widetilde{u}]_{\widetilde{\mu}}\circ\varphi\in L^q(U,\mu)$ by definition. Let $\widetilde{u}$ and $\widetilde{u}'$ be representatives of $[\widetilde{u}]_{\widetilde{\mu}}$. Then
$$[\widetilde{u}]_{\widetilde{\mu}}\circ \varphi=[\widetilde{u}\circ \varphi]_{\mu}=[\widetilde{u}'\circ \varphi]_{\mu}=[\widetilde{u}']_{\widetilde{\mu}}\circ \varphi.$$ 
This proves that $I_5$ is well-defined. It can be verified directly that  $I_3=I_5\circ I_4\circ I_2^{-1}$.
		\[ \psset{arrows=->, arrowinset=0.25, linewidth=0.6pt, nodesep=3pt, labelsep=2pt, rowsep=0.7cm, colsep = 1.1cm, shortput =tablr}
		\everypsbox{\scriptstyle}
		\]
		\[ \begin{tikzcd}
			\widetilde{B}_0 \arrow{r}{I_2} \arrow[swap]{d}{I_4} & B_0 \arrow{d}{I_3} \\
			L^q(\widetilde{U},\widetilde{\mu}) \arrow{r}{I_5}& L^q(U,\mu)
		\end{tikzcd}
		\]

		\noindent{\em Step 4}. We show that the map $I_3: B_0\rightarrow L^q(U, \mu)$ is compact. Since the unit ball $\widetilde{B}_0$ is relatively compact in $L^{q}(\widetilde{U}, \widetilde{\mu})$,  $I_4$ is compact.  Let $(u_n)$ be a bounded sequence in $B_0$, i.e., $u_n\in C^{\infty}_c(U)$ and $\|u_n\|_{W^{1,2}_0(U)}\leq 1$. Then $(u_n\circ\varphi^{-1})\in  C^{\infty}_c(\widetilde{U})$, and $\|u_n\circ\varphi^{-1}\|_{W^{1,2}_0(\widetilde{U})}\leq 1$ by Claim 1.
		Since $I_4$ is compact, we see that  $(u_n\circ\varphi^{-1})$ has a convergent subsequence $(u_{n_j}\circ\varphi^{-1})$ in $L^q(\widetilde{U}, \widetilde{\mu})$. Hence there exists $\widetilde{u}\in L^q(\widetilde{U}, \widetilde{\mu})$ such that
		\begin{align}\label{eq:j3.11}
			\|u_{n_j}\circ\varphi^{-1}-\widetilde{u}\|_{L^q(\widetilde{U}, \widetilde{\mu})}\rightarrow 0\,\,(j\rightarrow\infty).
		\end{align}
		Let $u:=\widetilde{u}\circ\varphi$. By using Claim 2, we see that $u\in L^q(U, \mu)$. Combining  \eqref{eq:j3.11} and Claim 2, we obtain $\|u_{n_j}-u\|_{L^q(U, \mu)}\rightarrow 0\,(j\rightarrow\infty)$.  Thus $I_3$ is compact.
		
		\noindent{\em Step 5}. We prove that $B$ is relatively compact in $L^q(M, \mu)$. We know that  $\{U_{\alpha}\}_{\alpha=1}^{N}$ is a finite open cover of $\supp(\mu)$. Let $f_\alpha\in C^\infty_c(U_\alpha)$. For each $\widetilde{\epsilon}_\alpha>0$, $\alpha=2,\ldots,N$, there exists $\delta_\alpha\in(0,\epsilon_\alpha)$ such that for $\delta\in(0,\delta_\alpha)$,
			\begin{align}
			&\int_{\cup_{i=1}^{\alpha-1}(U_{\alpha}\cap(U_i\backslash B^M(p_i,\epsilon_i-\delta)))}|f_\alpha|^2\,d\nu<\frac{\epsilon_\alpha}{2}\qquad\text{and} \label{eq:int1}\\
				&\int_{\cup_{i=1}^{\alpha-1}(U_{\alpha}\cap(U_i\backslash B^M(p_i,\epsilon_i-\delta)))}|\nabla f_\alpha|^2\,d\nu<\frac{\epsilon_\alpha}{2}\label{eq:int}.
			\end{align}
		Let $\underline{\delta}:=\min\{\delta_\alpha:\alpha=2,\ldots,N\}$, $V_1:=U_1$, and $V_{\alpha}:=U_{\alpha}\backslash\overline{\cup_{i=1}^{\alpha-1} (U_i\cap B^M(p_i,\epsilon_i-\underline{\delta}))}$ for $\alpha=2,\ldots,N$.
		Then $V_{\alpha}\subseteq U_{\alpha}$ for $\alpha=1,\ldots,N$, and $\{V_{\alpha}\}_{\alpha=1}^{N}$ is a  finite open cover of $\supp(\mu)$. Let $\{g_{\alpha}\}_{\alpha=1}^{N}$ be a $C^{\infty}$ partition of unity subordinate to $\{V_{\alpha}\}_{\alpha=1}^{N}$. Let $(u_n)$ be a bounded sequence in $B$. Then we can write $u_n|_{\supp(\mu)}=\sum_\alpha g_{\alpha}u_n$, and thus ${\rm supp}(g_{\alpha}u_n)\subseteq V_{\alpha}$ and $g_{\alpha}u_n\in C_{c}^{\infty}\left(V_{\alpha}\right)$.
		Let $D_1:=V_1$, $D_{\alpha}:=U_{\alpha}\backslash\overline{\cup_{i=1}^{\alpha-1}(U_\alpha\cap U_i)}$, $\alpha=2,\ldots,N$. Then $D_\alpha\subseteq V_\alpha$ for $\alpha=1,\ldots,N$, and $\{D_\alpha\}_{\alpha=1}^N$ is a disjoint collection.
	As $\|u_n\|_{W^{1,2}_0(M)}\leq 1$, we have $\|g_\alpha u_n\|_{W^{1,2}_0(D_\alpha)}\leq 1$. Combining \eqref{eq:int1} and \eqref{eq:int}, we obtain $\|g_\alpha u_n\|_{W^{1,2}_0(V_\alpha\backslash D_\alpha)}<\widetilde{\epsilon}_\alpha$. Hence $\|g_\alpha u_n\|_{W^{1,2}_0(V_\alpha)}\leq 1+\widetilde{\epsilon}_\alpha\leq2$.
		Thus $(g_{\alpha}u_n)$ is a bounded sequence in $ B_{{\alpha}}:=\{u \in C_{c}^{\infty}\left(V_{\alpha}\right):\|u\|_{W^{1,2}_0\left(V_{\alpha}\right)} \leq 2\}$. Analogous to the mapping $I_3$ in Step 3, we defined $T_{\alpha}: B_{{\alpha}}\rightarrow  L^q(V_{\alpha}, \mu)$. Then for all $\alpha=1,\ldots,N$, $T_{\alpha}$ are compact by Step 4.    Since $T_1$ is compact, there exists a subsequence  $(g_{1}u_{n_{j}}^{(1)}) \subset (g_1u_{n})$ converging to $u^{(1)}$ in $L^q(V_1, \mu)$. Similarly, there exists a subsequence  $(g_{2}u_{n_{j}}^{(2)}) \subset (g_{2}u_{n_{j}}^{(1)}) $ converging to $u^{(2)}$ in $L^q(V_2, \mu)$.  Continuing in this way, we see that the ``diagonal sequence" $(g_{\alpha}u_{n_{j}}^{(\alpha)})$ is a subsequence of $(g_{\alpha}u_n)$ such that for every $\alpha=1,\ldots,N$, $g_{\alpha}u_{n_{j}}^{(\alpha)}\rightarrow u^{(\alpha)}\, (j\rightarrow\infty)$ in $L^q(V_{\alpha},\mu)$.
		Write $u_{n_j}^{(N)}:=\sum_{\alpha=1}^{N}g_{\alpha}u_{n_{j}}^{(\alpha)}$ and $u:=\sum_{\alpha=1}^{N} u^{(\alpha)} \in L^q(M,\mu)$.
		Then
		\begin{align*}
			\|u_{n_j}^{(N)}-u\|^q_{L^q(M, \mu)}&=\int_{\supp(\mu)}\Big|\sum_{\alpha=1}^{N} g_{\alpha}u_{n_{j}}^{(\alpha)}-\sum_{\alpha=1}^{N} u^{(\alpha)}\Big|^q\,d\mu\\
			&\leq C\sum_{\alpha=1}^{N} \int_{V_{\alpha}}|g_{\alpha}u_{n_{j}}^{(\alpha)}-u^{(\alpha)}|^q\,d\mu
			\rightarrow 0 \,\,(j\rightarrow \infty).
		\end{align*}
		Hence  $B$ is relatively compact in $L^q(M, \mu)$.

		Conversely, assume that for $q>2,$ the unit ball $B$ is relatively compact in $L^q(M, \mu)$. By using a method similar to Steps $3$ and 4, we have $I_4=I_5^{-1}\circ I_3\circ I_2$ and $I_4$ is compact. Applying Lemma \ref{lem:3.10} and Theorem \ref{thm:maz}, we can show that \eqref{eq:3.18} and \eqref{eq:3.19} hold.	
	\end{proof}

	\begin{thm}\label{thm:3.12}Let $M$ be a complete Riemannian $n$-manifold. Let $n\geq 2$ and $2<q<\infty$, and let $\mu$ be a finite positive Borel measure on $M$ with compact support.  Let $B:=\{u \in C_{c}^{\infty}(M):\|u\|_{W^{1,2}_0(M)} \leq 1\} $.
		\begin{enumerate}
			\item[(a)] If $\underline{\operatorname{dim}}_{\infty}(\mu)>q(n-2) / 2$, then $B$ is relatively compact in $L^{q}(M, \mu)$.
			\item[(b)] If $\underline{\operatorname{dim}}_{\infty}(\mu)<q(n-2) / 2$, then $B$ is not relatively compact in $L^{q}(M, \mu)$.
		\end{enumerate}
	\end{thm}
	The proof of this theorem mainly uses Lemma \ref{lem:3.1}(b) and Theorem \ref{thm:3.11}; it is similar to that of \cite[Theorem 3.2]{Hu-Lau-Ngai_2006} and so is omitted.

	We let $C(\overline{\Omega})$ be the space of all  functions in $C(\Omega)$ which are bounded and uniformly continuous on $\Omega$. To prove Theorem \ref{thm:1.1}, we need the following theorem. 
	\begin{thm}\label{lem:3.13}
		Let $M$ be a complete Riemannian $1$-manifold, and let $\Omega \subseteq M$ be a bounded open set. Then $W^{1,2}_0(\Omega)$ is compactly embedded in $C(\overline{\Omega})$.
	\end{thm}
	
 We postpone its proof to the appendix. A similar theorem with $\Omega$ replaced by a compact manifold can be found in \cite[Theorem 8.2]{Hebey-Robert_2008}.

	\begin{proof}[Proof of Theorem \ref{thm:1.1}]\,
		(a) For the case $n=1$, we have $W^{1,2}_0(\Omega)$ is compactly embedded in $C(\overline{\Omega})$ by Theorem \ref{lem:3.13}.  Let $(u_m)$ be a bounded sequence in $W^{1,2}_0(\Omega)$. Then there exists a subsequence $(u_{m_k})$ that converges to some $u\in C(\overline{\Omega})$. Hence
		\begin{align*}
			\lim\limits_{k\rightarrow\infty}\|u_{m_k}-u\|_{C(\overline{\Omega})}=\lim\limits_{k\rightarrow\infty}\|u_{m_k}-u\|_\infty=0.
		\end{align*}
		We know that $u_{m_k}\in C(\overline{\Omega})\subseteq L^{2}(\Omega, \mu)$ and $u\in C(\overline{\Omega})\subseteq L^{2}(\Omega, \mu)$. Thus
		\begin{align*}
			\lim\limits_{k\rightarrow\infty} \int_\Omega|u_{m_k}-u|^2\,\,d\mu\leq\lim\limits_{k\rightarrow\infty}\int_\Omega\|u_{m_k}-u\|^2_\infty \,d\mu=\lim\limits_{k\rightarrow\infty}\|u_{m_k}-u\|^2_\infty\int_\Omega \,d\mu=0.
		\end{align*}
		Hence $(u_{m_k})$ converges to $u$ in $L^{2}(\Omega, \mu)$. Therefore $W^{1,2}_0(\Omega)$ is compactly embedded in $L^{2}(\Omega, \mu)$, and thus (MPID) holds.  Moreover, the embedding $\operatorname{dom}(\mathcal{E}_D) \hookrightarrow L^{2}(\Omega, \mu)$ is compact.  By using Theorem \ref{thm:3.12} and the argument in \cite[Theorem 1.1]{Hu-Lau-Ngai_2006}, we can prove that Theorem \ref{thm:1.1} holds in the case $n\geq 2$; we omit the proof.
		
	(b) For $n=1$ and $\partial\Omega=\emptyset$, i.e., $\Omega=M$ is compact, we have $W^{1,2}(M)$ is compactly embedded in $C(M)$ by \cite[Theorem 8.2]{Hebey-Robert_2008}.  By using a  method similar to that of (a), we have  $W(M)$  is compactly embedded in $L^{2}(\Omega, \mu)$, and thus (MPIE) holds. Moreover, the embedding $\operatorname{dom}(\mathcal{E}_E) \hookrightarrow L^{2}(M, \mu)$ is compact. Similarly, by using Theorem \ref{thm:3.12} and the argument in \cite[Theorem 1.1]{Hu-Lau-Ngai_2006}, we can prove that Theorem \ref{thm:1.1} holds in the case $n\geq 2$; we omit the  details.
			\end{proof}
		
		\begin{proof}[Proof of Theorem \ref{thm:1.2}]\, This theorem is a direct consequence of Theorem \ref{thm:1.1} and \cite[Theorem B.1.13]{Kigami_2001}; we  omit the details.	
		\end{proof}

	\section{Compact embedding theorem for measures without compact support}\label{S:comp}
\setcounter{equation}{0}
\subsection{Kre\u{\i}n-Feller operators defined by measures without compact support}

	Let $M$ be a complete Riemannian $n$-manifold and let $\Omega\subseteq M$ be an open subset. We recall that if $\partial \Omega=\emptyset$, then $\Omega=M$.   We recall the following definitions of {\em Poincar\'e inequalities} depending on  boundary conditions:
		
\noindent ($\partial\Omega\neq \emptyset$)   If there exists a constant $C>0$ such that for all $u\in C_c^{\infty}({\Omega})$,
		\begin{align}\label{eq:PI1*}
			\int_\Omega |u|^2\,d\nu\leq C\int_\Omega |\nabla u|^2\,d\nu,
		\end{align}
		then we say that the {\em Poincar\'e inequality holds for the case $\partial\Omega\neq\emptyset$ (PID)}.
		
\noindent($\partial \Omega=\emptyset$)  If there exists a constant $C>0$ such that for all $u\in \mathcal{C}(M)$,
		\begin{align}\label{eq:PIN1*}
			\int_M |u|^2\,d\nu\leq C\int_M |\nabla u|^2\,d\nu,
		\end{align}
		then we say that the {\em Poincar\'e inequality holds for the case $\partial\Omega=\emptyset$ (PIE)}.

	\begin{thm}[McKean \cite{McKean_1970}]
		Let $M$ be a complete simply connected
	Riemannian $n$-manifold with sectional curvature bounded from above by a negative constant $-\kappa$. Let $\lambda(D)$ be the first (non-zero) eigenvalue (for either the Dirichlet problem
	or the Neumann problem) of the Laplacian for the compact subdomain $D$.  Then
	$\lambda(M):=\inf \lambda(D) \geq (n-1)^2\kappa/4$.
	\end{thm}
 Notice that the condition $\lambda(M)>0$ is equivalent to the validity of (PID) for all open subsets $\Omega\subseteq M$, and the condition $\lambda(M)> 0$ is equivalent to the validity of (PIE) for $M$. McKean's theorem implies that a complete simply connected Riemannian manifold $M$ with negative sectional curvature satisfies (PID) or (PIE). 

Next, we provide some other examples of manifolds that satisfy (PID). Let $M$ be a complete noncompact Riemannian $n$-manifold. We say that $M$ has {\em $k$ ends $E_1,\ldots,E_k$} if there exists a bounded open set $U$ such that
	\begin{align}\label{ends}
		M\backslash\overline{U}=\bigcup_{i=1}^kE_i,	
	\end{align}
	where each $E_i$ is a noncompact connected component of $	M\backslash\overline{U}$.	By \cite[Lemma 4.1]{Ding-Wang_1991}, suppose that $M$ has  $k$ ends $E_1,\ldots,E_k$ corresponding to the decomposition \eqref{ends}. If ${\rm Vol}(M)=\infty$ and $\lambda(E_i)>0$ for $1\leq i\leq k$, then $\lambda(M)>0$. For $n=3$, the following example in \cite{Benjamini-Cao_1996} shows that $\lambda(M)>0$.
	\begin{exam}(Benjamini and Cao \cite[ Proposition 1.1]{Benjamini-Cao_1996})
		 Let $g = e^{2\eta_1(s)}dx^2 + e^{2\eta_2(s)}dy^2 + ds^2$ be a warped product metric defined on $\R^3$. Suppose that there exists a constant $\delta>0$ such that 
		$$\eta'_1(s)+\eta'_2(s)\geq \delta>0$$
		holds for all $s\in\R$. Then the Riemannian manifold $M:=(\R^3,g)$ has positive Cheeger isoperimetric constant $h(M)\geq\delta>0$ and hence
		$\lambda(M)\geq\delta^2/4>0$.
	Thus {\rm (PID)} holds.
	\end{exam}

	Let $\mu$ be a positive finite Borel measure on $M$ with unbounded support and with $\supp(\mu)\subseteq \overline{\Omega}$. In this subsection, if $\partial\Omega\neq\emptyset$, then we define the Kre\u{\i}n-Feller operator $\Delta_{\mu}^D$ on $L^2(\Omega, \mu)$. If $\partial\Omega=\emptyset$, then $\Omega=M$,  and we  define the Kre\u{\i}n-Feller operator $\Delta_{\mu}^E$ on $L^2(M, \mu)$. The method is similar to that in Section \ref{S:L}.

	We let (MPID) and (MPIE) be as in \eqref{eq:PI} and \eqref{eq:PIN}; let $\mathcal{E}_D$, $\mathcal{E}_E$, and  $\mathcal{E}_{*}$ be the quadratic forms defined as in (\ref{eq(1.1)}), \eqref{eq:new}, and (\ref{eq(2.1)}), respectively, where $\Omega\subseteq M$ is an open set.
	\begin{prop}\label{prop:9.1}
		Let $M$ be a complete Riemannian $n$-manifold and let $\Omega\subseteq M$ be an open subset. Let $\mu$ be a positive finite Borel measure on $M$ (need not have compact support) and assume that  $\supp(\mu)\subseteq \overline{\Omega}$.
	\begin{enumerate}
		\item[(a)] Assume  that $\partial\Omega\neq \emptyset$, (PID) holds, and $\mu$ satisfies (MPID). Then $\operatorname{dom}(\mathcal{E}_D)$ is dense in $L^{2}(\Omega, \mu)$. Moreover,  $\left(\mathcal{E}_{D}^*, \operatorname{dom}(\mathcal{E}_D)\right)$ is a Hilbert space.
		\item[(b)] Assume that  $\partial\Omega= \emptyset$, (PIE) holds, and $\mu$ satisfies (MPIE). Then $\operatorname{dom}(\mathcal{E}_E)$ is dense in $L^{2}(M, \mu)$. Moreover,  $\left(\mathcal{E}_{E}^*, \operatorname{dom}(\mathcal{E}_E)\right)$ is a Hilbert space.
\end{enumerate}
	\end{prop}
	Proposition \ref{prop:9.1} can be proved by using Lemma \ref{lem:2.2}. The proof is also similar to that of 	Proposition \ref{prop:2.3} and is omitted.

	Proposition \ref{prop:9.1} implies that if condition (PID) holds and $\mu$ satisfies (MPID), then the quadratic form $(\mathcal{E}_D, \operatorname{dom}(\mathcal{E}_D))$ is closed on $L^{2}(\Omega, \mu)$.  Hence there exists a nonnegative self-adjoint operator $H$ on $L^{2}(\Omega, \mu)$ such that $\operatorname{dom}(H) \subseteq \operatorname{dom}\left(H^{1 / 2}\right)=\operatorname{dom}(\mathcal{E}_D)$ and for all  $u, v \in \operatorname{dom}(\mathcal{E}_D)$,
	$$
	\mathcal{E}_D(u, v)=\left\langle H^{1 / 2} u, H^{1 / 2} v\right\rangle_{L^{2}(\Omega, \mu)}.
	$$
	Moreover, $u \in \operatorname{dom}(\Delta_\mu^D)$ if and only if $u \in \operatorname{dom}(\mathcal{E}_D)$ and there exists $f \in L^{2}(\Omega, \mu)$ such that $\mathcal{E}_D(u, v)=\langle f, v\rangle_{L^{2}(\Omega, \mu)}$ for all $v \in \operatorname{dom}(\mathcal{E}_D)$. Note that for all $u \in \operatorname{dom}(H)$ and $v \in \operatorname{dom}(\mathcal{E}_D)$,
	\begin{align}\label{eq:j(2.2)}
		\int_{\Omega} \langle\nabla u, \nabla v\rangle \, d \nu=\mathcal{E}_D(u, v)=\langle H u, v\rangle_{L^{2}(\Omega, \mu)}.
	\end{align}
Similarly, $\mathcal{E}_E(u,v)$ satisfies the above properties. We let $\Delta_{\mu}^D:=-H$ and call it the   {\em Kre\u{\i}n-Feller operator with respect to} $\mu$. Similarly, we can define $\Delta_\mu^E$.

We omit the proofs of Proposition \ref{prop:j2.4} and Theorem \ref{thm:j2.5} as they are similar to those of \cite[Proposition 2.2]{Hu-Lau-Ngai_2006} and \cite[Theorem 2.3]{Hu-Lau-Ngai_2006}, respectively.
	\begin{prop}\label{prop:j2.4}
			Let $n$, $M$, $\Omega$, and $\mu$ be defined as in Proposition \ref{prop:9.1}.
		\begin{enumerate}
			\item[(a)] 	Assume  that $\partial\Omega\neq \emptyset$, {\rm(PID)} holds, and $\mu$ satisfies {\rm(MPID)}. For $u \in \operatorname{dom}(\mathcal{E}_D)$ and $f \in L^{2}(\Omega, \mu)$, the following conditions are equivalent:
			\begin{enumerate}
				\item[(i)] $u \in \operatorname{dom}(-\Delta_\mu^D)$ and $-\Delta_\mu^D u=f$;
				\item[(ii)] $-\Delta u=f d \mu$ in the sense of distribution; that is, for any $v \in \mathcal{D}(\Omega)$,
				\begin{align*}
					\int_{\Omega} \langle\nabla u, \nabla v\rangle  \,d\nu=\int_{\Omega} v f\, d \mu.
				\end{align*}
			\end{enumerate}
			\item[(b)]	Assume that  $\partial\Omega= \emptyset$, {\rm(PIE)} holds, and $\mu$ satisfies {\rm(MPIE)}. For $u \in \operatorname{dom}(\mathcal{E}_E)$ and $f \in L^{2}(M, \mu)$, the following conditions are equivalent:
				\begin{enumerate}
					\item[(i)] $u \in \operatorname{dom}(-\Delta_\mu^E)$ and $-\Delta_\mu^E u=f$;
					\item[(ii)] $-\Delta u=f d \mu$ in the sense of distribution; that is, for any $v \in \mathcal{D}(M)$,
					\begin{align*}
						\int_{M} \langle\nabla u, \nabla v\rangle  \,d\nu=\int_{M} v f\, d \mu.
					\end{align*}
			\end{enumerate}
		\end{enumerate}
	\end{prop}
	
 For any $u \in \operatorname{dom}\left(\Delta_{\mu}^D\right)$, Proposition \ref{prop:9.1} implies that $\Delta u=\Delta_{\mu}^D u\, d \mu$ in the sense of distribution.

	\begin{thm}\label{thm:j2.5}
			Let $n$, $M$, $\Omega$, and $\mu$ be defined as in Proposition \ref{prop:9.1}.
				\begin{enumerate}
			\item[(a)] Assume that  $\partial\Omega\neq \emptyset$, {\rm(PID)} holds, and $\mu$ satisfies {\rm(MPID)}.
		Then for any $f \in L^{2}(\Omega, \mu)$, there exists a
		unique $u \in \operatorname{dom}\left(\Delta_{\mu}^D\right)$  such that $\Delta_{\mu}^D u=f.$ The operator
		$$(\Delta_{\mu}^D)^{-1}: L^{2}(\Omega, \mu) \rightarrow \operatorname{dom}(\Delta_{\mu}^D),\quad f\mapsto u $$
		is bounded and has norm at most $C$, the constant in \eqref{eq:PI} .
		\item[(b)] Assume that  $\partial\Omega= \emptyset$ , {\rm(PIE)} holds, and $\mu$ satisfies {\rm(MPIE)}.
			Then for any $f \in L^{2}(M, \mu)$, there exists a
			unique $u \in \operatorname{dom}\left(\Delta_{\mu}^E\right)$  such that $\Delta_{\mu}^E u=f.$ The operator
			$$(\Delta_{\mu}^E)^{-1}: L^{2}(M, \mu) \rightarrow \operatorname{dom}(\Delta_{\mu}^E),\quad f\mapsto u $$
			is bounded and has norm at most $C$, the constant in  \eqref{eq:PIN}.
	\end{enumerate}
	\end{thm}

	\subsection{Compact embedding theorem}
	\begin{lem}\label{lem:f4}
		Let $\mu$ be a positive finite Borel measure on a complete metric space $\mathfrak{X}$ (need not have compact support).
		\begin{enumerate}
			\item[(a)]  If $\mu$ is upper (resp. lower) $s$-regular for some $s>0$, then $\underline{\operatorname{dim}}_{\infty}(\mu) \geq s$ (resp. $\left.\overline{\operatorname{dim}}_{\infty}(\mu) \leq s\right)$.
			\item[(b)]  Conversely, if $\underline{\operatorname{dim}}_{\infty}(\mu) \geq s$ for some $s>0$, then $\mu$ is upper  $\alpha$-regular for any $0<\alpha<s$.
		\end{enumerate}
	\end{lem}
	The proof of Lemma \ref{lem:f4} is similar to that of \cite[Proposition 3.1]{Hu-Lau-Ngai_2006} and is omitted.

Let $M$ be a  Riemannian $n$-manifold and  $\{(U_i, \psi_i)\}_{i\in \Lambda}$ be an atlas of $M$, where $\psi_i: U_i\rightarrow \psi_i(U_i)\subseteq \R^n$ are diffeomorphisms. 
	Let $(\rho_{i})_{i\in \Lambda}$ be a smooth partition of unity subordinate to the cover $(U_i)_{i\in \Lambda}$. The triple $\{(U_i, \psi_i, \rho_i)\}_{i\in \Lambda}$ is called a {\em trivialization} of the manifold $M$. We say that a family $(U_i)_{i\in\Lambda}$ of
	open subsets of $M$ is a {\em uniformly locally finite covering of $M$} if $\cup_{i\in \Lambda}U_i=M$ and there exists an integer $N$ such that each point $p\in M$ has a
	neighborhood which intersects at most $N$ of the $U_i$.
	
The following lemma follows directly from \cite[Lemma 5]{Hebey_1996_1}.
	\begin{lem}(Hebey \cite[Lemma 5]{Hebey_1996_1})\label{lem:8.2}
		Let $M$ be a Riemannian $n$-manifold of bounded geometry, and let $0 <\epsilon'< {\rm inj}(M) /3.$ There exists a sequence $\{p_i\}_{i\in\Lambda}$ of points of $M$ such that for any  $\epsilon\in[\epsilon',{\rm inj}(M)/3)$, the following hold.
			\begin{enumerate}
		\item [(i)] There exists a constant $N>0$ such that each geodesic ball $B^M(p_i,\epsilon)$ intersects at most $N$ geodesic balls $B^M(p_j,\epsilon)$.
		\item [(ii)] The family $B^M(p_i,\epsilon)$ of  geodesic balls is a uniformly locally finite covering of $M$, and $N$ has an upper bound in terms of $n$, $\epsilon'$, and $\epsilon$.
		\item [(iii)] For any $i \neq j$, $B^M(p_i,\epsilon/2)\cap B^M(p_j,\epsilon/2)=\emptyset.$		 
	\end{enumerate}
	\end{lem}

	Let $M$ be a Riemannian $n$-manifold of bounded geometry. Let $\epsilon'\in(0,{\rm inj}(M)/3)$. Let $\{p_i\}_{i\in\Lambda}\subseteq M$ be a sequence such that for $\epsilon\in[\epsilon',{\rm inj}(M)/3)$, the conditions in Lemma \ref{lem:8.2} are satisfied. Then for each $i\in\Lambda$, the exponential map
	$\exp_{p_i}:B^{T_{p_i}M}(0,\epsilon)\rightarrow B^M(p_i,\epsilon)$  is a diffeomorphism. Every orthonormal basis $(b_i)$ for $T_{p_i}M$ determines a basis isomorphism  $E_{p_i}:T_{p_i}M  \rightarrow \R^n$, $i\in\Lambda$.  Moreover, for such $i$, $E_{p_i}$ is an isometry. For $i\in\Lambda$, let $B(0,\epsilon):=E_{p_i}(B^{T_{p_i}M}(0,\epsilon))$ and let $S_i: B(0,\epsilon) \rightarrow B(z_i,\epsilon)$ be a similitude. Define a family of normal coordinate maps
	\begin{align}\label{eq:8.1}
		\varphi_i:=S_i \circ E_{p_i} \circ \exp _{p_i}^{-1}: B^M({p_i},\epsilon) \rightarrow B(z_i,\epsilon)\subseteq\R^n, \,\, i\in \Lambda,
	\end{align}
	where $\varphi({p_i})=z_i$ and the sets $\varphi_i(B^M(p_i,\epsilon))$ are disjoint. $\{( B^M({p_i},\epsilon), \varphi_i)\}_{i\in \Lambda}$ is called a {\em uniformly locally finite geodesic atlas} of $M$.  Moreover, there exists a partition of unity $(\rho_i)_{i\in\Lambda}$ subordinate to $(B^M(p_i,\epsilon))_{i\in \Lambda}$ such that for all
	$k\in \N_0$, there exists a constant $C_k > 0$ such that 
	\begin{align}\label{eq:88}
		|D^\alpha(\rho_i\circ \varphi_i)|\leq C_k\quad \text{for all multi-indices}\,\, \alpha\,\, \text{satisfying}\,\, |\alpha|\leq k
	\end{align}
(see \cite[Proposition 7.2.1]{Triebel_1992}).
	$\{(B^M({p_i},\epsilon), \varphi_i, \rho_i)\}_{i\in \Lambda}$ is called a {\em uniformly locally finite geodesic trivialization} of the manifold $M$. For the following definition, we refer the reader to \cite{Grosse-Schneider_2013,Triebel_1992}.

	\begin{defi}\label{def:8.2}
		Let $M$ be a complete Riemannian $n$-manifold of bounded geometry with {a uniformly locally finite} geodesic trivialization
		$\{(B^M({p_i},\epsilon), \varphi_i, \rho_i)\}_{i\in \Lambda}$ as above. Define  
		 \begin{align}\label{eq:8.12222}
			W^{1,2}_0(M):=\Big\{u\in \mathcal{D}'(M):\|u\|_{W^{1,2}_0(M)}:=\sum_{i\in \Lambda}\|(\rho_i u)\circ \varphi_i^{-1}\|_{W^{1,2}_0(\R^n)}<\infty\Big\}.
		\end{align} 
	\end{defi}
	 Note that $W^{1,2}_0(M)$ in Definition \ref{def:8.2} generalizes the classical Sobolev space, which is defined in Section \ref{S:Pre*}. By \cite[Section 7.4.5]{Triebel_1992}, the norm in \eqref{eq:8.12222} is equivalent to the norm in \eqref{eq:H}.

	\begin{rema}\label{rema:8.2}
		Let $M$ be a Riemannian $n$-manifold of bounded geometry and let $\Omega\subseteq M$ be an open subset.  Let $\mu$ be a positive finite Borel measure on $M$ (need not have compact support) and assume that  $\supp(\mu)\subseteq \overline{\Omega}$. Let  $\{(B^M(p_i,\epsilon),\varphi_i)\}_{i\in \Lambda}$ be a uniformly locally finite geodesic atlas  on $\supp(\mu)$. Then   $\widetilde{\mu}:=\sum_{i\in \Lambda}\mu\circ\varphi_i^{-1}$ is a positive finite Borel measure on $\R^n$. In fact, by definition, $\varphi_i$ is a diffeomorphism and $\{\varphi_i(B^M(p_i,\epsilon))\}_{i\in \Lambda}$ is a  disjoint collection. Hence each normal coordinate map $\varphi_i$ induces a measure $\widetilde{\mu}_i:=\mu \circ \varphi_i^{-1}$ on $\varphi_i(B^M(p_i,\epsilon))\subseteq \R^n$. Therefore $\widetilde{\mu}$ is a Borel measure with  ${\rm supp}(\widetilde{\mu})\subseteq \overline{\bigcup_{i\in \Lambda} \varphi_i(B^M(p_i,\epsilon))}$. By  Lemma \ref{lem:8.2} and the definition of a uniformly locally finite atlas, we can show that $\widetilde{\mu}$ is a finite measure.
	\end{rema}

	In order to prove Lemma \ref{lem:9.3}, we need the following proposition,
	the proof of which is analogous to that used in
	Proposition \ref{prop:3.8}. We omit the details.

	\begin{prop}\label{prop:8.1}
	Let $M$ be a Riemannian $n$-manifold of bounded geometry and let $\Omega\subseteq M$ be an open subset.  Let $\mu$ be a positive finite Borel measure on $M$ (need not have compact support) and assume that  $\supp(\mu)\subseteq \overline{\Omega}$. Let $\{(U_{i},\varphi_{i})\}_{i\in \Lambda}$ be a uniformly locally finite atlas  on $\supp(\mu)$, where $U_i:=B^M(p_i,\epsilon)$, $\epsilon$, and $\varphi_{i}$ are defined as in the paragraph leading to \eqref{eq:8.1}.
	\begin{enumerate}
			\item[(a)] For each $i\in \Lambda$ and $x\in \varphi_i(U_i)\backslash \{\varphi_i(p_i)\}$,  there exist positive numbers $\beta$, $\rho$ and $\delta$  satisfying $\beta=4\delta/\rho$, $\rho+2\delta\le \epsilon$, and $0<\delta< \pi\rho/4$ such that $\widetilde{V}_{\delta}(x):=\widetilde{V}_{\delta,\rho,\beta,z_i,l_i}(x)$ satisfies
			\begin{align*}
				B(x,\delta)\subseteq \widetilde{V}_{\delta}(x)\subseteq B(x,\delta'),
			\end{align*}
			where  $\delta'$ can be taken to be $10\delta$.\\
			\item[(b)]  For $i\in \Lambda$ and $x\in \varphi_i(U_i)\backslash \{\varphi_i(p_i)\}$, define $V_{\delta}(\varphi^{-1}_i(x)):=\varphi^{-1}_i(\widetilde{V}_{\delta}(x))$. Let $\overline{\kappa}(M)$ be defined as in \eqref{cur}. Then for  $\rho\in(0,\min\{\epsilon,\pi/(4\sqrt{\overline{\kappa}(M)})\})$, $i\in \Lambda$, and $x\in \varphi_i(U_i)\backslash \{\varphi_i(p_i)\}$, there exist positive constants $c'$ and $c$ such that for $\delta$ satisfying $0<\delta<\min\{\pi/(4\sqrt{\overline{\kappa}(M)})-\rho/2,\pi\rho/4\}$, we have
			\begin{align*}
				B^M(\varphi^{-1}_i(x),c'\delta)\subseteq V_{\delta}(\varphi^{-1}(x))\subseteq B^M(\varphi^{-1}_i(x),c\delta).
			\end{align*}
		\end{enumerate}
	\end{prop}
	
	\begin{lem}\label{lem:9.3}
	Let $M$, $\Omega$, and $\mu$ be as in Proposition \ref{prop:8.1}. Let $\widetilde{\mu}$ be defined as in Remark \ref{rema:8.2}. Let $\{(U_{i},\varphi_{i})\}_{i\in \Lambda}$ be a uniformly locally finite geodesic atlas  on $\supp(\mu)$, where $U_i:=B^M(p_i,\epsilon)$ is a geodesic ball and $\varphi_i$ is defined as in  \eqref{eq:8.1}.  Let $n\geq 2$ and $2<q<\infty$. Then conditions  (a) and (b) below are equivalent.
		\begin{enumerate}
			\item[(a)]For $n>2$,
			\begin{align}
				&\lim _{\delta \rightarrow 0^{+}} \sup _{w \in M ; r \in(0, \delta)} r^{1-n/2} \mu(B^M(w,r))^{1/q}=0 \label{e:1}\\
				\text {\rm and } \quad &\lim _{|w|_M \rightarrow \infty} \sup _{r\in (0,1)} r^{1-n/2} \mu\left(B^M(w,r)\right)^{1 / q}=0; \label{e:2}
			\end{align}
		for $n=2$,
		\begin{align}
				&\lim _{\delta \rightarrow 0^{+}} \sup _{w \in M ; r \in(0, \delta)}|\ln r|^{1/2} \mu(B^M(w,r))^{1/q}=0 \label{e:3}\\ 
				\text {\rm and } \quad&\lim _{|w|_M \rightarrow \infty} \sup _{r\in (0,1)} |\ln r|^{1/2} \mu(B^M(w,r))^{1/q}=0. \label{e:4}
			\end{align}
			\item[(b)]For $n>2$,
			\begin{align}
				&\lim _{\delta \rightarrow 0^{+}} \sup _{x \in \R^n ; r \in(0, \delta)} r^{1-n/2}\, \widetilde{\mu}\left(B(x,r)\right)^{1 / q}=0 \label{e:5}\\
				\text {\rm and } \quad&\lim _{|x| \rightarrow \infty} \sup _{r\in (0,1)} r^{1-n/2} \widetilde{\mu} \left(B(x,r)\right)^{1 / q}=0; \label{e:6}
			\end{align}
		for $n=2$,
		\begin{align}
				&\lim _{\delta \rightarrow 0^{+}} \sup _{x \in \R^n ; r \in(0, \delta)}|\ln r|^{1 / 2}\, \widetilde{\mu}\left(B(x,r)\right)^{1 / q}=0\label{e:7}\\
				\text {\rm and } \quad&\lim _{|x| \rightarrow \infty} \sup _{r\in (0,1)} |\ln r|^{1/2} \widetilde{\mu}(B(x,r))^{1/q}=0\label{e:8}. 
			\end{align}
		\end{enumerate}
	\end{lem}
	
	\begin{proof}\, Assume that $(a)$ holds. We consider the case $n>2$. First, assuming that \eqref{e:1} holds, we can use the results of Proposition \ref{prop:8.1}, along with an  argument in Lemma \ref{lem:3.10}, to obtain \eqref{e:5}.
	
		Next assume that \eqref{e:2} holds. For any $\epsilon'>0$, there exists  $R_M>0$ such that for $|w|_M>R_M$,
		$
			\sup _{ r \in(0, 1)} r^{1-n/2} \mu(B^M(w,r))^{1/q}<\epsilon',
	$
		i.e., for $r \in(0, 1)$,
		\begin{align}\label{eq:n=2}
			r^{1-n/2} \mu(B^M(w,r))^{1/q}<\epsilon'.
		\end{align}
Let $\rho$ be a positive number. For  $0<\rho<\min\{\pi/(4\sqrt{\overline{\kappa}(M)}),\epsilon\}$, let $L:=\min\{\pi/(4\sqrt{\overline{\kappa}(M)})\\-\rho/2,\pi\rho/4\})$. For $i\in \Lambda$, let $x:=\varphi_i(w)$.	We consider the following two cases. Note that if $L>1$, then we ignore Case $2$.
		
		\noindent{\em Case 1. $r\in(0,L)$.} Let $q_0\in M$ be a fixed point and let 
		$$\Lambda_1:=\{i\in\Lambda: U_i\cap \big(M\backslash\overline{B^M(q_0,R_M+2L)}\big)\neq\emptyset\}.$$ \eqref{eq:n=2} implies that  for $w\in \cup_{i\in \Lambda_1}U_i$,
			\begin{align}\label{equ:4}
			r^{1-n/2} \mu(B^M(w,r))^{1/q}<\epsilon'.
		\end{align}
		\noindent{\em Case 1a. $w\in\cup_{i\in \Lambda_1} U_i\backslash \{p_i\}$.} Using Proposition \ref{prop:8.1} and \eqref{equ:4}, we have for $x\in \varphi_i(U_i)\backslash \{ \varphi_i(p_i)\}$, $i\in \Lambda_1$, and $r\in(0,L)$,
 \begin{align}\label{equ:8}
 r^{1-n/2} \widetilde{\mu}(B(x,r))^{1/q}<c^{n/2-1}\epsilon'.
\end{align}
	\noindent{\em Case 1b. $w\in \{p_i:i\in\Lambda_1\}$.} By the definitions of $\widetilde{\mu}$ and $\varphi_i$, for $x\in\{\varphi_i(p_i),i\in\Lambda_1\}$, we have
	$\mu(B^M(w,r))=\widetilde{\mu}(B(x,r))$. By \eqref{eq:n=2}, for $x\in\{\varphi_i(p_i),i\in\Lambda_1\}$ and $r\in(0,L)$, we have
	\begin{align}\label{equ:7}
		r^{1-n/2} \widetilde{\mu}(B(x,r))^{1/q}<\epsilon'.
	\end{align}
	Let $\underline{\epsilon}:=\min\{\epsilon',c^{n/2-1}\epsilon'\}$. Combining \eqref{equ:8} and \eqref{equ:7}, we obtain for $x\in\varphi_i(U_i)$ and $i\in\Lambda_1$,
	\begin{align}\label{equ:11}
		\sup_{r\in(0,L)}r^{1-n/2} \widetilde{\mu}(B(x,r))^{1/q}<\underline{\epsilon}.
	\end{align}
	 By the definition of $\varphi_i$, there exists some $R_E>0$ such that $\cup_{i\in (\Lambda\backslash \Lambda_1)}\varphi_i(U_i)\subseteq B(0,R_E)$.  Hence the above work shows that  for $|x|>R_E$, \eqref{equ:11} holds.
		Thus 
		\begin{align}\label{eq:s}
			\lim _{|x| \rightarrow \infty} \sup _{r\in (0,L)} r^{1-n/2} \widetilde{\mu}(B(x,r))^{1/q}=0. 
		\end{align}
		\noindent{\em Case 2. $r\in[L,1)$.} By the fact that $\widetilde{\mu}$ is a finite measure on $\R^n$,
		 for any $\epsilon'>0$, there exists $x_0\in \R^n$ such that $\widetilde{\mu}(\R^n\backslash B(0,|x_0|-1))<\epsilon'$. Hence for $x\in \R^n$ satisfying $|x|>|x_0|$, we have $\widetilde{\mu}(B(x,r))<\epsilon'$, and thus 
		$r^{1-n/2} \widetilde{\mu}(B(x,r))^{1/q}<L^{1-n/2}\epsilon'^{(1/q)}.$
		It follows that
		\begin{align}\label{eq:ns}
			\lim _{|x| \rightarrow \infty} \sup _{r\in [L,1)} r^{1-n/2} \widetilde{\mu}(B(x,r))^{1/q}=0. 
		\end{align}
		Combining \eqref{eq:s} and \eqref{eq:ns} shows that \eqref{e:6} holds.
		
		For the case $n=2$, using a similar proof as above, we can show that \eqref{e:3} and \eqref{e:4} hold. This proves $(b)$. The proof of the converse is similar. We omit the details.
	\end{proof}

	\begin{proof}[Proof of Theorem \ref{thm:8.11}]
		\,Assume \eqref{eq:m1}--\eqref{eq:m4} hold.  We will prove that  the until ball $B$ is relatively compact in $L^q(M,\mu)$. Let  $\widetilde{B}:=\{u \in C_{c}^{\infty}(\R^n):\|u\|_{W^{1,2}_0(\R^n)} \leq 1\}$. Lemma  \ref{lem:9.3} implies that \eqref{e:5}--\eqref{e:8} hold.
		Let $(u_j)_{j\in\Z_+}$ be a bounded sequence in $B$. We only prove the case that $\cup_{j\in\Z_+}\supp(u_j)$ is unbounded; the case that $\cup_{j\in\Z_+}\supp(u_j)$ is bounded can be handled as in the proof of Theorem \ref{thm:3.11}. 	Let $\{(U_{i},\varphi_{i},\rho_{i})\}_{i\in \Lambda}$ be a uniformly locally finite geodesic trivialization on $\supp(\mu)$,  where $U_i:=B^M(p_i,\epsilon)$ is a geodesic ball, $\varphi_i$ is defined as  in  \eqref{eq:8.1}, and $(\rho_i)_{i\in\Lambda}$ is a $C^{\infty}$ partition of unity subordinate to  $\{U_i\}_{i\in \Lambda}$ and satisfying \eqref{eq:88}. For $j\in\Z_+$, let $\overline{F}_j:=\supp(u_j)$ and let $I_j:=\{i\in\Lambda: U_i\cap\overline{F}_j \neq\emptyset\}$. Then 
		$$u_j=\sum_{i\in \Lambda}\rho_iu_j=\sum_{i\in I_j}\rho_iu_j, \quad \supp(\rho_iu_j)\subseteq U_i,\quad\text{and}\quad\rho_iu_j\in C_c^{\infty}(U_i)\quad\text{for each}\,\,i\in I_j.$$
		For $i\in I_j$ and $j\in\Z_+$, the properties of the normal coordinate maps $\varphi_i$ imply that  $\widetilde{U}_i:=B(z_i,\epsilon)$ are disjoint and $\rho_iu_j\circ\varphi_i^{-1}\in C_c^{\infty}(\widetilde{U}_i)$.
		For $j\in\Z_+$, let
		$$\widetilde{u}_j:=\sum_{i\in I_j}\rho_iu_j\circ\varphi_i^{-1}.$$
		Then $\widetilde{u}_j\in C_c^{\infty}(\R^n)$.
	By Definition \ref{def:8.2}, we have $$\|u_j\|_{W^{1,2}_0(M)}:=\sum_{i\in \Lambda}\|\rho_iu_j\circ\varphi_i^{-1}\|_{W^{1,2}_0(\R^n)}
		=\sum_{i\in I_j}\|\rho_iu_j\circ\varphi_i^{-1}\|_{W^{1,2}_0(\R^n)}.$$
		 Hence,
		$$\|\widetilde{u}_j\|_{W^{1,2}_0(\R^n)}=\sum_{i\in I_j}\|\rho_iu_j\circ\varphi_i^{-1}\|_{W^{1,2}_0(\R^n)}=\|u_j\|_{W^{1,2}_0(M)}\leq 1.$$
		Thus $(\widetilde{u}_j)_{j\in\Z_+}$ is a bounded sequence in $\widetilde{B}$.
		Theorem \ref{thm:maz} implies that $\widetilde{B}$ is relatively compact in $L^q(\R^n, \widetilde{\mu})$. Hence $(\widetilde{u}_j)$ has a convergent subsequence $(\widetilde{u}_{j_k})$ in $L^q(\R^n, \widetilde{\mu})$. Therefore there exists $\widetilde{u}\in L^q(\R^n, \widetilde{\mu})$ such that 
		\begin{align}\label{eq:cong}
			\lim_{k\rightarrow \infty}\|\widetilde{u}_{j_k}-\widetilde{u}\|_{L^q(\R^n, \widetilde{\mu})}= 0.
		\end{align}
		Next, we consider the following two cases.
		
		\noindent{\em Case 1. $\widetilde{u}=0$.} We have
		\begin{align*}
			\|u_{j_k}\|_{L^q(M, \mu)} 
			&\leq\sum_{i\in I_k}\|\rho_iu_{j_k}\|_{L^q(M, \mu)}\qquad\qquad\,\,\,\,(\text{by Minkowski inequality} )\\
			&=\sum_{i\in I_k}\|\rho_iu_{j_k}\circ\varphi_i^{-1}\|_{L^q(\R^n, \widetilde{\mu})}\qquad(\text{by Claim 2 of Theorem \ref{thm:3.11}})\\
			&=\|\widetilde{u}_{j_k}\|_{L^q(\R^n, \widetilde{\mu})}\rightarrow 0 \quad\qquad\qquad(\text{by \eqref{eq:cong}})
		\end{align*}
		as $k\rightarrow \infty$. Hence $B$ is relatively compact in $L^q(M, \mu)$.
		
		\noindent{\em Case 2. $\widetilde{u}\neq 0$.} By \eqref{eq:cong}, for any $\epsilon_1>0$, there exists $l>0$ such that for $m,k>l$,
		\begin{align}
			\|\widetilde{u}_{j_k}-\widetilde{u}_{j_m}\|_{L^q(\R^n, \widetilde{\mu})}< \epsilon_1.
		\end{align}
		We let
		\begin{align*}
		\widetilde{u}_{j_k}:=\sum_{i\in I_k}\rho_{i}u_{j_k}\circ\varphi_i^{-1}\quad\text{and}\quad
		\widetilde{u}_{j_m}:=\sum_{i\in I_m}\rho_{i}u_{j_m}\circ\varphi_i^{-1}.
	\end{align*}
		Notice that $I_\alpha:=I_k\cap I_m\neq\emptyset$. Hence,
		\begin{align*}
			\epsilon_1>\|\widetilde{u}_{j_k}-\widetilde{u}_{j_m}\|_{L^q(\R^n, \widetilde{\mu})}=&\sum_{i\in I_\alpha}\|\rho_iu_{j_k}\circ\varphi_i^{-1}-\rho_iu_{j_m}\circ\varphi_i^{-1}\|_{L^q(\R^n, \widetilde{\mu})}\\
			&+\sum_{i\in I_k\backslash I_\alpha}\|\rho_iu_{j_k}\circ\varphi_i^{-1}\|_{L^q(\R^n, \widetilde{\mu})}\\&+\sum_{i\in I_m\backslash I_\alpha}\|\rho_iu_{j_m}\circ\varphi_i^{-1}\|_{L^q(\R^n, \widetilde{\mu})},
		\end{align*}
		and thus
		\begin{align*}
			&\sum_{i\in I_k\backslash I_\alpha}\|\rho_iu_{j_k}\circ\varphi_i^{-1}\|_{L^q(\R^n, \widetilde{\mu})}<\epsilon_1,\\
			&\sum_{i\in I_m\backslash I_\alpha}\|\rho_iu_{j_m}\circ\varphi_i^{-1}\|_{L^q(\R^n, \widetilde{\mu})}<\epsilon_1,\quad\text{and}\\
				&\sum_{i\in I_\alpha}\|\rho_iu_{j_k}\circ\varphi_i^{-1}-\rho_iu_{j_m}\circ\varphi_i^{-1}\|_{L^q(\R^n, \widetilde{\mu})}<\epsilon_1.
		\end{align*}
		Therefore,
		\begin{align*}
			&	\sum_{i\in I_k\backslash I_\alpha}\|\rho_iu_{j_k}\|_{L^q(M, \mu)}=\sum_{i\in I_k\backslash I_\alpha}\|\rho_iu_{j_k}\circ\varphi_i^{-1}\|_{L^q(\R^n, \widetilde{\mu})}<\epsilon_1,\\
			&			\sum_{i\in I_m\backslash I_\alpha}\|\rho_iu_{j_m}\|_{L^q(M, \mu)}=	\sum_{i\in I_m\backslash I_\alpha}\|\rho_iu_{j_m}\circ\varphi_i^{-1}\|_{L^q(\R^n, \widetilde{\mu})}<\epsilon_1,\\
			&	\sum_{i\in I_\alpha}\|\rho_iu_{j_k}-\rho_iu_{j_m}\|_{L^q(M, \mu)}=\sum_{i\in I_\alpha}\|(\rho_iu_{j_k}-\rho_iu_{j_m})\circ\varphi_i^{-1}\|_{L^q(\R^n, \widetilde{\mu})}<\epsilon_1.
		\end{align*}
		Hence for any $\epsilon_1>0$, there exists $l>0$ such that for all $k,m>l,$
		\begin{align*}
			\|u_{j_k}-u_{j_m}\|_{L^q(M, \mu)}=&\|\sum_{i\in I_k}\rho_iu_{j_k}-	\sum_{i\in I_m}\rho_iu_{j_m}\|_{L^q(M, \mu)}\\
			=&\|\sum_{i\in I_\alpha}(\rho_iu_{j_k}-\rho_iu_{j_m})+\sum_{i\in I_k\backslash I_\alpha}\rho_iu_{j_k}-	\sum_{i\in I_m\backslash I_\alpha}\rho_iu_{j_m}\|_{L^q(M, \mu)}\\
			\leq&\sum_{i\in I_\alpha}\|\rho_iu_{j_k}-\rho_iu_{j_m}\|_{L^q(M, \mu)}+\sum_{i\in I_k\backslash I_\alpha}\|\rho_iu_{j_k}\|_{L^q(M, \mu)}\\&+	\sum_{i\in I_m\backslash I_\alpha}\|\rho_iu_{j_m}\|_{L^q(M, \mu)}
			<3\epsilon_1.
		\end{align*}
		Hence $(u_{j_k})$ is a Cauchy sequence in $L^q(M, \mu)$. Since $L^q(M, \mu)$ is complete, $(u_{j_k})$ converges in $L^q(M, \mu)$. Hence $B$ is relatively compact in $L^q(M, \mu)$. 
		
		Conversely, assume that the unit ball $B$ is relatively compact in $L^q(M,\mu)$. We will show that \eqref{eq:m1}--\eqref{eq:m4} hold. Let  $B_1:=\{u \in C_{c}^{\infty}(M):\|u\|_{W^{1,2}_0(M)} \leq 2\}$. We divide the proof into  four parts.
		
		\noindent {\em Part I. Use the properties of a bounded sequence in $\widetilde{B}$ to find a corresponding bounded sequence in $B_1$}.
	
		Let $(\widetilde{f}_j)_{j\in\Z_+}$ be a bounded sequence in $\widetilde{B}$. We only prove the case that   $\cup_{j\in\Z_+}\supp(\widetilde{f}_j)$ is unbounded; the case that $\cup_{j\in\Z_+}\supp(f_j)$ is bounded can be handled similarly. 	Let $\{(U_{i},\varphi_{i})\}_{i\in \Lambda}$ be a uniformly locally finite geodesic atlas on $\supp(\mu)$,  where $U_i:=B^M(p_i,\epsilon)$ is a geodesic ball, $\varphi_i$ is defined as  in  \eqref{eq:8.1}. For $j\in \Z_+$, let $\overline{\widetilde{F}}_j:=\supp(\widetilde{f}_j)$. For any $i\in \Lambda$, let  $B(z_i,\epsilon):=\varphi_i(U_i)$. Let $\Lambda^1:=\{i\in\Lambda:B(z_i,\epsilon)\cap (\cup_{j\in \Z_+} \overline{\widetilde{F}}_j)\neq\emptyset \}$ and let $\Gamma_j:=\{i\in\Lambda^1:B(z_i,\epsilon)\cap \overline{\widetilde{F}}_j\neq\emptyset\}$. Note that for each $j\in \Z_+$, $\#\Gamma_j<\infty$. Hence there exists an open ball $\widetilde{B}_j$ such that $\overline{\widetilde{F}}_j\subseteq \widetilde{B}_j$ and $\overline{\cup_{i\in\Gamma_j} B(z_i,\epsilon)}\subseteq \widetilde{B}_j$. By Remark \ref{rema:8.2}, $\widetilde{\mu}:=\sum_{i\in \Lambda}\mu\circ\varphi_i^{-1}$ is a positive finite Borel measure on $\R^n$. For each $j\in \Z_+$, if $\Gamma_j= \emptyset$, i.e., $\Lambda^1=\emptyset$, then $\|\widetilde{f}_j\|_{L^q(\R^n,\widetilde{\mu})}=0$. Hence there exists a subsequence $(\widetilde{f}_{j_k})\subseteq (\widetilde{f}_j)$ converging to $0$ in $L^q(\R^n,\widetilde{\mu})$. The proof is complete.  If $\Gamma_j\neq \emptyset$, then for $i\in\Gamma_j$, we reduce the radius of $B(z_i,\epsilon)$.  For each $\widetilde{\epsilon}_j>0$, there exists $\delta_j\in(0,\epsilon)$ such that for $\delta\in(0,\delta_j)$,
		\begin{align}
			&\int_{\cup_{i\in \Gamma_j}B(z_i,\epsilon)\backslash \overline{B(z_i,\epsilon-\delta)}}|\widetilde{f}_j|^2\,dx<\frac{\widetilde{\epsilon}_j}{2},\label{eq:l5}\\
			&\int_{\cup_{i\in \Gamma_j}B(z_i,\epsilon)\backslash \overline{B(z_i,\epsilon-\delta)}}|\nabla \widetilde{f}_j|^2\,dx<\frac{\widetilde{\epsilon}_j}{2},\label{eq:l6}\\
			&\int_{\cup_{i\in \Gamma_j}B(z_i,\epsilon)\backslash \overline{B(z_i,\epsilon-\delta)}}|\widetilde{f}_j|^q\,d\widetilde{\mu}<\frac{\widetilde{\epsilon}_j}{2^j}\label{eq:l7}.
		\end{align}
Let  $\widetilde{U}_{N_j}:=\widetilde{B}_j\backslash  \overline{\cup_{i\in \Gamma_j}B(z_i,\epsilon-\delta_j)}$ and let $\widetilde{W}:=\supp(\widetilde{\mu})\backslash \big(\overline{\cup_{i\in\Lambda^1} B(z_i,\epsilon)\cup (\cup_{j\in\Z_+}\overline{\widetilde{F}}_j)}\big)$. Then $\{B(z_i,\epsilon)\}_{i\in \Lambda^1}\cup\{\widetilde{U}_{N_j}\}_{j\in \Z_+}\cup \widetilde{W}$  is an open cover of  $\supp(\widetilde{\mu})$.  We only prove the case when $\widetilde{U}_{N_j}$ is a nonempty collection and  $\widetilde{W} $ is nonempty; the other cases can be handled similarly. Let $(\widetilde{\zeta}_{\alpha})$ be a $C^{\infty}$ partition of unity subordinate to $\{B(z_i,\epsilon)\}_{i\in \Lambda^1}\cup\{\widetilde{U}_{N_j}\}_{j\in \Z_+}\cup \widetilde{W}$. Then
 \begin{align}\label{eq:s1}
		\widetilde{f}_j=\sum_{\alpha}\widetilde{\zeta}_{\alpha}\widetilde{f}_j=\sum_{i\in \Gamma_j}\widetilde{\zeta}_{i}\widetilde{f}_j+\widetilde{\zeta}_{N_j}\widetilde{f}_j,\,\, \supp(\widetilde{\zeta}_{i}\widetilde{f}_j)\subseteq B(z_i,\epsilon),\text{and}\,\,\supp(\widetilde{\zeta}_{N_j}\widetilde{f}_j)\subseteq \widetilde{U}_{N_j}.
		\end{align}
 Moreover, $\widetilde{\zeta}_{i}\widetilde{f}_j\in C_c^{\infty}(\R^n)$ and  $\widetilde{\zeta}_{N_j}\widetilde{f}_j\in C_c^{\infty}(\R^n)$. 
		 For any $i\in \Gamma_j$ and $j\in\Z_+$, by the definition of the normal coordinate maps $\varphi_i$,
		$$\widetilde{\zeta}_{i}\widetilde{f}_j\circ \varphi_i\in C_c^{\infty}(M)\qquad\text{and}\qquad\supp(\widetilde{\zeta}_{i}\widetilde{f}_j\circ \varphi_i)\subseteq B^M(p_i,\epsilon).$$
		Let $f_j:=\sum_{i\in \Gamma_j}\widetilde{\zeta}_{i}\widetilde{f}_j\circ \varphi_i.$
		Then $f_j\in  C_c^{\infty}(M)$. 
		Therefore,
		\begin{align*}
			\|f_j\|_{W^{1,2}_0(M)}\leq&\sum_{i\in \Gamma_j}\|\widetilde{\zeta}_{i}\widetilde{f}_j\circ \varphi_i\|_{W^{1,2}_0(M)} \quad(\text{by Minkowski inequality})\\
			=&\sum_{i\in \Gamma_j}\|\widetilde{\zeta}_{i}\widetilde{f}_j\|_{W^{1,2}_0(\R^n)}\qquad\quad(\text{by Claim 1 of Theorem \ref{thm:3.11}})\\
			\leq& 1+\widetilde{\epsilon}_j.\qquad\qquad\qquad\qquad\,(\text{by \eqref{eq:l5} and \eqref{eq:l6}})
		\end{align*}
		Thus for all $j\in \Z_+$, $\|f_j\|_{W_0^{1,2}(M)}\leq 1+\widetilde{\epsilon}_j\leq 2$. Hence 
		 $(f_j)_{j\in\Z_+}$ is a bounded sequence in $B_1$.
		
		We know that each geodesic ball $B^M(p_i,\epsilon)$  in $\supp(\mu)$ can intersect at most $N(\in \Z_+)$ geodesic balls $B^M(p_j,\epsilon)$, where $i\neq j$. Let $\mathcal B:=\{B^M(p_i,\epsilon)\}_{i\in \Lambda}$. This means that each point in $\mathcal B$ can be covered by at most $N+1$ of the geodesic balls $B^M(p_j,\epsilon)$. Let $$\mathcal{A}_k:=\big\{B^M(p_i,\epsilon):B^M(p_i,\epsilon)\,\, \text{intersects $k-1$ geodesic balls in}\,\, \mathcal B\big\},\,\,k=1,\ldots, N+1.$$

		\begin{figure}[h!]
			\centerline{
				\begin{tikzpicture}[scale=0.7]
					\draw[magenta, line width=1.3pt](0,0)arc(0:360:1);    
					\draw[magenta] (-1.5,-0.2) node[right]{$B_{7,1}^1$};
					\draw[cyan, line width=1.3pt](-1.5,-0.2)arc(0:360:1); 
					\draw[cyan] (-3.2,-0.6) node[right]{$B_{3,1}^1$};
					\draw[green!70!black, line width=1.3pt](1.5,1.2)arc(0:360:1);
					\draw[green!70!black] (-0.05,1.23) node[right]{$B_{2,1}^2$}; 
					\draw[blue, line width=1.3pt](3,0.4)arc(0:360:1); 
					\draw[blue] (1.6,0.25) node[right]{$B_{4,1}^2$};
					\draw[green!70!black, line width=1.3pt](-3,0.8)arc(0:360:1); 
					\draw[green!70!black] (-4.9,1) node[right]{$B_{2,1}^1$};
					\draw[blue, line width=1.3pt](-0.4,2.8)arc(0:360:1);
					\draw[blue] (-2,3.1) node[right]{$B_{4,1}^1$};
					\draw[red, line width=1.3pt](-1.4,1.2)arc(0:360:1); 
					\draw[red] (-3.05,1.3) node[right]{$B_{1,1}$};	
					\draw[brown, line width=1.3pt](-1.8,3.)arc(0:360:1); 
					\draw[brown] (-3.7,3.1) node[right]{$B_{5,1}^1$}; 
					\draw[red, line width=1.3pt](1.5,2.7)arc(0:360:1); 
					\draw[red] (-0.2,2.52) node[right]{$B_{1,2}$};
					\draw[green!70!black, line width=1.3pt](0.9,3.8)arc(0:360:1); 
					\draw[green!70!black] (-0.9,4.2) node[right]{$B_{2,2}^1$};
					\draw[black, line width=1.3pt](0.2,1.7)arc(0:360:1); 
					\draw[black] (-1.66,1.44) node[right]{$B_{6,1}^1$}; 
					\draw[cyan, line width=1.3pt](1.8,-0.2)arc(0:360:1); 
					\draw[cyan] (-0.,-0.3) node[right]{$B_{3,1}^2$};
					\draw[cyan, line width=1.3pt](2.4,4.05)arc(0:360:1); 
					\draw[cyan] (1,4.2) node[right]{$B_{3,2}^1$};
					\draw[brown, line width=1.3pt](3,2.1)arc(0:360:1); 	
					\draw[brown] (1.5,2.2) node[right]{$B_{5,2}^1$};	
				\end{tikzpicture}
			}
			\caption{A possible arrangement of some elements of the sets $\mathcal L_k$, $k=1,\ldots,7$. Note that in this example a geodesic ball can intersect  at most six other geodesic balls. Step II.1. We choose a geodesic ball that intersects with six geodesic balls and name it $B_{1,1}$. Then, we choose another geodesic ball that does not intersect $B_{1,1}$ but intersects the other six geodesic balls, and denote it by $B_{1,2}$. Step II.2. From the geodesic balls intersecting $B_{1,1}$, we choose one and name it $B_{2,1}^1$. Then, from the geodesic balls intersecting $B_{1,2}$ but not  $B_{1,1}$, we choose one and name it $B_{2,2}^1$.
			Step II.2a. We choose a geodesic ball that intersects five other geodesic balls and name it $B_{2,1}^2$. 
		 Step II.3. From the geodesic balls intersecting  $B_{1,1}$, we choose one and name it $B_{3,1}^1$. Then, from the geodesic balls intersecting  $B_{1,2}$ but not  $B_{1,1}$, we choose one and name it $B_{3,2}^1$. Again, from the geodesic balls intersecting  $B_{2,1}^2$ but not  $B_{1,1}\cup B_{1,2}$, we choose another and name it $B_{3,1}^2$. Since none of the remaining geodesic balls intersect four other geodesic balls, we stop this step.
		Step II.4. From the geodesic balls intersecting $B_{1,1}$, we choose one and denote it by $B_{4,1}^1$. Since $B_{4,1}^1\cap B_{1,2}\neq\emptyset$, there is no need to choose another geodesic ball from those intersecting  $B_{1,2}$. Then, from the geodesic balls intersecting with $B_{2,1}^2$ but not $B_{1,1}\cup B_{1,2}$, we choose one and denote it by $B_{4,1}^2$. Since none of  remaining geodesic balls  intersect three other geodesic balls, we stop this step.
		Step II.5. From the geodesic balls intersecting $B_{1,1}$, we choose one and denote it by $B_{5,1}^1$. Then, from the geodesic balls intersecting  $B_{1,2}\backslash B_{1,1}$, we choose one and name it $B_{5,2}^1$. Since geodesic balls $B_{5,2}^1\cap B_{2,1}^2\neq \emptyset$, there is no need to choose another geodesic ball from those intersecting  $B_{2,1}^2$. Since none of the remaining geodesic balls intersect with two other geodesic balls, we stop this step.
		Step II.6. From the geodesic balls intersecting  $B_{1,1}$, we choose one and name it $B_{6,1}^1$. Since $B_{5,2}^1$, $B_{1,2}$, and $B_{2,1}^2$ all intersect each other, there is no need to choose another geodesic ball from those intersecting $B_{1,2}$ and $B_{2,1}^2$. Since none of the remaining geodesic balls  intersect  only one geodesic ball, we stop this step.
		Step II.7. Only one geodesic ball remains, which we call $B_{7,1}^1$.
				}\label{fig:j0}
		\end{figure}
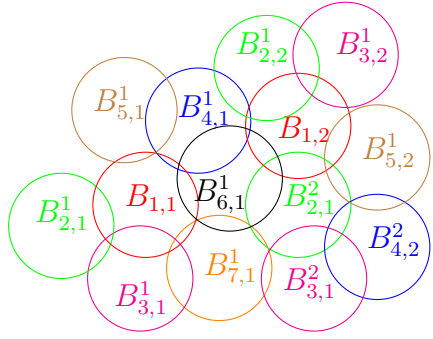
		
		\noindent {\em Part II. Partition $\mathcal B$ into $N+1$ disjoint layers, $\mathcal L_k$, $k=1,\dots, N+1$, each consisting of mutually disjoint geodesic balls from $\mathcal B$. } 
				
				We divide Part II into $N+1$ steps.
		
		\noindent{\em Step II.1.} If $\mathcal A_{N+1}\neq\emptyset$, 
			  we choose a geodesic ball $B_{1,1}\in\mathcal A_{N+1}$. Assume that $B_{1,1},\dots, B_{1,t}$, where $t>1$, have been chosen. If possible, we choose $B_{1,t+1}\in\mathcal A_{N+1}$ so that $B_{1,t+1}\cap(\cup_{j=1}^{t}B_{1,j})=\emptyset$; otherwise, we stop and proceed to Step 2. Let $$\mathcal{T}_1:=\big\{t=1: B_{1,t}\in\mathcal A_{N+1}\big\}\bigcup\big\{t>1: B_{1,t}\in\mathcal A_{N+1}\,\, \text{and}\,\, B_{1,t}\cap(\cup_{j=1}^{t-1}B_{1,j})=\emptyset \big\}.$$ Then 
		$$
		\mathcal L_1:=\big\{B_{1,t}:t\in\mathcal{T}_1\big\}
		$$
		is a finite or infinite family of mutually disjoint geodesic balls in $\mathcal A_{N+1}$ (see Figure \ref{fig:j0}). If $\mathcal A_{N+1}=\emptyset$, we let $s:=\max\{k\in\{1,\ldots,N\}: \mathcal A_{k}\neq\emptyset\}$. By using the above method, we choose a finite or infinite family $\mathcal L_1$ of mutually disjoint geodesic balls from $\mathcal A_{s}$.

		\noindent{\em Step II.2.} If $\mathcal B\backslash \mathcal L_1=\emptyset$, we terminate the process.  If $\mathcal B\backslash \mathcal L_1\neq\emptyset$,  we choose a geodesic ball $B_{2,1}^{1}\in \mathcal B\backslash \mathcal L_1$ such that $B_{2,1}^{1}\cap B_{1,1}\neq\emptyset$. 
		
		\noindent{\em Step II.2a.} Let 
		$$\mathcal I_2^1:=\big\{t\in\mathcal{T}_1:  \exists \underline{B}\in\mathcal B\setminus\mathcal L_1 \,\,\text{s.t.}\,\, \underline{B}\cap B_{1,t}\neq\emptyset \,\,\text{and}\,\, \underline{B}\cap B_{2,1}^1=\emptyset\big\}.
		$$
		If $\mathcal I_2^1=\emptyset$, we stop and proceed to Step 2b. If $\mathcal I_2^1\neq\emptyset$, we let $t_1:=\min\mathcal I_2^1$ and 
		choose $B_{2,2}^1\in\mathcal B\setminus\mathcal L_1$ such that  $B_{2,2}^1\cap B_{1,t_1}\neq\emptyset$ and $B_{2,2}^1\cap B_{2,1}^1=\emptyset$. In the same way,  we let 
		$$\mathcal I_2^2:=\big\{t\in\mathcal{T}_1:t>t_1, \exists \underline{B}\in\mathcal B\setminus\mathcal L_1 \,\,\text{s.t.}\,\, \underline{B}\cap B_{1,t}\neq\emptyset\, \,\text{and}\,\, \underline{B}\cap (\cup_{j=1}^2B_{2,j}^1)=\emptyset\big\}.
		$$
		If $\mathcal I_2^2=\emptyset$, we stop and proceed to Step 2b. If $\mathcal I_2^2\neq\emptyset$, we let 
		$t_2:=\min \mathcal I_2^2$ and 
		choose $B_{2,3}^1\in\mathcal B\setminus\mathcal L_1$ such that  $B_{2,3}^1\cap B_{1,t_2}\neq\emptyset$ and $B_{2,3}^1\cap (\cup_{j=1}^2B_{2,j}^1)=\emptyset$.
		Continuing this process, for $i\geq3$, let $t_{i-1}:=\min \mathcal I_2^{i-1}$ and 
		$$\mathcal I_2^{i}:=\{t\in\mathcal{T}_1:t>t_{i-1}, \exists \underline{B}\in\mathcal B\setminus\mathcal L_1 \,\,\text{s.t.}\,\, \underline{B}\cap B_{1,t}\neq\emptyset\, \,\text{and}\,\, \underline{B}\cap (\cup_{j=1}^iB_{2,j}^1)=\emptyset\}.$$ Let 
		\begin{align*}
		\Lambda_2^1:=\big\{i=1:B_{2,i}^{1}\in \mathcal B\backslash \mathcal L_1\,\,\text{and}\,\,B_{2,i}^{1}\cap B_{1,1}\neq\emptyset\big\}\bigcup\big\{i>1: B_{2,i}^1\in\mathcal B\setminus\mathcal L_1,\\ B_{2,i}^1\cap B_{1,t_{i-1}}\neq\emptyset \,\,\text{and}\,\,B_{2,i}^1\cap (\cup_{l=1}^{i-1}B_{2,l}^1)=\emptyset&\big\}.
		\end{align*}
		Then we can choose a finite or infinite family $
		\mathcal L_2^1:=\big\{B_{2,i}^1:i\in\Lambda_2^1\big\}.
		$ of mutually disjoint geodesic balls  from $\mathcal B\backslash \mathcal L_1$.
		
		\noindent{\em Step II.2b.} If possible, we choose $B_{2,i}^{2}\in (\mathcal A_{N+1}\cup \mathcal A_N)\backslash \mathcal L_1$ such that $B_{2,i}^{2}\cap (\cup_{j\in \Lambda_2^1}B_{2,j}^{1}\cup(\cup_{l=0}^{i-1}B_{2,l}^{2}))=\emptyset$, where $B_{2,0}^{2}:=\emptyset$, and $i\geq 1$; otherwise, we stop.  Let $$\Lambda_2^2:=\{i\geq 1:B_{2,i}^{2}\in (\mathcal A_{N+1}\cup \mathcal A_N)\backslash \mathcal L_1\,\,\text{and}\,\,B_{2,i}^{2}\cap (\cup_{j\in \Lambda_2^1}B_{2,j}^{1}\cup(\cup_{l=0}^{i-1}B_{2,l}^{2}))=\emptyset\}.$$ Then
		\begin{align*}
			\mathcal L_2^2:=\big\{B_{2,i}^2:i\in \Lambda_2^2\big\}
		\end{align*}
		is a finite or infinite family of mutually disjoint geodesic balls in $\mathcal B\backslash \mathcal L_1$. We rename the elements in $\mathcal L_2^1\cup \mathcal L_2^2$, by introducing an index set $\mathcal T_2$, as $\{B_{2,t}: t\in\mathcal T_2\}$.
		Hence,
		\begin{align*}
			\mathcal L_2:=\mathcal L_2^1\cup \mathcal L_2^2=\big\{B_{2,t}:t\in\mathcal{T}_2\big\}
		\end{align*}
		is a finite or infinite family of mutually disjoint geodesic balls in $\mathcal B\backslash \mathcal L_1$ (see Figure \ref{fig:j0}).
		
		\noindent{\em Step II.k, $k=3,4,\ldots,N$.} Continuing in this way,  we terminate the process if $\mathcal B\backslash (\cup_{i=1}^{k-1}\mathcal L_i)=\emptyset$, $k=3,\ldots,N$.  If $\mathcal B\backslash (\cup_{i=1}^{k-1}\mathcal L_i)\neq\emptyset$,  we choose $\mathcal L_k$, $k=3,\ldots,N$.

	\noindent{\em Step II.N+1.}		If $\mathcal B\backslash (\cup_{i=1}^{N}\mathcal L_i)=\emptyset$, we terminate the process. If $\mathcal B\backslash (\cup_{i=1}^{N}\mathcal L_i)\neq\emptyset$, we let $\mathcal{T}_{N+1}:=\{t\geq 1:B_{N+1,t}\in \mathcal B\backslash (\cup_{k=1}^{N}\mathcal L_k)\}$. Then
		\begin{align*}
			\mathcal L_{N+1}:=\big\{B_{N+1,t}:t\in\mathcal{T}_{N+1}\big\}
		\end{align*}
	is a finite or infinite family of mutually disjoint geodesic balls in $\mathcal B\backslash (\cup_{k=1}^{N}\mathcal L_k)$ (see Figure \ref{fig:j0}).

\noindent {\em Part III. Extract a convergent subsequence of $(f_j)_{j\in \Z_+}$ in $L^q(M,\mu)$}.
	
	We divide this part into $T$ steps.
			
		\noindent{\em Step III.1.} For $j\in \Z_+$, we know that $f_j:=\sum_{i\in \Gamma_j}\widetilde{\zeta}_{i}\widetilde{f}_j\circ \varphi_i.$ Let 
			\begin{align*}
			L_1:=\big\{l\in\{1,\ldots,N+1\}:\exists\,\,\text{infinitely many}\,\, \widetilde{\zeta}_{i}\widetilde{f}_j\circ \varphi_i,\,\text{where}\,\, i\in \Gamma_j \,\,\text{and}\,\,j\in \Z_+,&\\ 
			\text{s.t.}\,\,\supp(\widetilde{\zeta}_{i}\widetilde{f}_j\circ \varphi_i)\subseteq \mathcal \cup_{t\in \mathcal{T}_l}B_{l,t}&\big\},
		\end{align*}
	and let $\#L_1=T\in (0,N+1]$, where $\#A$ denotes the cardinality of a set $A$. Let $ l_1:=\min L_1$ and let
		\begin{align*}
		P_{l_1}:=\big\{t\in\mathcal{T}_{l_1}:\exists\,\,\text{infinitely many}\,\, \widetilde{\zeta}_{i}\widetilde{f}_j\circ \varphi_i,\,\text{where}\,\, i\in \Gamma_j\,\,\text{and}\,\, j\in \Z_+,&\\ 
		\text{s.t.}\,\,\supp(\widetilde{\zeta}_{i}\widetilde{f}_j\circ \varphi_i)\subseteq B_{l_1,t}&\big\}.
	\end{align*}
	Let
			\begin{align*}
				\mathcal J^1:=\big\{j\in \Z_+: \supp(\widetilde{\zeta}_{i}\widetilde{f}_j\circ \varphi_i)\subseteq  B_{l_1,t},\, \text{where}\,\,i\in \Gamma_j\,\,\text{and}\,\, t\in P_{l_1}\big\}. 
			\end{align*}
		Then the set $\mathcal J^1$ is infinite.
		For each $j\in \mathcal J^1$, let
				\begin{align*}
				\mathcal J^1_j:=\big\{i\in \Gamma_j: \supp(\widetilde{\zeta}_{i}\widetilde{f}_j\circ \varphi_i)\subseteq   B_{l_1,t},\, \text{where}\,\, t\in P_{l_1}\big\},
			\end{align*}
			and let $f_{1,j}:=\sum_{i\in\mathcal J_j^1}\widetilde{\zeta}_{i}\widetilde{f}_j\circ \varphi_i$. Then $f_{1,j}\in B_1$. That $B$ is relatively compact in $L^{q}\left(M, \mu\right)$ implies that $B_1$
			is relatively compact in $L^{q}\left(M, \mu\right)$. Hence there exists a subsequence   $(f_{1,j{k_1}})\subseteq (f_{1,j})$ converging to $f^1$ in $L^q(M, \mu)$. Hence for each $i\in\mathcal J_{jk_1}^1\subseteq \mathcal J_{j}^1$, there exists $f^{1,i}\in L^q(M, \mu)$ such that 
			\begin{align*}
			\lim_{k_1\rightarrow\infty}\|\widetilde{\zeta}_{i}\widetilde{f}_{j{k_1}}\circ \varphi_i-f^{1,i}\|_{L^q(M, \mu)}=0.
		\end{align*}
	If $\mathcal J_{jk_1}^1=\Gamma_{jk_1}$, we terminate the process; otherwise, we continue to Step 2.
	
\noindent{\em Step III.2.} For $j\in \mathcal J^1$ and $k_1\in \Z_+$, we know that $f_{jk_1}:=\sum_{i\in \Gamma_{jk_1}}\widetilde{\zeta}_{i}\widetilde{f}_{jk_1}\circ \varphi_i.$ Let 
		\begin{align*}
			L_2:=\big\{l\in L_1:\exists\,\,\text{infinitely many}\,\, \widetilde{\zeta}_{i}\widetilde{f}_{j{k_1}}\circ \varphi_i,\,\text{where}\,\, i\in \Gamma_{jk_1}, j\in \mathcal J^1,\text{and}\,\, k_1\in \Z_+,&\\
			\text{s.t.}\,\,\supp(\widetilde{\zeta}_{i}\widetilde{f}_{j{k_1}}\circ \varphi_i)\subseteq  \cup_{t\in \mathcal{T}_l}B_{l,t}&\big\}.
		\end{align*}
	Let $l_2:=\min \{L_2\backslash \{l_1\}\}$
		and let
	\begin{align*}
	P_{l_2}:=\big\{t\in\mathcal{T}_{l_2}:\exists\,\,\text{infinitely many}\,\, \widetilde{\zeta}_{i}\widetilde{f}_{j{k_1}}\circ \varphi_i,\,\text{where}\,\, i\in \Gamma_{jk_1}, j\in \mathcal J^1,\,\,\text{and}\,\, k_1\in \Z_+,&\\ 
	\text{s.t.}\,\,\supp(\widetilde{\zeta}_{i}\widetilde{f}_{j{k_1}}\circ \varphi_i)\subseteq B_{l_2,t}&\big\}.
\end{align*}
Let
	\begin{align*}
			\mathcal J^2:=\{k_1\in \Z_+: \supp(\widetilde{\zeta}_{i}\widetilde{f}_{j{k_1}}\circ \varphi_i)\subseteq B_{l_2,t},\, \text{where}\,\,i\in \Gamma_{jk_1}, j\in \mathcal J^1,\text{and}\,\, t\in P_{l_2}\}.
		\end{align*}
		Then the set $\mathcal J^2$ is infinite.
		For each $j\in \mathcal J^1$ and $k_1\in \mathcal J^2$, let
		\begin{align*}
			\mathcal J^2_{jk_1}:=\{i\in \Gamma_{jk_1}: \supp(\widetilde{\zeta}_{i}\widetilde{f}_{j{k_1}}\circ \varphi_i)\subseteq  B_{l_2,t},\, \text{where}\,\, t\in P_{l_2}\},
		\end{align*}
			 and let $f_{2,j{k_1}}:=\sum_{i\in\mathcal J_{jk_1}^2}\widetilde{\zeta}_{i}\widetilde{f}_{j{k_1}}\circ \varphi_i$. Then $f_{2,j{k_1}}\in B_1$. Since $B_1$ is relatively compact in $L^q(M,\mu)$, there exists a subsequence  $(f_{2,j{{k_1}{k_2}}})\subseteq (f_{2,j{k_1}})$ converging to $f^2$ in $L^q(M, \mu)$. Hence for each $i\in\mathcal J_{jk_1k_2}^2\subseteq \mathcal J_{jk_1}^2$, there exists $f^{2,i}\in L^q(M, \mu)$ such that 
		\begin{align*}
			\lim_{k_2\rightarrow\infty}\|\widetilde{\zeta}_{i}\widetilde{f}_{j{{k_1}{k_2}}}\circ \varphi_i-f^{2,i}\|_{L^q(M, \mu)}=0.
		\end{align*}
		If $\mathcal J_{jk_1k_2}^1\cup\mathcal J_{jk_1k_2}^2=\Gamma_{jk_1k_2}$, we terminate the process; otherwise, we continue to Step $m$.
		
\noindent{\em Step III.m, $m=3,4,\ldots,T-1$}. We continue the process. In the same way, for $m=3,4,\ldots,T-1$, $j\in \mathcal J^{m-1}$, $k_{s_1}\in\mathcal J^{s_1+1},s_1=1,\ldots,m-2,$ and $k_{m-1}\in \Z_+$,  there exists a subsequence  $(f_{m,j{{k_1}{{\ldots}{k_{m}}}}})\subseteq (f_{m,j{{k_1}{{\ldots}{k_{m-1}}}}})$ converging to $f^{m}$ in $L^q(M, \mu)$. Hence for each $i\in\mathcal J_{j{{k_1}{{\ldots}{k_{m}}}}}^{m}\subseteq \mathcal J_{j{{k_1}{{\ldots}{k_{m-1}}}}}^{m}$, there exists $f^{m,i}\in L^q(M, \mu)$ such that 
\begin{align*}
	\lim_{k_{m}\rightarrow\infty}\|\widetilde{\zeta}_{i}\widetilde{f}_{j{{k_1}\ldots{k_{m}}}}\circ \varphi_i-f^{m,i}\|_{L^q(M, \mu)}=0.
\end{align*}
If $\cup_{i=1}^m \mathcal{J}_{j{{k_1}{{\ldots}{k_{m}}}}}^{i}=\Gamma_{j{{k_1}{{\ldots}{k_{m}}}}}$, we terminate the process; otherwise, we continue to Step $T$.

\noindent{\em Step III.T.} For $j\in \mathcal J^{T-1}$, $k_{s_1}\in\mathcal J^{s_1+1},s_1=1,\ldots,T-2,$ and $k_{T-1}\in \Z_+$, we know that $f_{j{{k_1}{{\ldots}{k_{T-1}}}}}:=\sum_{i\in \Gamma_{j{{k_1}{{\ldots}{k_{T-1}}}}}}\widetilde{\zeta}_{i}\widetilde{f}_{j{{k_1}{{\ldots}{k_{T-1}}}}}\circ \varphi_i.$ Let 
			\begin{align*}
			L_T:=\{l\in L_{T-1}:\exists\,\,\text{infinitely many}\,\, \widetilde{\zeta}_{i}\widetilde{f}_{j{{k_1}{{\ldots}{k_{T-1}}}}}\circ \varphi_i,\,\text{where}\,\, i\in \Gamma_{j{{k_1}{{\ldots}{k_{T-1}}}}}, j\in \mathcal J^{T-1},\,\,\text{and}  &\\k_{s_1}\in\mathcal J^{s_1+1},s_1=1,\ldots,T-2, k_{T-1}\in \Z_+,\,
			\text{s.t.}\,\, \supp(\widetilde{\zeta}_{i}\widetilde{f}_{j{{k_1}{{\ldots}{k_{T-1}}}}}\circ \varphi_i)\subseteq \cup_{t\in \mathcal{T}_l}B_{l,t}&\},
		\end{align*}
	and let $l_T:=\min \{L_T\backslash \{\cup_{s_2=1}^{T-1}l_{s_2}\}\}$. 
		Let
		\begin{align*}
			\mathcal J^T:=\{k_{T-1}\in \Z_+: \supp(\widetilde{\zeta}_{i}\widetilde{f}_{j{{k_1}{{\ldots}{k_{T-1}}}}}\circ \varphi_i)\subseteq  \mathcal B_{l_T,t}, \, \text{where}\,\,i\in\Gamma_{j{{k_1}{{\ldots}{k_{T-1}}}}}, j\in \mathcal J^{T-1}, \text{and}&\\ k_{s_1}\in\mathcal J^{s_1+1},s_1=1,\ldots,T-2&\}.
		\end{align*}
		Then the set $\mathcal J^T$ is infinite.
		For each  $j\in \mathcal J^{T-1}$ and $k_{s_1}\in\mathcal J^{s_1+1}$, where $s_1=1,\ldots,T-1$, we let
	\begin{align*}
			\mathcal J^T_{j{{k_1}{{\ldots}{k_{T-1}}}}}:=\{i\in \Gamma_{j{{k_1}{{\ldots}{k_{T-1}}}}}: \supp(\widetilde{\zeta}_{i}\widetilde{f}_{j{{k_1}{{\ldots}{k_{T-1}}}}}\circ \varphi_i)\subseteq   B_{l_T,t},\, \text{where}\,\, t\in P_{l_T}\},
		\end{align*}
		and let $f_{T,j{{k_1}{{\ldots}{k_{T-1}}}}}:=\sum_{i\in \mathcal{J}_{{j{{k_1}{{\ldots}{k_{T-1}}}}}}^{T}}\widetilde{\zeta}_{i}\widetilde{f}_{j{{k_1}{{\ldots}{k_{T-1}}}}}\circ \varphi_i$. Then $f_{T,j{{k_1}{{\ldots}{k_{T-1}}}}}\in B_1$. Since $B_1$ is relatively compact in $L^q(M,\mu)$, there exists a subsequence  $\{f_{T,j{{k_1}{{\ldots}{k_{T}}}}}\}\subseteq \{f_{T,j{{k_1}{{\ldots}{k_{T-1}}}}}\}$ converging to $f^{T}$ in $L^q(M, \mu)$. Hence for each $i\in\mathcal J_{j{{k_1}{{\ldots}{k_{T}}}}}^{T}$, there exists $f^{T,i}\in L^q(M, \mu)$ such that 
		\begin{align*}
			\lim_{k_{T}\rightarrow\infty}\|\widetilde{\zeta}_{i}\widetilde{f}_{j{{k_1}\ldots{k_{T}}}}\circ \varphi_i-f^{T,i}\|_{L^q(M, \mu)}=0.
		\end{align*}
Notice that $\cup_{m=1}^T \mathcal{J}_{j{{k_1}{{\ldots}{k_{T}}}}}^{m}=\Gamma_{j{{k_1}{{\ldots}{k_{T}}}}}$.  Let 
	\begin{align*}
f_{j{{k_1}{{\ldots}{k_{T}}}}}:=\sum_{m=1}^Tf_{m,j{{k_1}{{\ldots}{k_{T}}}}}=\sum_{m=1}^T\sum_{i\in \mathcal{J}_{{j{{k_1}{{\ldots}{k_{T}}}}}}^{m}}\widetilde{\zeta}_{i}\widetilde{f}_{j{{k_1}{{\ldots}{k_{T}}}}}\circ \varphi_i,
	\end{align*}
and
\begin{align*}
f:=\sum_{m=1}^{T}f^{m}=\sum_{m=1}^{T}\sum_{i\in \mathcal{J}_{{j{{k_1}{{\ldots}{k_{T}}}}}}^{m}}f^{m,i}.
\end{align*}
Then $(f_{j{{k_1}{{\ldots}{k_{T}}}}})\subseteq (f_j)$  is a convergent subsequence in $L^q(M,\mu)$, and thus for $i\in \mathcal{J}_{{j{{k_1}{{\ldots}{k_{T}}}}}}^{m}$ and $m=1,\ldots,T$,
		\begin{align}\label{eq:cov1}
			\lim_{k_T\rightarrow \infty}\|\widetilde{\zeta}_{i}\widetilde{f}_{j{{k_1}{{\ldots}{k_{T}}}}}\circ \varphi_i-f^{m,i}\|_{L^q(M,\mu)}=0.
		\end{align}

		\noindent {\em Part IV. Find a convergent subsequence of $(\widetilde{f}_j)_{j\in \Z_+}$ in $L^q(\R^n,\widetilde{\mu})$.}
		
		 For $i\in \mathcal{J}_{{j{{k_1}{{\ldots}{k_{T}}}}}}^{m}$ and $m=1,\ldots,T$, by \eqref{eq:cov1} and the definition and properties of the normal coordinate maps $\varphi_i$, we have $f^{m,i}\circ \varphi_i^{-1}\in L^q(\R^n,\widetilde{\mu})$ and
		\begin{align}\label{eq:con}
			\lim_{k_T\rightarrow\infty}\|\widetilde{\zeta}_{i}\widetilde{f}_{j{{k_1}{{\ldots}{k_{T}}}}}-f^{m,i}\circ \varphi_i^{-1}\|_{L^q(\R^n, \widetilde{\mu})}=0.
		\end{align} 
		Let
			\begin{align*}
		&\widetilde{f}_{j{{k_1}{{\ldots}{k_{T}}}}}^*:=\sum_{m=1}^{T}\widetilde{f}_{m,j{{k_1}{{\ldots}{k_{T}}}}}:=\sum_{m=1}^{T} 
		\sum_{i\in \mathcal{J}_{j{{k_1}{{\ldots}{k_{T}}}}}^{m}}\widetilde{\zeta}_{i}\widetilde{f}_{j{{k_1}{{\ldots}{k_T}}}},
			\end{align*}
		and 
			\begin{align*}
	   \widetilde{f}:=\sum_{m=1}^{T}\widetilde{f}^{m}:=\sum_{m=1}^{T}\sum_{i\in \mathcal{J}_{j{{k_1}{{\ldots}{k_{T}}}}}^{m}}f^{m,i}\circ \varphi_i^{-1}.
	\end{align*}
As each $f^{m,i}\circ \varphi_i^{-1}\in L^q(\R^n,\widetilde{\mu})$, so does $\widetilde{f}$. Note that $\widetilde{f}_{j{{k_1}{{\ldots}{k_{T}}}}}=\widetilde{f}_{j{{k_1}{{\ldots}{k_{T}}}}}^*+\widetilde{\zeta}_{N_{j{{k_1}{{\ldots}{k_{T}}}}}}\widetilde{f}_{j{{k_1}{{\ldots}{k_{T}}}}}$, where $\widetilde{\zeta}_{N_{j{{k_1}{{\ldots}{k_{T}}}}}}\widetilde{f}_{j{{k_1}{{\ldots}{k_{T}}}}}$ is given in \eqref{eq:s1}. Let $\widetilde{h}_{j{{k_1}{{\ldots}{k_{T}}}}}:=\widetilde{\zeta}_{N_{j{{k_1}{{\ldots}{k_{T}}}}}}\widetilde{f}_{j{{k_1}{{\ldots}{k_{T}}}}}$.       Then 
\begin{align}\label{eq:n1}
	\lim_{k_{T}\rightarrow\infty}\|\widetilde{h}_{j{{k_1}{{\ldots}{k_{T}}}}}\|_{L^q(\R^n,\widetilde{\mu})}^q&=
	\lim_{k_{T}\rightarrow\infty}\int_{\widetilde{U}_{N_{j{{k_1}{{\ldots}{k_{T}}}}}}}|\widetilde{h}_{j{{k_1}{{\ldots}{k_{T}}}}}|^qd\widetilde{\mu}\nonumber\\&=
	\lim_{k_{T}\rightarrow\infty}	\int_{\cup_{i\in \Gamma_{j{{k_1}{{\ldots}{k_{T}}}}}}B(z_i,\epsilon)\backslash \overline{B(z_i,\epsilon-\delta)}}|\widetilde{h}_{j{{k_1}{{\ldots}{k_{T}}}}}|^qd\widetilde{\mu}\nonumber\\&\leq\lim_{k_{T}\rightarrow\infty}\frac{\epsilon_{j{{k_1}{{\ldots}{k_{T}}}}}}{2^{k_{T}}}=0\qquad\qquad(\text{by \eqref{eq:l7}}).
\end{align}
	    Hence
		\begin{align*}
			\|\widetilde{f}_{j{{k_1}{{\ldots}{k_{T}}}}}-\widetilde{f}\|_{L^q(\R^n,\widetilde{\mu})}=&\|\sum_{m=1}^{T}\widetilde{f}_{m,j{{k_1}{{\ldots}{k_{T}}}}}+\widetilde{h}_{j{{k_1}{{\ldots}{k_{T}}}}}-\sum_{m=1}^{T}\widetilde{f}_{m}\|_{L^q(\R^n,\widetilde{\mu})}\\
			\leq&\sum_{m=1}^{T}\|\widetilde{f}_{m,j{{k_1}{{\ldots}{k_{T}}}}}-\widetilde{f}_{m}\|_{L^q(\R^n,\widetilde{\mu})}+\|\widetilde{h}_{j{{k_1}{{\ldots}{k_{T}}}}}\|_{L^q(\R^n,\widetilde{\mu})}\\
			=&\sum_{m=1}^{T}\sum_{i\in \mathcal{J}_{j{{k_1}{{\ldots}{k_{T}}}}}^{m}}\|\widetilde{\zeta}_{i}\widetilde{f}_{j{{k_1}\ldots{k_{T}}}}-f^{m,i}\circ \varphi_i^{-1}\|_{L^q(\R^n,\widetilde{\mu})}+\|\widetilde{h}_{j{{k_1}{{\ldots}{k_{T}}}}}\|_{L^q(\R^n,\widetilde{\mu})}\\
			\rightarrow& 0\qquad \qquad(\text{by \eqref{eq:con} and \eqref{eq:n1}}) 
		\end{align*} 
		as $k_T\rightarrow \infty$. Hence $\widetilde{B}$ is relatively compact in $L^q(\R^n,\widetilde{\mu})$. Theorem \ref{thm:maz} implies that \eqref{maz1}--\eqref{maz4} hold.
		Applying Lemma \ref{lem:9.3} completes the proof of Theorem \ref{thm:8.11}.
	\end{proof}
	\begin{thm}\label{thm:8.12}	Let $n\geq 2$, $M$ be a smooth connected oriented Riemannian $n$-manifold of bounded geometry.
		Let $2<q<\infty$ and $\mu$ be a finite positive Borel measure on $M$ (need not have compact support). Let $B=\{u \in C_{c}^{\infty}(M):\|u\|_{W^{1,2}_0(M)} \leq 1\}$.
		\begin{enumerate}
			\item[(a)] If $\underline{\operatorname{dim}}_{\infty}(\mu)>q(n-2) / 2$, then $B$ is relatively compact in $L^{q}(M, \mu)$.
			\item[(b)] If $\underline{\operatorname{dim}}_{\infty}(\mu)<q(n-2) / 2$, then $B$ is not relatively compact in $L^{q}(M, \mu)$.
		\end{enumerate}
	\end{thm}
\begin{proof}
	Use Lemma \ref{lem:f4}(b) and Theorem \ref{thm:8.11}. The proof  is similar to that of \cite[Theorem 3.2]{Hu-Lau-Ngai_2006} and is omitted.
\end{proof}


	\begin{proof}[Proof of Theorem \ref{thm:8.13}]\,
		(a) Use (PID), the result of Theorem \ref{thm:8.12}, and a similar argument as that in Theorem \ref{thm:1.1}(a).  
		
		(b)  Use (PIE), the result of Theorem \ref{thm:8.12}, and a similar argument as that in Theorem \ref{thm:1.1}(b). 
	\end{proof}

\begin{rema}
The Poincar\'e ball $\mathbb{B}^n(R)$, Poincar\'e half-space $\mathbb{U}^n(R)$, hyperboloid model $\mathbb{H}^n(R)$, and Beltrami-Klein model $\mathbb{K}^n(R)$ are homogeneous and of constant sectional curvature $(-1/R)$. Thus they are of bounded geometry and satisfy (PIE).

		\end{rema}

The following example shows that if the sectional curvature of a Riemannian manifold $M$ is unbounded, then Lemma \ref{lem:9.3} need not hold. In such a case, we cannot generalize  the compact embedding theorem of Maz'ja to Riemannian manifolds by using the method in this paper, and thus cannot obtain Theorems \ref{thm:8.13} and \ref{thm:8.14}. 
	
	\begin{exam}\label{ex:1}
		For any integer $k\geq 2$, let $\textbf{y}_k:=(k,0,0)$ and define $f_k:\R^3\rightarrow \R$ as 
		$f_k(\textbf{x}):=e^{-1/\sqrt{1-k^2|\textbf{x}-\textbf{y}_k|^2}}$,  where $\textbf{x}$ satisfies $|\textbf{x}-\textbf{y}_k|<1/k$. Let $F(\textbf{x}):=\sum_{k\geq2} f_k(\textbf{x})$.  Let $M$ be the 3-dimensional smooth manifold in $\R^4$ formed by the graph of $F(\textbf{x})$. Let $\{(U_{i},\varphi_{i})\}_{i\in \Lambda}$ be an atlas on $M$, where $U_i:=B^M(p_i,\epsilon)$ is a geodesic ball and $\varphi_i$ is defined as in  \eqref{eq:8.1}. 
		Let $\widetilde{\mu}$ be the 3-dimensional Lebesgue measure on $\R^3$ and let $\mu:=\sum_{i\in \Lambda}\widetilde{\mu}\circ\varphi_i$ be the positive finite Borel measure induced on $M$. Then the sectional curvature of $M$ is not bounded above. For $2<q<6$,
		\begin{align*}
			\lim _{r \rightarrow 0^{+}} \sup _{x \in \R^n ; \delta \in(0,r)} \delta^{-1/ 2}\, \widetilde{\mu}\left(B(x,\delta)\right)^{1 / q}=0,
		\end{align*}	
		but
		\begin{align}\label{ex:11}
			\lim _{r \rightarrow 0^{+}} \sup _{w \in M ;  \delta\in(0, r)} \delta^{-1/2}\, \mu\left(B^M(w,\delta)\right)^{1 / q}=+\infty.
		\end{align}
	\end{exam}
	\begin{figure}[h!]
		\centerline{
			\begin{tikzpicture}[scale=0.7]
				\draw[fill=magenta](-2,0)circle(.04);
				\draw[black](0,0)arc(0:360:2);    
				\draw[black](2,0)arc(0:360:4);
				\draw[purple,dashed](-1,3)arc(0:360:1);
				\draw[blue](-2,0)--(0.2,3.35);
				\draw[blue](-2,0)--(-4.2,3.35);
				\draw[fill=magenta](-2,3)circle(.04);
				\draw[blue](-1.97,-0.04)--(-1.78,-0.18);
				\draw[blue](-0.85,1.64)--(-0.65,1.535);
				\draw[blue](0.21,3.33)--(0.44,3.225);
				\draw[<->](-1.82,-0.1)--(-0.71,1.5);
				\draw[<->](-0.67,1.57)--(0.37,3.2);
				\draw[black] (-2.55,-0.2) node[right]{$z$};
				\draw[black] (-3.53,1.3) node[right]{$e$};
				\draw[black] (-1.55,1.55) node[right]{$c$};
				\draw[black] (-2.1,2.8) node[right]{$x$};
				\draw[black] (-4.8,3.6) node[right]{$a$};
				\draw[black] (-0.08,3.7) node[right]{$b$};
				\draw[black] (-1.35,0.6) node[right]{$\rho$};
				\draw[black] (-0.35,2.2) node[right]{$2\delta$};
				\draw[fill=black](-3.09,1.67)circle(.04);
				\draw[fill=black](-4.2,3.35)circle(.04);
				\draw[fill=black](0.2,3.35)circle(.04);
				\draw[fill=black](-2,3)circle(.04);
				\draw[fill=black](-0.91,1.66)circle(.04);
				\foreach \i/\tex in {0/(a)}
				\draw(-2,-4.4)node[below]{\tex};
				\foreach \i/\tex in {0/\text{}}
				\draw(0,-2)node[below]{\tex};
			\end{tikzpicture}\qquad\qquad\quad
			\begin{tikzpicture}[scale=0.7]
				\draw[black] (-2,6)..controls (-2.7,5.8) and (-2.8,3.5) .. (-3,1);
				\draw[black] (-3,1)..controls (-3.2,0) and (-3.4,-1) .. (-4,-1.5);
				\draw[black] (-2,6)..controls (-1.3,5.8) and (-1.2,3.5) .. (-1,1);
				\draw[black] (-1,1)..controls (-0.8,0) and (-0.6,-1) .. (0,-1.5);
				\draw[blue] (-2.57,5)..controls (-2.2,4.85) and (-1.8,4.85) .. (-1.43,5);
				\draw[blue] (-3,1)..controls (-2.3,0.8) and (-1.7,0.8) .. (-1,1);
				\draw[fill=magenta](-2,6)circle(.04);
				\draw[fill=magenta](-2.57,5)circle(.04);
				\draw[fill=magenta](-1.43,5)circle(.04);
				\draw[fill=magenta](-3,1)circle(.04);	
				\draw[fill=magenta](-1,1)circle(.04);
				\draw[fill=magenta](-2,3)circle(.04);
				\draw[black] (-2.2,6.3) node[right]{$p$};
				\draw[black] (-3.2,5) node[right]{$\eta$};	
				\draw[black] (-1.5,5) node[right]{$\xi$};	
				\draw[black] (-3.7,1) node[right]{$\alpha$};	
				\draw[black] (-1.05,1) node[right]{$\beta$};	
				\draw[black] (-2.2,2.75) node[right]{$w$};	
				\path[fill=blue!20,opacity=.3] (-2.57,5) to [out=255, in=270](-3,1) 
				to[out=-22, in=5] (-1,1) to[out=95, in=285](-1.43,5) to [out=10, in=-30](-2.57,5) ;
				\foreach \i/\tex in {0/(b)}
				\draw(-2,-2)node[below]{\tex};
				\foreach \i/\tex in {0/\text{}}
				\draw(0,-2)node[below]{\tex};
			\end{tikzpicture}
		}
		\caption{(a) shows the set $\varphi(B^M(p,\rho+2\delta))$ in $\R^3$. (b) shows, for a fixed $k$, a 2-dimensional cross-section of the 3-dimensional smooth manifold $M$. The shaded part of (b) is $R_{\delta}^M$.}\label{fig:j5}
	\end{figure}
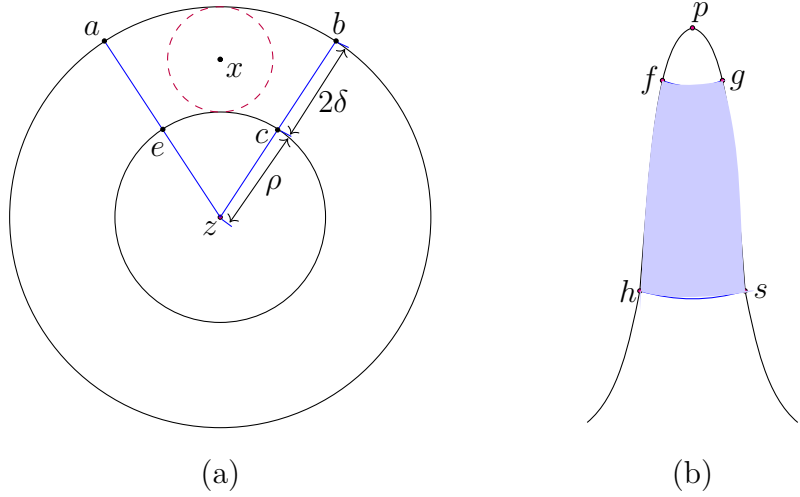
	\begin{proof}\,In $\R^4:=\{(x_1,x_2,x_3,x_4):x_1,x_2,x_3,x_4\in \R\}$, let $x_2=x_3=0$ and look at the $x_1x_4$-plane. 
		Let 
			\begin{align}\label{eq:inf}
				x_4=x_4(x_1):=e^{-1/\sqrt{1-k^2|x_1-k|^2}},
				\end{align}
		and let  $p_k:=(k,0,0,1/e)\in M$ satisfy \eqref{eq:inf}. Then $x_4'|_{p_k}=0$ and $x_4''|_{p_k}=-k^2/e$. Hence
		$$K|_{p_k}=\frac{|x_4''|}{(1+(x_4')^2)^{3/2}}=\frac{k^2}{e},$$
		where $K|_{p_k}$ is the curvature of the curve $x_4(x_1)$ at ${p_k}$.
		Then $K_M({p_k})=k^4/e^2$. As $k$ can be arbitrarily large, the sectional curvature of $M$ is not bounded above.
		
		Let $(U_i,\varphi_i)$ be a coordinate chart with  $p_i\in \supp(\mu)$.	For simplicity, we denote $(U_i,\varphi_i)$ by $(U,\varphi)$ and $p_i$ by $p$. Let $B(z,\epsilon):=\varphi(B^M(p,\epsilon))$. For each $ x\in B(z,\epsilon)\backslash \{z\}$, there exist positive constants $\rho$ and $\delta$ such that $B(x,\delta)\subseteq  B(z,\rho+2\delta)\setminus B(z,\rho)$ and $\rho+2\delta\leq \epsilon$.  Let $R_{\delta}:=B(z,\rho+2\delta)\backslash B(z,\rho)$. Let $a$ (resp. $e$) and $b$ (resp. $c$) be the intersections of  two rays with $\partial B(z,\rho+2\delta)$ (resp. $\partial B(z,\rho)$). Note that  $|ez|=|cz|=\rho$, $|ae|=|bc|=2\delta$ (see Figure \ref{fig:j5}(a)). Using the normal coordinate map $\varphi$, we define $R_{\delta}^M:=\varphi^{-1}(B(z,\rho+2\delta))\backslash \varphi^{-1}(B(z,\rho))$. Then there exist points $\alpha, \beta, \eta, \xi\in \overline{R_{\delta}^M}$ such that $\alpha$, $\beta$, $\eta$, $\xi$ satisfy \eqref{eq:inf}, $L(\gamma_{p\eta})=L(\gamma_{p\xi})=\rho$, and $L(\gamma_{\alpha\eta})=L(\gamma_{\beta\xi})=2\delta$ (see  Figure \ref{fig:j5}(b)). Since $x_4=e^{-1/\sqrt{1-k^2|\textbf{x}-\textbf{y}_k|^2}}$, we have $$|\textbf{x}-\textbf{y}_k|=\frac{\sqrt{(\ln x_4)^2-1}}{k\ln x_4}.$$
We define the function $r(x_4):=\sqrt{(\ln x_4)^2-1}/(|k\ln x_4|)$ on $M$. Now fix $x_4$. If $k$ tends to $\infty$, then $r(x_4)$ tends to $0$. Hence for each integer $k>0$, there exists $m>0$ such that 
		\begin{align}\label{eq:hs}
			L(\gamma_{\alpha\beta})=\delta^m,
			\end{align}
		where $m$ is proportional to $k$, and $\delta\in(0,1)$. Let $\gamma_{\alpha\eta}$ be the admissible curve from $\alpha$ to $\eta$ satisfying \eqref{eq:inf} and let 
		$$\sigma:=\lfloor \delta^{1-m}\rfloor+1, \quad\alpha:=(t_1^1,t_1^2,t_1^3,h_1),\quad \text{and}\quad \eta:=(t_\sigma^1,t_\sigma^2,t_\sigma^3,h_\sigma).$$
			 For $1\leq i\leq \sigma$, let $t_i:=(t_i^1,t_i^2,t_i^3,h_i)$ and $t_{i+1}:=(t_{i+1}^1,t_{i+1}^2,t_{i+1}^3,h_{i+1})$ such that $t_i, t_{i+1}\in\gamma_{\alpha\eta}$ and $L(\gamma_{t_it_{i+1}})=2\delta^{m}$. Note that $t_1=\alpha$ and $t_\sigma=\eta$.
		For $1\leq i\leq \sigma$, let $R_{\delta,i}^M:=(\{x_4\leq h_{i+1}\}\cap R_{\delta}^M)\backslash(\{x_4<h_{i}\}\cap R_{\delta}^M)$ and let $m\in [1,\sigma]$ such that 
		\begin{align*}
			{\rm Vol}(R^M_{\delta,m}):=\max\big\{{\rm Vol}(R_{\delta,i}^M): i\in [1,\sigma]\big\}.
		\end{align*}
		Let $w$ be the center of $R^M_{\delta,m}$. For any point $u\in R^M_{\delta,m}$, there exists a point $v\in R^M_{\delta,m}$ such that $v$ and $u$ are on the same horizontal line, and $v$ and $w$ are on the same vertical line.
		 By  using equations $L(\gamma_{t_it_{i+1}})=2\delta^{m}$ and \eqref{eq:hs}, we have
	$$d_M(v,w)\leq \delta^m\qquad \text{and}\qquad d_M(v,u)\leq  \delta^m.$$ 
		Hence $d_M(u,w)\leq 2\delta^m.$
	Thus there exists $B^M(w,2\delta^m)$ such that ${R}^M_{\delta,m}\subseteq B^M(w,2\delta^m)$.
		Hence 
		\begin{align}\label{eq:R1}
	\mu( R_{\delta}^M)/\delta^{1-m}\leq\mu( {R}^M_{\delta,m})\leq \mu(B^M(w,2\delta^m)).
		\end{align}
		We know that 
		\begin{align}\label{eq:eq}
			\widetilde{\mu}(R_{\delta})&=\widetilde{\mu}(B(z,\rho+2\delta))- \widetilde{\mu}(B(z,\rho))\nonumber\\
			&=\mu\circ\varphi^{-1}(B(z,\rho+2\delta))- \mu\circ\varphi^{-1}(B(z,\rho))\,\,\,\,(\text{ definition of}\,\,\widetilde{\mu} )\nonumber\\
			&=\mu(R_{\delta}^M)\qquad\qquad\qquad\qquad\qquad\qquad\quad\, \qquad(\text{ definition of}\,\, R_{\delta}^M)
			\end{align}
		and
			\begin{align*}
			\lim _{r \rightarrow 0^{+}} \sup _{x \in \R^n ; \delta \in(0,r)} \delta^{-1/ 2}\, \widetilde{\mu}\left(B(x,\delta)\right)^{1 / q}=0,
		\end{align*}	
	where $2<q<6$.
		Thus
	\begin{align}\label{eq:545}
			(2\delta^m)^{1-3/ 2} \mu\left(B^M(w,2\delta^m)\right)^{1 / q}
			&\geq (2\delta^m)^{-1/ 2} (\mu( R_{\delta}^M )/\delta^{1-m})^{1 / q} \qquad\qquad\qquad(\text{by \eqref{eq:R1}})\nonumber\\
			&=c_1\delta^{-m/2+(m-1)/q}(\widetilde{\mu}(R_{\delta}))^{1/q}\qquad\qquad\quad\,\,\, \quad(\text{by \eqref{eq:eq}})\nonumber\\
			&=c_1\delta^{-m/2+(m-1)/q}(4\pi/3(\rho+2\delta)^3-\rho^3)^{1 / q}\nonumber\\
			&=c_2(\delta^{-mq/2+m-1}(6\rho^2\delta+12\rho\delta^2+8\delta^3))^{1 / q}\nonumber\\
			&=c_2(6\rho^2\delta^{m(1-q/2)}+12\rho\delta^{1+m(1-q/2)}+8\delta^{2+m(1-q/2)})^{1 / q}.
		\end{align}
		If $k$ tends to $\infty$, i.e., $m$ tends to $\infty$, then $\delta^m$ tends to $0$. Since $2<q<6$, we have $1-q/2<0$. It follows from \eqref{eq:545} that
		\begin{align*}
			\lim _{r \rightarrow 0^{+}} \sup _{w \in M ;  2\delta^m\in(0, r)} (2\delta^m)^{1-3/ 2}\, \mu\left(B^M(w,2\delta^m)\right)^{1 / q}=+\infty,
		\end{align*}
		and thus \eqref{ex:11} holds.
		\end{proof}
	
If  $\Omega$ is an unbounded open set of a Riemannian manifold $M$, then (PID), (MPID), (PIE), and (MPIE) may not hold. In this case, we cannot apply these four important inequalities to prove Theorem \ref{thm:8.13}, Theorem \ref{thm:8.14}, and Proposition \ref{prop:9.1}. We give the following counterexamples for (PIE) and (MPIE). 	
	
\begin{exam}\label{ex:2}
Let $M=\R$.
	\begin{enumerate}
		\item [(a)] (PID) fails on $\R$. 
		\item [(b)]  (MPID) fails on $\R$.
\end{enumerate}
\end{exam}
\begin{proof}(a)	Let  $u_k\in  W_0^{1,2}(\R)$ be defined as 
$$u_k:=\left\{
\begin{aligned}
	&1,\qquad\qquad\qquad\qquad\qquad\,	[-1,1], \\
	&x/(1-k)+k/(k-1),\quad(1,k), \\ 
	&x/(k-1)+k/(k-1),\quad(-k,-1), \\ 
	&0,\qquad\qquad\qquad\qquad\qquad\,\text{otherwise}.
\end{aligned}
\right.$$
Then
\begin{align*}
	&\int_\R|\nabla u_k|^2dx\leq  2\int_{-k}^{-1}\Big|\frac{1}{k-1}\Big|^2dx=\frac{2}{k-1}\quad\text{and}\quad
		\int_\R|u_k|^2dx\geq \int_{-1}^11^2dx=2.
\end{align*}
Hence \eqref{eq:PI} would imply
$$\int_\R|u_k|^2dx\leq C\int_\R|\nabla u_k|^2dx\leq \frac{2C}{k-1}.$$
Letting $k\rightarrow\infty$ yields $2\leq 0$, a contradiction.

(b)	Next, let $\mu(\R)=1$ and let $u_k\in  W_0^{1,2}(\R)$ be defined as 
$$u_k:=\left\{
\begin{aligned}
	&1,\qquad \quad\qquad\,  [-k,k], \\
	&x/k+2,\quad\quad \, (-2k,-k),\\ 
	&-x/k+2,\quad (k,2k), \\ 
	&0,\qquad\qquad \quad \, \text{otherwise}.
\end{aligned}
\right.$$
Then 
\begin{align*}
	\int_\R|\nabla u_k|^2dx\leq  2\int_{-2k}^{-k}\Big|\frac{1}{k}\Big|^2dx=\frac{2}{k}\quad\text{and}\quad
	\int_\R|u_k|^2d\mu\geq\int_{-k}^k|u_k|^2d\mu=\mu(-k,k).
\end{align*}
Hence \eqref{eq:PIN} would imply
$$\int_\R|u_k|^2d\mu\leq C\int_\R|\nabla u_k|^2d\nu\leq \frac{2C}{k}.$$
Letting $k\rightarrow\infty$ yields $1\leq 0$, a contradiction.
\end{proof}	
	
		\section{Hodge theorem for Laplacians on space of $k$-forms on compact Riemannian manifolds}\label{S:Hodge}
\setcounter{equation}{0}
We first summarize some notations and definitions that will be used throughout this section. Then, by assuming that $\mu$ satisfies the Poincar\'e inequality, we define the Laplacian.  We generalize the compact embedding theorem of Maz'ja\cite{Maz'ja_1985} to $k$-forms on a Riemannian manifold. Moreover, we use $\underline{\dim}_{\infty}(\mu)$ to study the compactness of the embeddings $\operatorname{dom}(\mathcal{E}_D^k)\hookrightarrow L^{2}(\bigwedge^kT^*\Omega, \mu)$ and $\operatorname{dom}(\mathcal{E}_E^k)\hookrightarrow L^{2}(\bigwedge^kT^*M, \mu)$. We show that the classical Hodge theorem holds when $\Omega=M$ and the measure $\mu$ is absolutely continuous with a positive and bounded density.
\subsection{Notation in this section}\label{S:No}
We summarize some notations needed in the paper; details can be found in \cite{Tu_2011,Petersen_2016,Warner_1983}.  Let $(M,g)$ be a Riemannian $n$-manifold with Riemannian metric $g$.  Denote the {\em cotangent space} of $M$ at $p$ by $T^*_pM$. Let $${\bigwedge}^k T^*_pM:=\underbrace{T^*_pM\wedge\cdots\wedge T^*_pM}_{k\,\,\text{times}}$$
and let ${\bigwedge}^k T^*M:=\bigcup_{p\in M}{\bigwedge}^k T^*_pM$ denote the {\em $k$-th exterior power of the cotangent bundle}. The space of sections (resp. $C^\infty$ sections) of $\bigwedge^kT^*M$ is denoted by $\Gamma(\bigwedge^kT^*M)$ (resp. $\Gamma^\infty(\bigwedge^kT^*M)$).  Elements in  $\Gamma(\bigwedge^kT^*M)$ (resp. $\Gamma^\infty(\bigwedge^kT^*M)$) are called {\em  $k$-forms} (resp. {\em smooth $k$-forms}) on $M$.  For simplicity, we let
$$\mathcal{I}_{k,n} := \{I:=(i_1,\ldots,i_k)|1\leq i_1<i_2<\cdots<i_k\leq n\}$$
be the set of all strictly ascending multi-indices between 1 and $n$ of length $k$. Let $(U,(x^i))$ be a coordinate chart on $M$. Then on $U$, every $k$-form $\omega\in \Gamma(\bigwedge^kT^*U)$ can be written as a linear combination  
\begin{align}\label{eq:1.1*}
	\omega=\sum_{I\in \mathcal{I}_{k,n}} \omega_I\,dx^I,
\end{align}
where the coefficients $\omega_I:U\rightarrow \R$ are the components of $\omega$ in the coordinate chart and  $\{dx^I\}$ is an orthonormal frame for $\bigwedge^kT^*U$. A $k$-form $\omega=\sum_{I\in \mathcal{I}_{k,n}} \omega_I\,dx^I$ on $U$ is smooth if and only if the coefficient
functions $\omega_I$ are smooth on $U$; a $k$-form $\omega$ on $M$ is smooth if and only if its restriction to any chart $(U,(x^i))$ is  smooth (see, e.g., \cite[Proposition 18.8]{Tu_2011}). Let $(U,(x^i))$ be a coordinate chart on $M$. The {\em exterior derivative} $d:\Gamma^\infty(\bigwedge^kT^*M)\rightarrow \Gamma^\infty(\bigwedge^{k+1}T^*M)$ is defined locally on $U$ by the formula
\begin{align*}
	d\omega:=\sum_{r=1}^n\sum_{I\in \mathcal{I}_{k,n}}\frac{\partial \omega_I}{\partial x^r}dx^r\wedge dx^I.
\end{align*}
If $k\geq n$, we let $d\omega=0$. Note that $(d\omega)_p$ is independent of the choice of $U$
containing $p$ (see, e.g., \cite[Section 19]{Tu_2011}).

Define the {\em Hodge star operator} $\star:\bigwedge^kT_p^*M\rightarrow \bigwedge^{n-k}T_p^*M$ as 
$$\star(e^{i_1}\wedge\cdots\wedge e^{i_k})=e^{j_1}\wedge\cdots\wedge e^{j_{n-k}},\qquad 0\leq k\leq n,$$
where $e^{i_1}\wedge\cdots\wedge e^{i_k}\wedge e^{j_1}\wedge\cdots\wedge e^{j_{n-k}}=e^{1}\wedge\cdots\wedge e^{n}$ and $\{e^{1}\wedge\cdots\wedge e^{n}\}$ is a positively oriented orthonormal basis of $T_p^*M$. 
For $\omega,\eta\in \Gamma^\infty(\bigwedge^kT^*M)$, define the {\em pointwise inner product} of $\omega$ and $\eta$ as
$$\langle \omega, \eta\rangle :=\star(\omega\wedge\star\eta),\,\,\text{i.e.,}\,\,\langle \omega(p), \eta(p)\rangle :=\star(\omega(p)\wedge\star\eta(p))\quad\text{for all}\,\,p\in M.$$
For $\omega\in \Gamma^\infty(\bigwedge^kT^*M)$, let $|\omega|:=\langle \omega, \omega\rangle ^{1/2}$ be {\em the pointwise  norm} of $\omega$. Define the {\em inner product on $\Gamma^\infty(\bigwedge^kT^*M)$} as
\begin{align*}
	\left\langle \omega, \eta\right\rangle_M:=\int_M \left\langle \omega, \eta\right\rangle \,d\nu,\qquad {\magenta}\text{for all}\,\, \omega,\eta\in \Gamma^\infty({\bigwedge}^kT^*M).
\end{align*}
Define the {\em adjoint} $d^*:\Gamma^\infty(\bigwedge^{k+1}T^*M)\rightarrow \Gamma^\infty(\bigwedge^kT^*M)$
of $d$ via the formula
\begin{align*}
	\langle d^*\omega, \eta\rangle_M=\langle \omega, d\eta\rangle_M
\end{align*}
for all $\eta\in \Gamma^\infty(\bigwedge^kT^*M)$, where $0\leq k\leq n$.
Define the {\em support} of $\omega\in \Gamma^\infty(\bigwedge^kT^*M)$ as the closure of the set $\{p\in M: \omega(p)\neq 0\}$.

Let $\Omega\subseteq M$ be an open subset.  Note that if $\Omega=M$, then $\partial \Omega=\emptyset$. Let $\Gamma^C(\bigwedge^kT^*\Omega)$, $\Gamma^{\infty}(\bigwedge^kT^*\Omega)$, and $\Gamma_c^{\infty}(\bigwedge^kT^*\Omega)$ denote, respectively,  continuous $k$-forms,  $C^{\infty}$ $k$-forms, and $C^{\infty}$ $k$-forms with compact support.

Let $L^2(\bigwedge^kT^*\Omega)$ denote the space of measurable $k$-forms on $\Omega$ satisfying $\int_{\Omega}|\omega|^2\,d\nu<\infty$, with the inner product $\langle\cdot,\cdot\rangle_{L^2(\bigwedge^kT^*\Omega)}$ and associated norm $\|\cdot\|_{L^2(\bigwedge^kT^*\Omega)}$ defined respectively as 
\begin{align*}
	\left\langle \omega, \eta\right\rangle_{L^2(\bigwedge^kT^*\Omega)}:=\int_\Omega \left\langle \omega, \eta\right\rangle \,d\nu\quad\text{and}\quad \|\omega\|_{L^2(\bigwedge^kT^*\Omega)}:=\Big(\int_{\Omega}|\omega|^2\,d\nu\Big)^{1/2},
\end{align*}
where $\omega$ is measurable with respect to the Riemannian volume form and $|\omega|$ is defined a.e. in $\Omega$. Let $\mu$ be a positive finite Borel measure on $M$ with  ${\rm supp}(\mu)\subseteq\overline{\Omega}$. Let $L^2\big(\bigwedge^kT^*\Omega,\mu\big)$ denote the space of $k$-forms  measurable with respect to $\mu$ on $\Omega$ satisfying $\int_{\Omega}|\omega|^2\,d\mu<\infty$, with the inner product $\langle\cdot,\cdot\rangle_\mu$ and associated norm $\|\cdot\|_{\mu}$ defined respectively as
\begin{align*}
	\left\langle \omega, \eta\right\rangle_{\mu}:=\int_\Omega \left\langle \omega, \eta\right\rangle \,d\mu\quad\text{and}\quad \|\omega\|_{\mu}:=\Big(\int_{\Omega}|\omega|^2\,d\mu\Big)^{1/2}.
\end{align*}
Then $L^2\big(\bigwedge^kT^*\Omega,\mu\big)$ is a Hilbert space.

We refer the reader to \cite{Morrey_1966,Scott_1995} for the following  definitions. We write $\Omega'\subset\subset \Omega$ if $\Omega'$ is an open subset of $\Omega$ and its closure $\overline{\Omega'}$ is contained on $\Omega$.  Define
\begin{align*}
	L^1_{\rm loc}\Big({\bigwedge}^kT^*\Omega\Big):=\Big\{u\in \Gamma\Big({\bigwedge}^kT^*\Omega\Big):u\,\, \text{is measurable on}\,\,\Omega \,\, \text{and for any}&\\ \Omega'\subset \subset\Omega, u\in L^1\Big({\bigwedge}^kT^*\Omega'\Big)&\Big\}.
\end{align*}
Given a $k$-form $\omega=\sum_{I\in \mathcal{I}_{k,n}}\omega_Idx^I\in L^1_{\rm loc}(\bigwedge^kT^*M)$,  we say that $\omega$ has a {\em generalized gradient} if, for each coordinate system, the pullbacks of the coordinate functions $\omega_I$ have generalized gradients in the sense that for all $\varphi\in C^\infty_c(\R^n)$, 
\begin{align}\label{eq:g1*}
	\int_{\R^n}\omega_I(x)\frac{\partial^\alpha\varphi}{\partial x^\alpha}(x)dx=(-1)^{|\alpha|}\int_{\R^n} \frac{\partial^\alpha \omega_I}{\partial x^\alpha}(x)\varphi(x)dx.
\end{align}
Let
\begin{align}\label{eq:so1*}
	W\Big({\bigwedge}^kT^*\Omega\Big):=\Big\{\omega\in L^1_{\rm loc}\Big({\bigwedge}^kT^*\Omega\Big):\omega\,\,\text{has a generalized gradient}\Big\}.
\end{align}
Now choose an
atlas $\mathcal{A}$ for $\Omega$. For $(U, \varphi):=(U,(x^1,\ldots,x^n))\in \mathcal{A}$, write $\omega=\sum_{I\in \mathcal{I}_{k,n}}\omega_Idx^I\in L^1_{\rm loc}(\bigwedge^kT^*\Omega)$. Define the
{\em local gradient modulus} as
\begin{align*}
	|\nabla_U\omega(x)|^2:=\sum_{I,l}\bigg|\frac{\partial \omega_I}{\partial x^l}(x)\bigg|^2
\end{align*}
and the {\em global gradient modulus} as
\begin{align*}
	|\nabla\omega(x)|^2:=\sum_{U\in \mathcal{A}}\big|\nabla_U\omega(x)\big|^2.
\end{align*}
If  $\mathcal{A}$ is a locally finite
cover of $\Omega$, then we may define the {\em (classical) Sobolev space} as
\begin{align*}
	W^{1,2}_{\mathcal{A}}\Big({\bigwedge}^kT^*\Omega\Big):=\Big\{\omega\in \Gamma\Big({\bigwedge}^kT^*\Omega\Big):|\omega|,\,|\nabla\omega|\in L^2(\Omega)\Big\},
\end{align*}
with norm $\|\omega\|_{L^2(\bigwedge^kT^*\Omega)}+\|\nabla\omega\|_{L^2(\bigwedge^kT^*\Omega)}$.

We now specify a class of atlases that yield equivalent Sobolev spaces. We say that 
a coordinate chart $(U, \varphi)$ is {\em regular}, if  $\overline{U}$ is compact and there is another
chart $(V, \psi)$ with $\overline{U}\subseteq V$ and $\psi|_U=\varphi$.

Let $M$ be a compact Riemannian $n$-manifold. We may select finitely
many systems $\{(V_i,\varphi_i)\}_{i=1}^N$ satisfying $M\subseteq \cup_{i=1}^NV_i$.  By relabeling the $V_i$, we may choose a regular atlas $\{(U_i,\varphi_i|_{U_i})\}_{i=1}^N$ on $M$ such that $\overline{U}_{i}\subseteq V_{i}$ and $M\subseteq\cup_{i=1}^N U_i$.	

For $\omega=\sum_{I\in \mathcal{I}_{k,n}}\omega_Idx^I\in W({\bigwedge}^kT^*\Omega)$, where $0\leq k<n$, we define its 
{\em exterior derivative $d\omega$} as the $(k+1)$-form defined locally as
\begin{align}\label{eq:d*}
	d\omega:=\sum_{r=1}^n\sum_{I\in \mathcal{I}_{k,n}}\frac{\partial \omega_I}{\partial x^r}dx^r\wedge dx^I,
\end{align}
where $\partial \omega_I/\partial x^r$ is defined as in \eqref{eq:g1*}. For $k\geq n$, we let $d\omega=0$. 
For  $\omega=\sum_{I\in \mathcal{I}_{k,n}}\omega_Idx^I\in  W({\bigwedge}^kT^*\Omega)$ and $k\geq 1$, we 
define its adjoint $d^*\omega$ by the condition that 
\begin{align}\label{eq:ad*}
	\langle d^*\omega, \eta\rangle_{L^2(\bigwedge^kT^*\Omega)}=\langle \omega, d\eta\rangle_{L^2(\bigwedge^kT^*\Omega)}\qquad \text{for all}\,\,\eta\in W\Big({\bigwedge}^kT^*\Omega\Big).
\end{align}
If $\omega$ is a $0$-form, we define $d^*\omega =0$.
The {\em Sobolev space of $k$-forms} on $M$ is defined as
\begin{align*}
	W^{1,2}\Big({\bigwedge}^kT^*\Omega\Big):=\Big\{\omega\in W\Big({\bigwedge}^kT^*\Omega\Big)\bigcap L^2\Big({\bigwedge}^kT^*\Omega\Big):d\omega\in L^2\Big({\bigwedge}^{k+1}T^* \Omega\Big)\\\text{and}\,\,d^*\omega\in L^2\Big({\bigwedge}^{k-1}T^* \Omega\Big)\Big\}
\end{align*}
equipped with the inner product $\langle\cdot,\cdot\rangle_{W^{1,2}(\bigwedge^kT^*\Omega)}$ defined as
$$ \left\langle \omega, \eta\right\rangle_{W^{1,2}(\bigwedge^kT^*\Omega)} :=\left\langle \omega, \eta\right\rangle_{L^2(\bigwedge^kT^*\Omega)}+\left\langle d\omega, d\eta\right\rangle_{L^2(\bigwedge^kT^*\Omega)}+\left\langle d^*\omega, d^*\eta\right\rangle_{L^2(\bigwedge^kT^*\Omega)}.$$
The norm associated to $\langle\cdot,\cdot\rangle_{W^{1,2}(\bigwedge^kT^*\Omega)}$ is defined as
\begin{align*}
	\|u\|_{W^{1,2}(\bigwedge^kT^*\Omega)}:=\Big(\int_{\Omega}|u|^2\,d\nu+\int_{\Omega}|d u|^2\,d\nu+\int_{\Omega}|d^* u|^2\,d\nu\Big)^{\frac{1}{2}}.
\end{align*}
Let $W^{1,2}_0(\bigwedge^kT^*\Omega)$ denote the closure of $\Gamma^\infty_c(\bigwedge^kT^*\Omega)$ in the $W^{1,2}(\bigwedge^kT^*\Omega)$ norm.

It is known that regular atlases yield classical Sobolev spaces $W^{1,2}_{\mathcal{A}}(\bigwedge^kT^*\Omega)$  equivalent to $W^{1,2}(\bigwedge^kT^*\Omega)$ (see, e.g, \cite{Scott_1995, Morrey_1966}). 

\subsection{Kre\u{\i}n-Feller operators on  space of $k$-forms} \label{S:frac}
Let $M$ be a compact Riemannian $n$-manifold and let $\Omega \subseteq M$ be an open set. Let $\mu$ be a positive bounded regular Borel measure on $M$ with ${\rm supp}(\mu)\subseteq\overline{\Omega}$.  Our method of defining a   Laplacian on $L^2\big(\bigwedge^kT^*\Omega,\mu\big)$ associated with $\mu$ is similar to that in Section \ref{S:Pre*}. We only indicate the modifications. We recall the following definitions of {\em Poincar\'e inequalities} depending on  boundary conditions (see, e.g., \cite{Arnold-Falk-Winther_2006}):\\
	\noindent ($\partial\Omega\neq \emptyset$)   If there exists a constant $C>0$ such that for all $u\in \Gamma^{\infty}_c(\bigwedge^kT^*\Omega)$,
	\begin{align*}
		\int_\Omega |u|^2\,d\nu\leq C\int_\Omega \big(|d u|^2+ |d^*u|^2\big)\,d\nu,
	\end{align*}
	then we say that the {\em Poincar\'e inequality holds for the case $\partial\Omega\neq\emptyset$ (PID*)}.
	
	\noindent($\partial \Omega=\emptyset$)  If there exists a constant $C>0$ such that for all  $u\in \Gamma^{\infty}_c(\bigwedge^kT^*M)\bigcap (\widetilde{\mathcal{H}}^k(M))^{\perp}$,
	\begin{align*}
		\int_M |u|^2\,d\nu\leq C\int_M \big(|d u|^2+ |d^*u|^2\big)\,d\nu,
	\end{align*}
	then we say that the {\em Poincar\'e inequality holds for the case $\partial\Omega=\emptyset$ (PIE*)}.

Next we introduce the following {\em Poincar\'e inequalities for the measure $\mu$}, depending on the boundary conditions:
	
\noindent($\partial\Omega \neq \emptyset$) If there exists a constant $C>0$ such that for all $u\in \Gamma^{\infty}_c(\bigwedge^kT^*\Omega)$,
	\begin{align}\label{eq:PI*}
		\int_\Omega |u|^2\,d\mu\leq C\int_\Omega \big(|d u|^2+ |d^*u|^2\big)\,d\nu,
	\end{align}
	then we say that the {\em Poincar\'e inequality holds for the measure $\mu$ on $k$-forms and the case $\partial\Omega\neq\emptyset$ (MPID*)}.\\
\noindent($\partial\Omega=\emptyset$) If there exists a constant $C>0$ such that for all $u\in \Gamma^{\infty}_c(\bigwedge^kT^*M)\bigcap (\widetilde{\mathcal{H}}^k(M))^{\perp}$,
	\begin{align}\label{eq:PIN*}
		\int_M |u|^2\,d\mu\leq C\int_M \big(|d u|^2+ |d^*u|^2\big)\,d\nu,
	\end{align}
		then we say that the {\em Poincar\'e inequality holds for the measure $\mu$ on $k$-forms and the case $\partial\Omega=\emptyset$ (MPIE*)}.

(MPID*) (resp. (MPIE*))  implies that each equivalence class $u \in W^{1,2}_0(\bigwedge^k\Omega)$ (resp. $u \in W^{1,2}(\bigwedge^kM)\bigcap (\widetilde{\mathcal{H}}^k(M))^{\perp}$) contains a unique (in the $L^2\big(\bigwedge^k\Omega,\mu\big)$ sense) member $\overline{u}$ that belongs to $L^{2}\big(\bigwedge^k\Omega, \mu\big)$ (resp. $L^{2}\big(\bigwedge^kM, \mu\big)$) and satisfies both conditions below:
\begin{enumerate}
	\item[(1)] There exists a sequence $\left\{u_m\right\}$ in $\Gamma_{c}^{\infty}(\bigwedge^k\Omega)$ (resp. $\Gamma_{c}^{\infty}(\bigwedge^kM)\bigcap (\widetilde{\mathcal{H}}^k(M))^{\perp}$) such that $u_m \rightarrow \overline{u}$ in $W^{1,2}_0(\bigwedge^k\Omega)$ (resp. $W^{1,2}(\bigwedge^kM)\bigcap (\widetilde{\mathcal{H}}^k(M))^{\perp}$) and $u_m \rightarrow \overline{u}$ in $L^{2}(\bigwedge^k\Omega, \mu)$ (resp. $L^{2}(\bigwedge^kM, \mu)$);
	\item[(2)] $\overline{u}$ satisfies the inequality in \eqref{eq:PI*} (resp. \eqref{eq:PIN*}).
\end{enumerate}

	We call $\overline{u}$ the {\em $L^2\big(\bigwedge^k\Omega,\mu\big)$-representative} of $u$ (resp. {\em $L^2\big(\bigwedge^kM,\mu\big)$-representative} of $u$).  Assume  $\mu$ satisfies (MPID*) (resp. (MPIE*)) and define a mapping $\iota_D: W^{1,2}_0(\bigwedge^kT^*\Omega) \rightarrow L^2\big(\bigwedge^kT^*\Omega,\mu\big)$ (resp. $\iota_W: W^{1,2}(\bigwedge^kT^*M) W^{1,2}_0(\bigwedge^k\Omega)\rightarrow L^2\big(\bigwedge^kT^*M,\mu\big)$) by
$$
\iota_D(u)=\overline{u}\qquad (\text{resp.}\,\,\iota_W(u)=\overline{u}).
$$
We consider the subspace $\mathcal{N}_D^k$ of $W^{1,2}_0(\bigwedge^kT^*\Omega)$ defined as
$$
\mathcal{N}_D^k:=\left\{u \in W^{1,2}_0\Big({\bigwedge}^kT^*\Omega\Big):\|\iota_D(u)\|_{\mu}=0\right\}.$$
Let $({\mathcal{N}_D^k})^{\perp}$ be the orthogonal complement of $\mathcal{N}_D^k$ in $W^{1,2}_0(\bigwedge^kT^*\Omega)$.  Similarly, we can define $\mathcal{N}_W^k$ and $({\mathcal{N}_W^k})^{\perp}$. 

We consider a nonnegative bilinear form $\mathcal{E}_D^k(\cdot, \cdot)$ (resp. $\mathcal{E}_W^k(\cdot, \cdot)$) on $L^2\big(\bigwedge^kT^*\Omega,\mu\big)$ (resp. $L^2\big(\bigwedge^kT^*M,\mu\big)$ ) defined as
\begin{align}\label{eq(1.1*)}
	\mathcal{E}_D^k(u, v):=\int_\Omega \langle d u,d v\rangle + \langle d^* u,d^* v\rangle \,d\nu\quad\Big(\text{resp. } \mathcal{E}_W^k(u, v):=\int_M \langle d u,d v\rangle + \langle d^* u,d^* v\rangle \,d\nu\Big )\nonumber\\
\end{align}
with \textit{domain} $\operatorname{dom}(\mathcal{E}_D^k)=({\mathcal{N}_D^k})^{\perp}$ (resp. $\operatorname{dom}(\mathcal{E}_W^k)=({\mathcal{N}_W^k})^{\perp}$). Let 
 $\mathcal{E}_E^k(\cdot, \cdot)$ be a nonnegative bilinear form on  $L^2\big(\bigwedge^kT^*M,\mu\big)$ defined as
\begin{align}\label{eq:new1}
	\mathcal{E}_E^k(u, v):=\int_M \langle d u,d v\rangle + \langle d^* u,d^* v\rangle \,d\nu
\end{align}
and let $\operatorname{dom}(\mathcal{E}_E^k):=  	\widetilde{\mathcal{H}}^k(M) \oplus \operatorname{dom}(\mathcal{E}_W^k)$.
Let 
\begin{align}\label{eq(2.1*)}
	\mathcal{E}_*^{D,k}(u, v):=&\int_{\Omega} \big(\langle d u, d v\rangle + \langle d^* u, d^* v\rangle  \big)\,d\nu+\int_{\Omega} \langle u, v\rangle  \,d \mu\quad\text{and}\nonumber\\
	\mathcal{E}_*^{E,k}(u, v):=&\int_{M} \big(\langle d u, d v\rangle + \langle d^* u, d^* v\rangle  \big)\,d\nu+\int_{M} \langle u, v\rangle  \,d \mu
\end{align}
be nonnegative bilinear forms in $L^2\big(\bigwedge^kT^*\Omega,\mu\big)$ and $L^2\big(\bigwedge^kT^*M,\mu\big)$, respectively.
Then $\mathcal{E}_*^{D,k}(\cdot, \cdot)$ (resp. $\mathcal{E}_*^{E,k}(\cdot, \cdot)$) is an inner product on $\operatorname{dom}(\mathcal{E}_D^k)$ (resp. $\operatorname{dom}(\mathcal{E}_E^k)$). 

To prove Proposition \ref{prop:2.3*}, we need the following lemma. It can be proved by using Lemma \ref{lem:2.2} and partition of unity; we omit the proof.

\begin{lem}\label{lem:2.2*} Let $M$ be a compact Riemannian $n$-manifold and let $\Omega\subseteq M$ be an open subset. Let $\mu$ be a positive finite Borel measure on $M$ such that ${\rm supp}(\mu) \subseteq \overline{\Omega}$ and $\mu(\Omega)>0$. Then $\Gamma^\infty_c(\bigwedge^kT^*\Omega)$ is dense in $L^2\big(\bigwedge^kT^*\Omega,\mu\big)$.
\end{lem}

\begin{prop}\label{prop:2.3*}
	Let $M$, $\Omega$, and $\mu$ be as in Lemma \ref{lem:2.2*}. Let $\mathcal{E}_D^k$, $\mathcal{E}_E^k$, $\mathcal{E}_*^{D,k}$, and $\mathcal{E}_*^{E,k}$ be the quadratic forms defined as in \eqref{eq(1.1*)} and (\ref{eq(2.1*)}). 
	\begin{enumerate}
		\item[(a)] Assume  that $\partial\Omega\neq \emptyset$ and (MPID*) holds. Then $\operatorname{dom}(\mathcal{E}_D^k)$ is dense in $L^2\big(\bigwedge^kT^*\Omega,\mu\big)$. Moreover,  $\left(\mathcal{E}_*^{D,k}, \operatorname{dom}(\mathcal{E}_D^k)\right)$ is a Hilbert space.
		\item[(b)] Assume  that $\partial\Omega= \emptyset$, and (MPIE*) holds. Then $\operatorname{dom}(\mathcal{E}_E^k)$ is dense in $L^2\big(\bigwedge^kT^*M,\mu\big)$. Moreover,  $\left(\mathcal{E}_*^{E,k}, \operatorname{dom}(\mathcal{E}_E^k)\right)$ is a Hilbert space.
		\end{enumerate}
\end{prop}
\begin{proof} The proof of this proposition is similar to that of Proposition \ref{prop:2.3} and is omitted.
\end{proof}

If $\mu$ satisfies (MPID*), Proposition \ref{prop:2.3*} implies that the quadratic form $(\mathcal{E}_D^k, \operatorname{dom}(\mathcal{E}_D^k))$ is closed on $L^2\big(\bigwedge^kT^*\Omega,\mu\big)$.  Hence it follows from standard theory that there exists a nonnegative self-adjoint operator $-\Delta^{D,k}_\mu$ on $L^2\big(\bigwedge^kT^*\Omega,\mu\big)$ with $\operatorname{dom}(\Delta^{D,k}_\mu) \subseteq \operatorname{dom}\Big({(\Delta^{D,k}_\mu)}^{1 / 2}\Big)=\operatorname{dom}(\mathcal{E}_D^k)$, such that
$$
\mathcal{E}_D^k(u, v)=\left\langle-{(\Delta^{D,k}_\mu)}^{1 / 2} u, -{(\Delta^{D,k}_\mu)}^{1 / 2} v\right\rangle_{\mu} \quad \text { for all } u, v \in \operatorname{dom}(\mathcal{E}_D^k).
$$
Moreover, $u \in \operatorname{dom}(\Delta^{D,k}_\mu)$ if and only if $u \in \operatorname{dom}(\mathcal{E}_D^k)$ and there exists $f \in L^2\big(\bigwedge^kT^*\Omega,\mu\big)$ such that 
\begin{align}\label{no1*}
	\mathcal{E}_D^k(u, v)=\langle f, v\rangle_{\mu} \quad \text{for all}\,\, v \in \operatorname{dom}(\mathcal{E}_D^k)
\end{align}
(see, e.g., \cite{Kigami_2001}). 	We call $\Delta^{D,k}_\mu$ the {\em Dirichlet Laplacian or Kre\u{\i}n-Feller operator on $k$-forms with respect to} $\mu$. Similarly, we can define $\Delta^{E,k}_\mu$.  If no confusion is possible, we will denote  $\Delta_{\mu}^{D,k}$ and $\Delta_{\mu}^{E,k}$ simply by $\Delta_{\mu}^k$;  we also denote  $\operatorname{dom}(\mathcal{E}_D^k)$ and $\operatorname{dom}(\mathcal{E}_E^k)$ simply by $\operatorname{dom}(\mathcal{E}^k)$. Note that for all $u \in \operatorname{dom}(\Delta^{k}_\mu)$ and $v \in \operatorname{dom}(\mathcal{E}^k)$,
\begin{align}\label{eq(2.2*)}
	\int_{\Omega} \big(\langle du, dv\rangle +\langle d^*u, d^*v\rangle  \big)\, d\nu=\mathcal{E}^k(u, v)=\langle -\Delta^{k}_\mu u, v\rangle_{\mu}.
\end{align}

The proofs of Proposition \ref{prop:2.4*} and Theorem \ref{thm:2.5*} are similar to those of \cite[Proposition 2.2]{Hu-Lau-Ngai_2006} and \cite[Theorem 2.3]{Hu-Lau-Ngai_2006}, respectively, and are omitted.
\begin{prop}\label{prop:2.4*}
	Let $M$, $\Omega$, and $\mu$ be defined as in Lemma \ref{lem:2.2*}.
	\begin{enumerate}
		\item[(a)] 	Assume  that $\partial\Omega\neq \emptyset$ and $\mu$ satisfies (MPID*). For $u \in \operatorname{dom}(\mathcal{E}_D^k)$ and $f \in L^{2}(\bigwedge^kT^*\Omega, \mu)$, the following conditions are equivalent:
		\begin{enumerate}
			\item[(i)] $u \in \operatorname{dom}(-\Delta^{D,k}_\mu)$ and $-\Delta^{D,k}_\mu u=f$;
			\item[(ii)] $-\Delta^k u=f d \mu$ in the sense of distribution; that is,  for any $v \in \Gamma^\infty_c(\bigwedge^kT^*\Omega)$,
			\begin{align}\label{eq(2.3*)}
				\int_{\Omega}\big( \langle du, dv\rangle +\langle d^*u, d^*v\rangle \big) \, d\nu=\int_{\Omega} \langle v, f\rangle \, d \mu.
			\end{align}	
		\end{enumerate}
		\item[(b)]	Assume that  $\partial\Omega= \emptyset$ and $\mu$ satisfies (MPIE*). For $u \in \operatorname{dom}(\mathcal{E}_E^k)$ and $f \in L^{2}(\bigwedge^kT^*M, \mu)$, the following conditions are equivalent:
		\begin{enumerate}
			\item[(i)] $u \in \operatorname{dom}(-\Delta^{E,k}_\mu)$ and $-\Delta^{E,k}_\mu u=f$;
		\item[(ii)] $-\Delta^k u=f d \mu$ in the sense of distribution; that is,  for any $v \in \Gamma^\infty_c(\bigwedge^kT^*M)$,
		\begin{align}\label{eq(2.3*)}
			\int_{M}\big( \langle du, dv\rangle +\langle d^*u, d^*v\rangle \big) \, d\nu=\int_{M} \langle v, f\rangle \, d \mu.
		\end{align}	
		\end{enumerate}
		\end{enumerate}
\end{prop}

\begin{thm}\label{thm:2.5*}
	Let $n$, $M$, $\Omega$, and $\mu$ be defined as in Proposition \ref{prop:2.3*}.
\begin{enumerate}
	\item[(a)] Assume that  $\partial\Omega\neq \emptyset$ and $\mu$ satisfies {\rm(MPID*)}.
	Then for any $f \in L^{2}(\bigwedge^kT^*\Omega, \mu)$, there exists a
	unique $u \in \operatorname{dom}\left(\Delta_{\mu}^{D,k}\right)$  such that $\Delta_{\mu}^{D,k} u=f.$ The operator
	$$(\Delta_{\mu}^{D,k})^{-1}: L^{2}({\bigwedge}^kT^*\Omega, \mu) \rightarrow \operatorname{dom}(\Delta_{\mu}^{D,k}),\quad f\mapsto u, $$
	is bounded and has norm at most $C$, the constant in \eqref{eq:PI*} .
	\item[(b)] Assume that  $\partial\Omega= \emptyset$ and $\mu$ satisfies {\rm(MPIE*)}.
		Then for any $f \in L^{2}(\bigwedge^kT^*M, \mu)$, there exists a
		unique $u \in \operatorname{dom}\left(\Delta_{\mu}^{E,k}\right)$  such that $\Delta_{\mu}^{E,k} u=f.$ The operator
		$$(\Delta_{\mu}^{E,k})^{-1}: L^{2}({\bigwedge}^kT^*M, \mu) \rightarrow \operatorname{dom}(\Delta_{\mu}^{E,k}),\quad f\mapsto u, $$
		is bounded and has norm at most $C$, the constant in  \eqref{eq:PIN*}.
\end{enumerate}
\end{thm}

\subsection{Proof of Hodge theorem on space of  $k$-forms} \label{S:L*}
Let $M$ be a compact Riemannian $n$-manifold, and $\mu$ be a positive finite regular Borel measure on $M$ with compact support.

Let $\mathcal{B}^k:=\big\{u \in \Gamma^\infty_c(\bigwedge^kT^*M):\|u\|_{W^{1,2}_0(\bigwedge^kT^*M)} \leq 1\big\}$. Let $U\subseteq M$ be an open set. Let $\mathcal{B}^k_U:=\{u \in \Gamma^\infty_c({\bigwedge}^kT^*U):\|u\|_{W^{1,2}_0({\bigwedge}^kT^*U)} \leq 1\}$ and $\widetilde{\mathcal{B}}_U:=\{u \in C^\infty_c(U):\|u\|_{W^{1,2}_0(U)} \leq 1\}$.   The following theorem generalizes the compact embedding theorem of Maz'ja.
\begin{thm}\label{thm:3.11*}
	Let $M$ be a compact Riemannian $n$-manifold, $\Omega\subseteq M$ be an open subset, and $\mu$ be a positive finite Borel measure on $M$ with ${\rm supp}(\mu)\subseteq \overline{\Omega}$ and $\mu(\Omega)>0$.  For $q>2$, the unit ball $\mathcal{B}^k$
	is relatively compact in $L^{q}(\bigwedge^kT^*M, \mu)$ if and only if
	\begin{align}
		&\lim _{\delta \rightarrow 0^{+}} \sup _{w\in M ; r \in(0, \delta)} r^{1-n / 2} \mu\left(B^M(w,r)\right)^{1 / q}=0 \quad \text {\rm for } n>2,  \label{eq:3.18*} 
	\end{align}
	and
	\begin{align}
		&\lim _{\delta \rightarrow 0^{+}} \sup _{w \in M ; r \in(0, \delta)}|\ln r|^{1 / 2} \mu\left(B^M(w,r)\right)^{1 / q}=0 \quad \text {\rm for } n=2\label{eq:3.19*}.
	\end{align}
\end{thm}

\begin{proof}\,Assume that (\ref{eq:3.18*}) and (\ref{eq:3.19*}) hold. We will prove that for $q>2$, the unit ball $\mathcal{B}^k$
	is relatively compact in $L^{q}(\bigwedge^kT^*M, \mu)$. 
	
	\noindent{\em Step 1.} Let $(U,\varphi)$ be a regular coordinate chart on $\supp(\mu)$. Let $(\omega_j)_{j=1}^{\infty}$ be a bounded sequence in $\mathcal{B}^k_U$. 
	By \eqref{eq:1.1*}, for each $j$, we write
	\begin{align*}
		\omega_j=\sum_{I\in  \mathcal{I}_{k,n} } \omega_j^Idx^I,
	\end{align*}
	where the coefficients $\omega_j^I$ are smooth real-valued functions on $U$. Hence
	\begin{align*}
		\|\omega_j\|_{W_0^{1,2}({\bigwedge}^kT^*U)}^2&=\int_U|\omega_j|^2\,d\nu+\int_U|\nabla_U\omega_j|^2\,d\nu 
		=\int_U\Big(\sum_{I\in \mathcal{I}_{k,n}}|\omega_j^I|^2\,d\nu+\int_U\sum_{I,l}\bigg|\frac{\partial\omega_j^I}{\partial x^l}\bigg|^2\Big)\,d\nu \nonumber\\
		&=\sum_{I\in \mathcal{I}_{k,n}}\int_U\Big(|\omega_j^I|^2+\sum_{l}\Big|\frac{\partial\omega_j^I}{\partial x^l}\Big|^2\Big)\,d\nu =\sum_{I\in \mathcal{I}_{k,n}}\|\omega_j^I\|_{W_0^{1,2}(U)}^2.
	\end{align*}
	As $\|\omega_j\|_{W_0^{1,2}({\bigwedge}^kT^*U)}\leq 1$, we have $\sum_{I\in \mathcal{I}_{k,n}}\|\omega_j^I\|_{W_0^{1,2}(U)}^2\leq 1$. Hence for each $I$, $\|\omega_j^I\|_{W_0^{1,2}(U)}^2\leq 1$.  By \cite[Lemma 18.6]{Tu_2011}, $\omega_j^I$ is smooth, and by definition, $\omega_j^I$ has compact support. Thus $\omega_j^I\in C^\infty_c(U)$ for each $I$.  For each $I$, $(\omega_j^I)_{j=1}^{\infty}$ is a bounded sequence in  $\widetilde{\mathcal{B}}_U$.  If $M$ is replaced by $U$, then \eqref{eq:3.18*} and \eqref{eq:3.19*} continue to hold. 
  Hence $\widetilde{\mathcal{B}}_U$ is relatively compact in $L^{q}(U, \mu)$ by using Theorem \ref{thm:3.11}.  Thus there exists $\omega_I\in L^q(U,\mu)$ such that for each $I$,
	$
	\lim_{j\rightarrow\infty}\|\omega_j^I-\omega_I\|_{L^{q}(U, \mu)}=0.
	$
	Let $\omega:=\sum_{I\in  \mathcal{I}_{k,n} } \omega_Idx^I$. Then $\omega\in L^{q}({\bigwedge}^kT^*U, \mu)$ and
	\begin{align*}
		\lim_{j\rightarrow\infty}\|\omega_j-\omega\|_{\mu}=\lim_{j\rightarrow\infty}\sum_{I\in \mathcal{I}_{k,n}}\|\omega_j^I-\omega_I\|_{L^{q}(U, \mu)}=0.
	\end{align*}
	Hence $\mathcal{B}^k_U$ is relatively compact in $L^{q}({\bigwedge}^kT^*U, \mu)$.

	\noindent{\em Step 2}. By the compactness of $\supp(\mu)$, we can choose a $C^\infty$ atlas $\{(W_\alpha,\varphi_\alpha)\}_{\alpha=1}^N$ satisfying $\supp(\mu)\subseteq \cup_{\alpha=1}^N W_\alpha$.  By relabeling the $W_{\alpha}$, we can
	choose geodesic balls $F_{\alpha}:=B^M(p_{\alpha},\epsilon_{\alpha})\subseteq W_{\alpha}$ satisfying
	$\supp(\mu)\subseteq\bigcup_{\alpha=1}^N F_\alpha$ and $\overline{F}_{\alpha}\subseteq W_{\alpha}$.
	Then $\{(F_\alpha,\varphi_\alpha|_{F_{\alpha}})\}_{\alpha=1}^N$ is a regular atlas on $\supp(\mu)$.  Next, by using a method similar to that in Step 5 of the proof of Theorem \ref{thm:3.11}, 
	one can show that  $\mathcal{B}^k$ is relatively compact in $L^q(\bigwedge^kT^*M, \mu)$; we omit the details.

	Conversely, assume that for $q>2,$ the unit ball $\mathcal{B}^k$ is relatively compact in $L^q(\bigwedge^kT^*M, \mu)$. We will show that \eqref{eq:3.18*} and \eqref{eq:3.19*} hold.		Using the compactness of $M$, we may select a finite
	system $\{(V_i,\varphi_i)\}_{i=1}^N$ satisfying $M\subseteq \cup_{i=1}^NV_i$.  By relabeling the $V_i$, we can choose a regular atlas $\{(U_i,\varphi_i)\}_{i=1}^N$ satisfying $\overline{U}_i\subseteq V_i$ and $M\subseteq \cup_{i=1}^N U_i$. Fix $i$. Let $(u_m)$ be a bounded sequence in $\mathcal{B}^k_{U_i}$. Since $\mathcal{B}^k$ is relatively compact in $L^q(\bigwedge^kT^*M, \mu)$, there exists $u\in L^q(\bigwedge^kT^*M, \mu)$ such that 
	\begin{align}\label{eq:cont*}
		\lim_{m \rightarrow \infty}\|u_m-u\|_{\mu}=0.
	\end{align}
	Since $L^q({\bigwedge}^kT^*U_i, \mu)$ is complete, we can write $u=v_1+v_2$, where ${\rm supp}(v_1)\subseteq U_i$ and ${\rm supp}(v_2)\subseteq \partial U_i$. Suppose 
	$\int_{\partial U_i}|v_2|^qd\mu\neq 0.$
	Then, since $\int_{\partial U_i}|u_m|^qd\mu= 0$, we have
	$\lim_{m \rightarrow \infty}\|u_m\\-v_2\|_{\mu}\neq 0.$
	This contradicts \eqref{eq:cont*}. Hence $\int_{\partial U_i}|v_2|^qd\mu= 0.$ Thus $u=v_1$ $\mu$-a.e. in $\overline{U}_i$. Hence $\mathcal{B}^k_{U_i}$ is relatively compact in $L^q({\bigwedge}^kT^*U_i, \mu)$. For each $I$, let $\omega_j^I$ be a bounded sequence in $\widetilde{\mathcal{B}}_{U_i}$, i.e., $\omega_j^I\in C^\infty_c(U_i)$, and $\|\omega_j^I\|_{W_0^{1,2}(U_i)}\leq 1$.  For each $j$, we let $\omega_j=\sum_{I\in \mathcal{I}_{k,n}} \omega_j^Idx^I$. Then  $\omega_j\in \Gamma^\infty_c({\bigwedge}^kT^*U_i)$ and
	\begin{align*}
		\|\omega_j\|_{W_0^{1,2}({\bigwedge}^kT^*U_i)}=\sum_{I\in \mathcal{I}_{k,n}}\|\omega_j^I\|_{W_0^{1,2}(U_i)}\leq \sum_{I\in \mathcal{I}_{k,n}}1\leq C \quad\text{for each}\,\, j.
	\end{align*}
	Let $$\overline{\mathcal{B}^k_{i}}:=\Big\{u \in \Gamma^\infty_c\Big({\bigwedge}^kT^*U_i\Big):\|u\|_{W^{1,2}_0({\bigwedge}^kT^*U_i)} \leq C\Big\}.$$ Then $(\omega_j)$ is a bounded sequence in $\overline{\mathcal{B}^k_{i}}$. Since $\mathcal{B}^k_{U_i}$ is relatively compact in $L^q({\bigwedge}^kT^*U_i , \mu)$, $\overline{\mathcal{B}^k_{i}}$ is relatively compact in $L^q({\bigwedge}^kT^*U_i, \mu)$. Hence there exists $\omega\in L^q({\bigwedge}^kT^*U_i, \mu)$ such that 
	\begin{align*}
		\lim_{j\rightarrow\infty}\|\omega_j-\omega\|_{\mu}=0.
	\end{align*}
	Write $\omega:=\sum_{I\in \mathcal{I}_{k,n}}\omega_Idx^I$. Then $\omega_I\in L^{q}(U_i, \mu)$ for each $I$. Hence for each $I$,
	\begin{align*}
		\lim_{j\rightarrow\infty}\|\omega_j^I-\omega_I\|_{L^{q}(U_i, \mu)}=0.
	\end{align*}
	It follows that $\widetilde{\mathcal{B}}_{U_i}$ is relatively compact in $L^q(U_i, \mu)$, and thus
	\begin{align}
		&\lim _{\delta \rightarrow 0^{+}} \sup _{w\in U_i ; r \in(0, \delta)} r^{1-n / 2} \mu\left(B^M(w,r)\right)^{1 / q}=0 \quad \text {\rm for } n>2,\qquad\text{and}\label{1*}\\
		&\lim _{\delta \rightarrow 0^{+}} \sup _{w \in U_i ; r \in(0, \delta)}|\ln r|^{1 / 2} \mu\left(B^M(w,r)\right)^{1 / q}=0 \quad \text {\rm for } n=2.\label{2*}
	\end{align}
	Combining \eqref{1*}, \eqref{2*}, and the fact that $\{U_i\}_{i=1}^N$ is a finite open cover of $\supp(\mu)$,  we see that \eqref{eq:3.18*} and \eqref{eq:3.19*} hold.
\end{proof}

The proof of Theorem \ref{thm:3.12*} uses Theorem \ref{thm:3.11*} and Lemma \ref{lem:3.1}(b); it is similar to that of \cite[Theorem 3.2]{Hu-Lau-Ngai_2006} and so is omitted.
\begin{thm}\label{thm:3.12*}Let $M$ be a compact Riemannian $n$-manifold. Let $n \geq 2$ and $2<q<\infty$, and let $\mu$ be a finite positive Borel measure on $M$ with compact support.  
	\begin{enumerate}
		\item[(a)] If $\underline{\operatorname{dim}}_{\infty}(\mu)>q(n-2) / 2$, then $\mathcal{B}^k$ is relatively compact in $L^{q}(\bigwedge^kT^*M, \mu)$.
		\item[(b)] If $\underline{\operatorname{dim}}_{\infty}(\mu)<q(n-2) / 2$, then $\mathcal{B}^k$ is not relatively compact in $L^{q}(\bigwedge^kT^*M, \mu)$.
	\end{enumerate}
\end{thm}

We let $\Gamma^C(\bigwedge^kT^*\overline{\Omega})$ be the subspace of $\Gamma^C(\bigwedge^kT^*M)$ consisting of all  $k$-forms in $\Gamma^C(\bigwedge^kT^*\Omega)$ that are bounded and uniformly continuous on $\Omega$. 
To prove Theorem \ref{thm:1.1*}, we need the following theorem. 
\begin{thm}\label{lem:3.13*}
	Let $M$ be a compact Riemannian $1$-manifold, and let $\Omega \subseteq M$ be an open set. Then $W^{1,2}_0(\bigwedge^kT^*\Omega)$ is compactly embedded in $\Gamma^C(\bigwedge^kT^*\overline{\Omega})$. In particular, if $\partial\Omega=\emptyset$, then $W^{1,2}(\bigwedge^kT^*M)$ is compactly embedded in $\Gamma^C(\bigwedge^kT^*M)$.
\end{thm}
The proof of this theorem is similar to that of Theorem \ref{thm:3.11*} and is omitted. It uses Theorem \ref{lem:3.13} and partition of unity.

\begin{proof}[Proof of Theorem \ref{thm:1.1*}]\,
	(a) For the case $n=1$, we have by Theorem \ref{lem:3.13*} that  $W^{1,2}_0(\bigwedge^kT^*\Omega)$ is compactly embedded in $L^{2}(\bigwedge^kT^*\Omega, \mu)$, and thus (MPID*) holds. Moreover, the embedding $\operatorname{dom}(\mathcal{E}_D^k) \hookrightarrow L^{2}(\bigwedge^kT^*\Omega, \mu)$ is compact.  
	Using Theorem \ref{thm:3.12*}, we can prove that the hypotheses of Theorem \ref{thm:1.1*} are satisfied if $n\geq 2$. The proof is similar to that of \cite[Theorem 1.1]{Hu-Lau-Ngai_2006} and is omitted.
	
		(b) For the case $n=1$, by Theorem \ref{lem:3.13*}, we can show that (MPIE*) holds and that  the embedding $\operatorname{dom}(\mathcal{E}_E^k) \hookrightarrow L^{2}(\bigwedge^kT^*M, \mu)$ is compact.	Similarly, using Theorem \ref{thm:3.12*}, we can prove that the hypotheses of Theorem \ref{thm:1.1*} are satisfied if $n\geq 2$. The proof is similar to that of \cite[Theorem 1.1]{Hu-Lau-Ngai_2006} and is omitted.
\end{proof}

\begin{proof}[Proof of Theorem \ref{thm:1.2*}]\, This is a direct consequence of Theorem \ref{thm:1.1*} and \cite[Theorem B.1.13]{Kigami_2001}; we omit the proof.
\end{proof}

\subsection{A generalization of the classical Hodge theorem on $k$-forms} \label{S:H}

In this subsection, we show that the classical Hodge theorem holds when $\Omega=M$ and the measure $\mu$ is absolutely continuous with a positive and bounded density.

\begin{lem}\label{prop:5.1*}
	Assume the same hypotheses of Theorem \ref{thm:1.3*}. Then $\mathcal{N}_W^k=\{0\}$ and hence $({\mathcal{N}_W^k})^\perp=W^{1,2}\big({\bigwedge}^kT^*M\big)\bigcap (\widetilde{\mathcal{H}}^k(M))^{\perp}$.
\end{lem} 
\begin{proof}
	By the definitions of $\iota_W$ and $\mathcal{N}_W^k$, for any $u\in \mathcal{N}_W^k$,
	$
	\int_M \langle \iota_W(u),\iota_E(u)\rangle \,d\mu=0.
	$
	By assumption,	$d\mu=\rho\, d\nu,$ where $\rho>0$ is a density function. Thus $\langle \iota_W(u),\iota_W(u)\rangle =0$ $\nu$-a.e., and hence $\iota_W(u)=0$ $\nu$-a.e. As $\iota_W(u)=u$ $\nu$-a.e., we have
	$u=0$ $\nu$-a.e. This completes the proof.	
\end{proof}
\begin{prop}\label{prop:1*}
	Assume the same hypotheses of Theorem \ref{thm:1.3*}. Then $\mathcal{H}^k_{\mu}(M)=\mathcal{H}^k(M)=\widetilde{\mathcal{H}}^k(M)$.
\end{prop}
\begin{proof} By Lemma \ref{prop:5.1*}, we have $\mathcal{N}_W^k=\{0\}$ and  $({\mathcal{N}_W^k})^{\perp}=W^{1,2}(\bigwedge^kT^*M)\bigcap (\widetilde{\mathcal{H}}^k(M))^{\perp}$. Hence $\operatorname{dom}(\mathcal{E}_E^k)=W^{1,2}(\bigwedge^kT^*M)$. By \eqref{eq(2.2*)}, for $\omega\in \operatorname{dom}(\Delta_{\mu}^{E,k})$ and for any $u\in \operatorname{dom}(\mathcal{E}_E^k)$, we have
	\begin{align}\label{eq:5.2*}
		\langle -\Delta_{\mu}^{E,k}\omega,u\rangle_{\mu}=\int_M \langle d\omega,du\rangle \,d\nu+\int_M \langle d^*\omega,d^*u\rangle \,d\nu
		=\langle -\Delta^k\omega,u\rangle_{L^2(\bigwedge^kT^*M)}.
	\end{align}
	\noindent{\em Step 1.} We first prove $\mathcal{H}^k_{\mu}(M)\subseteq \mathcal{H}^k(M)$.
	Let $\omega\in \mathcal{H}^k_{\mu}(M)$, i.e., $\omega\in \operatorname{dom}(\Delta_{\mu}^{E,k})$ and  $\Delta_{\mu}^{E,k}\omega=0$. Then $\omega\in \operatorname{dom}(\Delta^k)$. For any  $u\in \operatorname{dom}(\mathcal{E}_E^k)$, by \eqref{eq:5.2*} and the fact that $\Delta_{\mu}^{E,k}\omega=0$,   we have
	\begin{align*}
		\langle \Delta^k\omega,u\rangle_{L^2(\bigwedge^kT^*M)}=0.
	\end{align*}
	In particular,
	$
	\langle \Delta^k\omega,\omega\rangle_{L^2(\bigwedge^kT^*M)}= 0.
	$
	Let $\omega=\sum_{m=1}^\infty c_m\psi_m$, where $c_m\in\R$ and $\left\{\psi_{m}\right\}_{m=1}^{\infty}$ is an orthonormal basis of $L^2(\bigwedge^kT^*M)$. Then for all $u\in  \operatorname{dom}(\mathcal{E}_E^k)$, 
	\begin{align*}
		\langle -\Delta^k\omega,u\rangle_{L^2(\bigwedge^kT^*M)}= \Big\langle-\sum_{m=1}^\infty \lambda_mc_m\psi_m ,u\Big\rangle_{L^2(\bigwedge^kT^*M)}.
	\end{align*}
	Hence $\Delta^k\omega=\sum_{m=1}^\infty \lambda_mc_m\psi_m$ and therefore, as $\lambda_m\geq 0$ for all $m$, we have
	\begin{align*}
		0&=	\langle \Delta^k\omega,\omega\rangle_{L^2(\bigwedge^kT^*M)}=	\sum_{m=1}^\infty\lambda_mc_m^2\langle \psi_m,\psi_m\rangle_{L^2(\bigwedge^kT^*M)}=\sum_{m=1}^\infty \lambda_mc_m^2.
	\end{align*}
	Assume that the 0-eigenvalue space is $d$-dimensional. For $m\geq d+1$, we have $\sum_{m=d+1}^\infty  \lambda_mc_m^2\\=0$ and  $\lambda_m>0$ for all $m\geq d+1$. Hence $c_m=0$ for all $m\geq d+1$. For $m=1,\ldots,d$, we have $\omega=\sum_{m=1}^{d}c_m\psi_m$. Hence $\Delta^k\omega=0$, and thus $\omega\in \mathcal{H}^k(M)$. This proves that $\mathcal{H}^{k}_{\mu}(M)\subseteq\mathcal{H}^k(M)$.
	
	\noindent{\em Step 2.} Applying a similar argument as that in  Step 1 show that $\mathcal{H}^k(M)\subseteq\mathcal{H}^k_{\mu}(M)$, $\mathcal{H}^k(M)\subseteq \widetilde{\mathcal{H}}^k(M)$, and $\widetilde{\mathcal{H}}^k(M)\subseteq \mathcal{H}^k(M)$. 
	
	Combining the two steps above,  we have $\mathcal{H}^k_{\mu}(M)=\mathcal{H}^k(M)=\widetilde{\mathcal{H}}^k(M).$
\end{proof}

\begin{proof}[Proof of the Theorem \ref{thm:1.3*}] Combining Proposition \ref{prop:1*} and the Hodge theorem (see \cite[Theorem 7.4.4]{Morrey_1966}) completes the proof.
\end{proof}

The following example shows that  Theorem \ref{thm:1.3*} need not hold if $\mu$ is not absolutely continuous.

\begin{exam}\label{exam:1*}
	Let $\mathbb{T}^2=\mathbb{S}^1\times\mathbb{S}^1$ be the 2-torus, where $\mathbb{S}^1$ is the 1-sphere. Let $\mu$ be the Dirac measure on $\mathbb{T}^2$. Then $\mathcal{H}^1(\mathbb{T}^2)\neq\mathcal{H}^1_{\mu}(\mathbb{T}^2)$.
\end{exam}
\begin{proof}Let $\underline{\mathcal{H}}^1(\mathbb{T}^2):=\{\omega\in\Gamma^\infty(\bigwedge^1T^*\mathbb{T}^2):\Delta^k\omega=0\}$. Since $\Gamma^\infty(\bigwedge^1T^*\mathbb{T}^2)\subseteq W^{1,2}(\bigwedge^1T^*\mathbb{T}^2)$, we have
	\begin{align}\label{eq:h2}
		\underline{\mathcal{H}}^1(\mathbb{T}^2)\subseteq \mathcal{H}^1(\mathbb{T}^2).
	\end{align}
	By the classical Hodge theorem (see, e.g., \cite[Theorem 47]{Petersen_2016}), we have $\underline{\mathcal{H}}^1(\mathbb{T}^2)\cong H^1_{\rm dR}(\mathbb{T}^2;\R)$, where $\cong$ denotes vector space isomorphism and $H^1_{\rm dR}(\mathbb{T}^2;\R)$ is the first de Rham cohomology of $\mathbb{T}^2$. Since $H^1_{\rm dR}(\mathbb{T}^2;\R)\cong \R^2$ (see, e.g., \cite[Section 28]{Tu_2011}), we have $\dim (\underline{\mathcal{H}}^1(\mathbb{T}^2))=2$.
	Combining this with \eqref{eq:h2}, we have
	\begin{align}\label{eq:h1}
		\dim (\mathcal{H}^1(\mathbb{T}^2))\geq2.
	\end{align}
	Note that if $\omega=\eta$ $\mu$-a.e., then $\|\omega-\eta\|_\mu=0$. It follows that  $\dim (L^2(\bigwedge^1T^*\mathbb{T}^2,\mu))=1$. Since $\mathcal{H}^1_{\mu}(\mathbb{T}^2)$ is a subspace of $L^2(\bigwedge^1T^*\mathbb{T}^2,\mu)$, 	we have, by \eqref{eq:h1}, $\mathcal{H}^1(\mathbb{T}^2)\neq\mathcal{H}^1_{\mu}(\mathbb{T}^2)$.
\end{proof}

	\section{Self-similar and self-conformal measures}\label{S:self}
	\setcounter{equation}{0}
	Let $M$ be a complete Riemannian $n$-manifold. In this section, for an invariant measure $\mu$ defined by an iterated function system $\{S_{i}\}_{i=1}^{m}$ of contractions on $M$, we can further strengthen Theorems \ref{thm:1.1} and \ref{thm:1.2}, . For $\tau=(i_1,\ldots,i_k)$ and the invariant set $K$, we let $K_{\tau}:=S_{\tau}(K)$.
	
	We call $\left\{S_{i}\right\}_{i=1}^{m}:M\rightarrow M$ an iterated function system of {\em $bi$-Lipschitz contractions} if for each $i\in\{1, \ldots, m\}$, there exist $c_{i}, r_{i}$ with $0<c_{i} \leq r_{i}<1$ such that
	\begin{align}\label{eq:4.1}
		c_{i}d_M(x,y) \leq d_M(S_{i}(x),S_{i}(y)) \leq r_{i}d_M(x,y) \quad \text { for all } x, y \in M.
	\end{align}
	
	\begin{lem}\label{lem:4.1}
		Let $\mu$ be an invariant measure of an {\rm IFS} $\left\{S_{i}\right\}_{i=1}^{m}$ of bi-Lipschitz contractions on $M$. Suppose the attractor $K$ is not a singleton. Then $\mu$ is upper $s$-regular for some $s>0$, and hence $\underline{\operatorname{dim}}_{\infty}(\mu)>0$.
	\end{lem}
	The proof of this lemma is similar to that of \cite[Lemma 5.1]{Hu-Lau-Ngai_2006}.

	It follows from Lemmas \ref{lem:3.1} and \ref{lem:4.1} that on a complete Riemannian $2$-manifold $M$, a measure $\mu$ in Lemma \ref{lem:4.1} satisfies $\underline{\operatorname{dim}}_{\infty}(\mu)>0=n-2$. Hence by Theorem \ref{thm:1.1}, we have the following corollary.
	\begin{cor}\label{cor:4.2}
			Let $M$ be a complete Riemannian $2$-manifold. Let $\left\{S_{i}\right\}_{i=1}^{m}$ be an {\rm IFS} of bi-Lipschitz contractions on M defined as in (\ref{eq:4.1}). Let $\mu$ be an invariant measure, and let $\Omega$ be a bounded open subset of $M$ such that ${\rm supp}(\mu) \subseteq$ $\overline{\Omega}$. Then the embedding $\operatorname{dom}(\mathcal{E}_D) \hookrightarrow L^{2}(\Omega, \mu)$ is compact. Consequently, the conclusions of Theorem \ref{thm:1.2} hold for such a measure $\mu$.
	\end{cor}
	
	

	The following proposition is needed in Propositions \ref{prop:4.4} and \ref{prop:4.5}.
	\begin{prop}\label{thm:4.3}
		Let $M$ be a complete Riemannian $n$-manifold. Let $\left\{S_{i}\right\}_{i=1}^{m}$ be a {\rm CIFS}  on $M$ satisfying {\rm (OSC)}, and let $\mu$ be an associated self-conformal measure with probability weights $\left\{p_{i}\right\}_{i=1}^{m}$. Then $\mu(K_i\cap K_j)=0$ for any $i\neq j$. Moreover, $\mu(K_{\tau})=p_{\tau}\mu(K)=p_{\tau}$ for any word $\tau$.
	\end{prop}
	\begin{proof}\,
The first assertion follows by using  results in \cite{Ngai-Xu_2021,Patzschke_1997} and a  method in \cite{Fan-Lau_1999,Lau-Rao-Ye_2001}. 
			For the second assertion, we use the result that $\mu(K_i\cap K_j)=0$ for any $i\neq j$ to conclude that
			$
			\sum_{i} \mu(K_{i})=\mu\big(\bigcup_{i} K_{i}\big)=\mu(K)=1.
			$
			By the self-conformality of $\mu$, we have
			\begin{align*}
				\mu(K_{i})=p_{i}\mu\circ S_i^{-1}(K_i)+\sum_{j \neq i} p_{j}\mu\circ S_j^{-1}(K_i)=p_{i}+\sum_{j \neq i} p_{j} \mu\left(S_{j}^{-1}\left(K_{i}\right)\right)\geq p_{i}.
			\end{align*}
			Consequently,  $\mu(K_{i})=p_{i}$ for each $i$. Repeating the above procedure, we have
			$\mu\left(K_{\tau}\right)=p_{\tau}$ for any word $\tau$.
	\end{proof}
	
	
	\begin{prop}\label{prop:4.4}
	Let $M$ be a complete Riemannian $n$-manifold with ${\rm Ric}\geq (n-1)\xi$ for some $\xi\in\mathbb{R}$. Let $\left\{S_{i}\right\}_{i=1}^{m}$ be a {\rm CIFS}  on $M$ satisfying {\rm (OSC)}, and let $\mu$ be an associated self-conformal measure with probability weights $\left\{p_{i}\right\}_{i=1}^{m}$.  Then
		$$
		\underline{\operatorname{dim}}_{\infty}(\mu)\geq \min _{1 \leq i \leq m}\left\{\frac{\ln p_{i}}{\ln \|S'_i\|}\right\}.
		$$
	\end{prop}
	\begin{proof}Define
			\begin{align*}
				\Gamma_t:=\left\{\tau\in\Sigma_{*}: \|S'_{\tau}\| <t \leq \|S'_{\tau^-}\|\right\},\quad t\in(0, 1).
			\end{align*}
		Then $K=\bigcup_{\tau \in \Gamma_t} K_{\tau}$.
			Let $s:=\min _{1 \leq i \leq m}\left\{\ln p_{i} / \ln \|S'_{i}\|\right\}$. Then 
	$	s \leq (\ln p_{i})/(\ln \|S'_{i}\|)$ for all $i=1,\ldots,m$.
			Hence, for all $i=1,\ldots,m$, $\|S'_{i}\|^s\geq \ p_{i}$.
			Therefore, by Proposition \ref{thm:4.3},
			\begin{align}\label{eq:4.3}
				\mu\left(K_{\tau}\right)=p_{\tau}  \leq \|S'_\tau\|^{s}<t^{s}.
			\end{align}
			On the other hand, let $L:=\max_{i}\|S'_{i}\|\in (0, 1)$ and $\tau=(i_1,\ldots,i_k)$. Then
			\begin{align*}
				\|S'_{\tau^-}\|\leq \|S'_{i_1}\|\|S'_{i_2}\|\cdots \|S'_{i_{k-1}}\|
				\leq (\max_{i}\|S'_{i}\|)^{k-1}=L^{k-1}.
			\end{align*}
		Let $k$ be the largest integer such that $t\leq L^{k-1}$. Then
			\begin{align*}
				k\leq \frac{\ln t}{\ln L}+1=:C'.
			\end{align*}
			Hence the length of the word $\tau\in\Gamma_t$ has an upper bound $C'$ that depends on $t$. By (BDP) and the definition of $\|S'_{\tau}\|$, there exists a positive constant $\widetilde L$ such that 
				\begin{align}\label{eq:pp}
				\|S'_{\tau}\|\geq \widetilde L\,\|S'_{\tau^-}\|\geq \widetilde L\,t.
				\end{align}
				 Let $\Omega$ be defined as in \eqref{eq:j1.7}. 
			For each $p\in \Omega$, let $q$ be any point on $\partial B^M(p,\rho)$, where $0<\rho\leq {\rm diam}(\Omega)$ is a positive constant. By \eqref{eq:pp}, \cite[Lemma 2.2]{Patzschke_1997}, and the definition of $\Gamma_t$, there exist constants $a_2>a_1>0$ (independent of $\tau$) such that 
			\begin{align}
				&d_M\left(S_{\tau}(p), S_{\tau}(q)\right) \leq \widetilde{C}\left\|S_{\tau}'\right\| d_M(p, q)\leq a_2t,\label{eq:4.4}\\
				&d_M\left(S_{\tau}(p), S_{\tau}(q)\right) \geq \widetilde{C}^{-1}\left\|S_{\tau}'\right\| d_M(p, q)\geq a_1t,\label{eq:4.5}
			\end{align}
			where $ d(p, q)=\rho$ and $a_2t\leq c_1$, where $c_1$ is a positive constant.
			Let $\{\Omega_i\}$ be a family of disjoint open subsets of $M$ satisfying \eqref{eq:j1.7}. Using (\ref{eq:4.4}) and (\ref{eq:4.5}), we see that each $\Omega_{i}$ contains a ball of radius $a_1t$ and is contained in a ball of radius $a_2t$.
		 By the Bishop-Gromov volume comparison theorem (see, e.g., \cite[Lemma 36]{Petersen_2016}), for any $p\in M$ and $0<r<R\leq c_2$, we have 
		\begin{align}\label{eq:dou}
		{\rm Vol}(B^M(p,R))\leq c_3\Big(\frac{R}{r}\Big)^n{\rm Vol}(B^M(p,r)),
		\end{align}
	where $c_3$ is a positive constant.
	Hence $B^M(p,t)$ intersects at most $c_3(1+2a_2)^n/a_1^n=:C$ of the $\overline{\Omega}_i$.
			Since $K$ satisfies (OSC),  there exists a constant $C>0$ such that for each $p\in K$, the ball $B^M(p,t)$ intersects at most $C$ sets of the form $K_\tau$, where $\tau\in \Gamma_t$.
			Therefore, it follows from $(\ref{eq:4.3})$ that
			$$\mu(B^M(p,t))\leq C t^{s}, $$
		proving that $\mu$ is upper $s$-regular. Hence $\underline{\operatorname{dim}}_{\infty}(\mu) \geq s$ by Lemma \ref{lem:3.1}.

	\end{proof}
	
	\begin{prop}\label{prop:4.5}
			Let $M$ be a complete Riemannian $n$-manifold with ${\rm Ric}\geq (n-1)\xi$ for some $\xi\in\mathbb{R}$. Let $\left\{S_{i}\right\}_{i=1}^{m}$ be an {\rm IFS} of  contractive similitudes on $M$ satisfying {\rm (OSC)}, and let $\mu$ be the associated self-similar measure with probability weights $\left\{p_{i}\right\}_{i=1}^{m}.$ Then
		$$
		\underline{\operatorname{dim}}_{\infty}(\mu)\geq \min _{1 \leq i \leq m}\left\{\frac{\ln p_{i}}{\ln r_i}\right\},
		$$
		where $r_i$ is the contraction ratio of $S_i$.
	\end{prop}
	\begin{proof}\,Use a similar argument as that in Proposition \ref{prop:4.4}, and apply Proposition \ref{thm:4.3}.
	\end{proof}

The following proposition can be proved easily by modifying \cite[Proposition 5.6]{Hu-Lau-Ngai_2006}; we omit the proof.	
	\begin{prop}\label{prop:4.7}
		Let $M$ be a complete Riemannian $n$-manifold. Let $\left\{S_{i}\right\}_{i=1}^{m}$ be a {\rm CIFS} on $M$  and for all $x$, $y\in M$,
		\begin{align*}
			d_M(S_{i}(x), S_{i}(y))\leq \|S'_{i}\|d_M(x,y),
		\end{align*}
		and let $\mu$ be the associated self-conformal measure with probability weights $\left\{p_{i}\right\}_{i=1}^{m}.$ Then
		$$
		\overline{\operatorname{dim}}_{\infty}(\mu) \leq \max _{1 \leq i \leq m}\left\{\frac{\ln p_{i}}{\ln \|S'_{i}\|}\right\}.
		$$
	\end{prop}
		
	\begin{proof}[Proof of Theorem \ref{thm:4.9}]\,  For (a), the equivalence between $({\rm C3})$ and $({\rm C4})$ is stated in Lemma \ref{lem:3.1}.  Next we show that condition $({\rm C2})$ implies $({\rm C3})$. Since $\overline{A}:=\max _{1 \leq i \leq m}\{p_{i} {\|S'_i\|}^{2-n}\}<1$, it follows that
		\begin{align*}
			\frac{\ln p_{i}}{\ln \|S'_i\|}>n-2 \quad\text{for all}\,\, i=1,\ldots,m.
		\end{align*}
	Using Proposition \ref{prop:4.4}, we have
		$\underline{\operatorname{dim}}_{\infty}(\mu)\geq \min _{1 \leq i \leq m}\left\{\ln p_{i}/\ln \|S'_i\|\right\}>n-2$.
	For (c), the implication $({\rm C3})\Rightarrow({\rm C1})$ is shown in Theorem \ref{thm:1.1}. 
	\end{proof}

	\begin{proof}[Proof of Theorem \ref{thm:4.8}]\,By combining the methods in the proofs of \cite[Proposition 2]{Naimark-Solomyak_1995} and \cite[Theorem 1.4]{Hu-Lau-Ngai_2006}, and using Proposition \ref{thm:4.3}, we can show that condition $({\rm C1})$ implies $({\rm C2})$. 	Using Proposition \ref{prop:4.5}, the implication $({\rm C2})\Rightarrow({\rm C3})$ is proved.
		The rest of the proof of this theorem is similar to that of Theorem \ref{thm:4.9}. Note that we use $r_i$ instead of $\|S'_{i}\|$.	
\end{proof}

		\section{Appendix} \label{S:App}

		\setcounter{equation}{0}
			\renewcommand\theequation{A.\arabic{equation}}

		We use Maclaurin series to prove Lemma \ref{lem:3.4}.
		\begin{proof}[Proof of Lemma \ref{lem:3.4}]\, (a)
				Let
				\begin{align}\label{eq:j3.2}
					\cosh y:=\cosh^2\Big(\frac{\rho+\lambda}{R_1}\Big)-\sinh^2\Big(\frac{\rho+\lambda}{R_1}\Big)\cos\Big(\frac{4\delta}{\rho}\Big).
				\end{align}
				Using the Maclaurin series of $\cosh y$, we write
				\begin{align*}
					\cosh y=1+\frac{y^2}{2}+\frac{y^4}{24}\Big(1+\cdots+\frac{24}{(2n)!}y^{2n-4}+\cdots\Big)=:1+\frac{y^2}{2}+\frac{y^4}{24}\eta(y).
				\end{align*}
				Then
				$$\eta(y)=\frac{24(\cosh y-y^2/2-1)}{y^4},\quad y\neq 0.$$
				It is straightforward to verify by using differentiation that $\eta(y)$ is an increasing function on $(0, \infty)$. Define $\eta(0):=1$. It follows from direct verification that $\lim_{y\rightarrow 0}\eta(y)=1=\eta(0)$. Hence $\eta(y)$ is continuous on $[0,\infty)$, and thus there exist positive constants $c_1$ and $c_2$ such that $c_{1}\leq \eta(y)\leq c_{2}$ for $0\leq y \leq L<\infty$.
				
				We know that the following Maclaurin series converge on the indicated intervals:
				\begin{align*}
					\cosh\Big(\frac{\rho+\lambda}{R_1}\Big),\,\,	\sinh\Big(\frac{\rho+\lambda}{R_1}\Big)\quad \text{for} \,\,\lambda\in[0,\infty ),\qquad
				\cos\Big(\frac{4\delta}{\rho}\Big)\quad \text{for} \,\,\delta\in[0,\infty).
				\end{align*}
				Hence we can expand the right side of \eqref{eq:j3.2} as a Maclaurin series in the variable $\delta$ as
				\begin{align*}
					&\cosh^2\Big(\frac{\rho+\lambda}{R_1}\Big)-\sinh^2\Big(\frac{\rho+\lambda}{R_1}\Big)\cos\Big(\frac{4\delta}{\rho}\Big)\nonumber\\=&	\cosh^2\Big(\frac{\rho+\lambda}{R_1}\Big)-\sinh^2\Big(\frac{\rho+\lambda}{R_1}\Big)+\frac{8\sinh^2\big(\frac{\rho+\lambda}{R_1}\big)}{\rho^2}\delta^2-\delta^4\xi(\lambda, \delta)\nonumber\\
				=:&s(\lambda,\delta)-\delta^4\xi(\lambda, \delta).
				\end{align*}	
				It is straightforward to verify by using differentiation that $\xi(\lambda, \delta)$ is a decreasing function of $\lambda$ on $[0,\infty)$, and a decreasing function of $\delta$ on $(0,\rho\pi/4]$. Thus $\xi(\lambda, \delta)>0$ on $[0,\infty)\times(0,\rho\pi/4]$.
				Define $\xi(\lambda,0):=32\sinh^2(\frac{\rho+\lambda}{R_1})/3\rho^4.$
				It follows from direct verification that $\lim_{\delta\rightarrow 0}\xi(\lambda,\delta)=\xi(\lambda,0)$. Hence $\xi(\lambda,\delta)$ is continuous on $[0,\rho\pi/2]\times [0,\rho\pi/4]$.
				Therefore, for $\delta\in[0,\rho\pi/4]$ and $\lambda\in [0,\rho\pi/2]$, there exist positive constants $c_3$ and $c_4$ depending on $\rho$ such that
				$c_3\leq \xi(\lambda,\delta)\leq c_4.$
		Since
				\begin{align*}
					1+\frac{y^2}{2}+\frac{y^4}{24}\eta(y)=s(\lambda,\delta)-\delta^4\xi(\lambda,\delta),
				\end{align*}
				we have
				\begin{align}\label{eq:a2}
					y^2+\frac{y^4}{12}\eta(y)=&2(-1+s(\lambda,\delta)-\delta^4\xi(\lambda,\delta))
					=\frac{16\sinh^2(\frac{\rho+\lambda}{R_1})}{\rho^2}\delta^2-2\delta^4\xi(\lambda, \delta).
				\end{align}
				Hence for $\lambda\in[0,2\delta]$,
			\begin{align*}
					y^2\leq\frac{16\sinh^2(\frac{\rho+\lambda}{R_1})}{\rho^2}\delta^2-2\delta^4\xi(\lambda, \delta)\leq\frac{16\sinh^2(\frac{\rho+2\delta}{R_1})}{\rho^2}\delta^2. 
				\end{align*}
				We can expand $\sinh^2(\frac{\rho+2\delta}{R_1})$ as a Maclaurin series in the variable $\delta$ as
				\begin{align*}
					\sinh^2\Big(\frac{\rho+2\delta}{R_1}\Big)=\sinh^2\Big(\frac{\rho}{R_1}\Big)+\frac{4\cosh\big(\frac{\rho}{R_1}\big)\sinh\big(\frac{\rho}{R_1}\big)}{R_1}\delta+\delta^2\phi(\delta).
				\end{align*}
			Hence
				\begin{align*}
					\phi(\delta)=\frac{-R_1\cosh(\frac{2\rho}{R_1})+R_1\cosh(\frac{2(2\delta+\rho)}{R_1})-4\delta\sinh(\frac{2\rho}{R_1})}{2R_1\delta^2}.
				\end{align*}
			By direct calculation, we have $\phi(\delta)>0$ for $\delta>0$. Define $\phi(0):= 4\cosh\big(\frac{2\rho}{R_1}\big)/R_1^2.$
				It follows from direct verification that $\lim_{\delta\rightarrow 0}\phi(\delta)=\phi(0)$. Hence $\phi(\delta)$ is continuous on $[0,\infty)$. Therefore, for $\delta\in [0,\rho\pi/4] $, there exist positive  constants $c_5$ and $c_6$ such that
				$c_5\leq \phi(\delta)\leq c_6.$
				It follows that for $\delta\in [0,\rho\pi/4]$,
				\begin{align*}
					y^2\leq&\frac{16\big(\sinh^2\big(\frac{\rho}{R_1}\big)+4\delta\cosh\big(\frac{\rho}{R_1}\big)\sinh\big(\frac{\rho}{R_1}\big)/R_1+\delta^2\phi(\delta)\big)}{\rho^2}\delta^2\\
					\leq&c_7\delta^2+c_8\delta^3+c_9\delta^4
					\leq c_{10}\delta^2.
				\end{align*}
				Thus
			$y\leq\sqrt{c_{10}}\delta=:C_2\delta$.
				By \eqref{eq:a2}, for $\delta\in [0,\rho\pi/4]$, we have
				\begin{align*}
					y^2+\frac{y^4}{12}c_2&\geq\frac{16\sinh^2(\frac{\rho+\lambda}{R_1})}{\rho^2}\delta^2-2\delta^4\xi(\lambda, \delta)\geq 	\frac{16\sinh^2(\frac{\rho}{R_1})}{\rho^2}\delta^2-2c_3\delta^4\\
					&=:c_{11}\delta^2(1-c_{12}\delta^2)\geq (c_{11}/2)\delta^2.
				\end{align*}	
				Hence for $y\in [0,L]$,
				\begin{align*}
					y^2&\geq (c_{11}/2)\delta^2/\Big(1+\frac{y^4}{12}c_2\Big)\geq (c_{11}/2)\delta^2/\Big(1+\frac{L^4}{12}c_2\Big)=:c_{12}\delta^2
				\end{align*}
				Thus for $y\in [0,L]$,
					$y\geq C_1\delta$.
		
	The proofs of (b) and (c) are similar; we omit the details.	
				\end{proof}

\begin{proof}[Proof of Proposition \ref{prop:3.0}]\, 
	 Note that for each $p\in M$ and linearly independent $X(p),Y(p)\in T_p M$, we have  $X(p)/|X(p)|$, $Y(p)/|Y(p)|\in S^{T_p M}(0,1)$ and $K(X(p)/|X(p)|,Y(p)/|Y(p)|)=K(X(p),Y(p))$. By the definition of $\underline{K}_M(p)$, we have
\begin{align*}
		\underline{K}_M(p)=\inf\{K(X(p),Y(p)):X(p),Y(p)\in S^{T_p M}(0,1)\,\text{are orthogonal}\}.	
	\end{align*}
Let $(E_1,\ldots,E_n)$ be a smooth local frame on a neighborhood $U$ of $p$ and assume that $|E_i(p)|=1$ for $i=1,\ldots,n$. Then there exist $X_0(p),Y_0(p)\in S^{T_p M}(0,1) $ such that 
	\begin{align*}
		\underline{K}_M(p)=K(X_0(p),Y_0(p))=\sum_{i,j,l,w=1}^na_ib_jb_la_w{\rm Rm}(E_i(p),E_j(p),E_l(p),E_w(p)),
	\end{align*}	
	where ${\rm Rm}$ is the Riemann curvature tensor. ${\rm Rm}(E_i(p),E_j(p),E_l(p),E_w(p))$ is a continuous function of $p$ implies that for any $\epsilon_1>0$, there exists $\delta>0$ such that for any $q\in B^M(p,\delta)$,
	\begin{align}\label{eq:j551}
		|{\rm Rm}(E_i(q),E_j(q),E_l(q),E_w(q))&-{\rm Rm}(E_i(p),E_j(p),E_l(p),E_w(p))|<\epsilon_1.
	\end{align}
	By using \eqref{eq:j551}, we have
	\begin{align*}
		\Big|K\Big(\sum_{i=1}^na_iE_i(q),\sum_{j=1}^nb_jE_j(q)\Big)-K\Big(\sum_{i=1}^na_iE_i(p),\sum_{j=1}^nb_jE_j(p)\Big)\Big|
		<\sum_{i,j,l,w=1}^n|a_ib_jb_la_w|\epsilon_1=:\epsilon. 
	\end{align*}
	Since $\underline{K}_M(q)\leq K(\sum_{i=1}^na_iE_i(q),\sum_{j=1}^nb_jE_j(q))$, it follows that 
	\begin{align}\label{eq:1.110}
		\underline{K}_M(q)<K_M(p)+\epsilon.
	\end{align}
	Suppose $\displaystyle{\varlimsup_{m\rightarrow\infty}}\underline{K}_M(p_m)>\underline{K}_M(p)$. Then there would exist $p_m\in B_{\delta}(p)$ such that $\underline{K}_M(p_m)>\underline{K}_M(p)$. 
	This contradicts \eqref{eq:1.110}. Hence 
	\begin{align}\label{eq:1.111}
		\displaystyle{\varlimsup_{m\rightarrow\infty}}\underline{K}_M(p_m)\leq \underline{K}_M(p).
	\end{align}
	
	We now prove the reverse inequality. For each $X(p_m)\in S^{T_{p_m}M}(0,1)$, we  can write $X(p_m)=\sum_{i=1}^nc_i^{m}E_i(p_m)$. For $m=1,2,\ldots$, define maps  $h_m:S^{T_{p_m}M}(0,1)\rightarrow S^{T_{p}M}(0,1)$ by $h_m(X(p_m)):=\sum_{i=1}^nc_i^{m}E_i(p).$
	The definition of $K_M(p_m)$ implies that 
 there exist $X_m(p_m)$, $Y_m(p_m)\in S^{T_{p_m}M}(0,1)$ such that 
	$$\underline{K}_M(p_m)=K(X_m(p_m), Y_m(p_m))=K\Big(\sum_{i=1}^na_i^mE_i(p_m),\sum_{j=1}^nb_j^mE_j(p_m)\Big).$$
	Thus $h_m(X_m(p_m))=\sum_{i=1}^na_i^{m}E_i(p)\in S^{T_{p}M}(0,1)$.
	By the compactness of $S^{T_{p}M}(0,1)$, there exists a subsequence $\{h_{m_k}(X_{m_k}(p_{m_k}))\}=\{\sum_{i=1}^na_{i}^{m_k}E_i(p)\}$ of  $\{h_m(X_m(p_m))\}$ such that 
	\begin{align}\label{eq:1.1100}
		\lim_{k\rightarrow\infty} h_{m_k}(X_{m_k}(p_{m_k}))=\widetilde{X}(p)\in S^{T_{p}M}(0,1). 
	\end{align}
	Write $\widetilde{X}(p):=\sum_{i=1}^n \alpha_iE_i(p)$. By \eqref{eq:1.1100} and  the continuity of the $E_i$, we have
	\begin{align}\label{eq:1.112}
		\lim_{k\rightarrow\infty}a_{i}^{m_k}= \alpha_i\quad\text{and}\quad \lim_{k\rightarrow\infty}E_i(p_{m_k})= E_i(p),\,\, i=1,\ldots,n.
	\end{align}
	Similarly, for $Y_m(p_m)\in  S^{T_{p_m}M}(0,1)$, we have
	$h_m(Y_m(p_m))=\sum_{j=1}^nb_j^mE_j(p)\in  S^{T_{p}M}(0,1)$, and thus
there exists a subsequence $\{h_{m_k}(Y_{m_k}(p_{m_k}))\}=\{\sum_{j=1}^nb_{j}^{m_k}E_j(p)\}$ of $\{h_m(Y_m(p_m))\}$ such that $\displaystyle{\lim_{k\rightarrow\infty}} h_{m_k}(Y_{m_k}(p_{m_k}))=\widetilde{Y}(p)\in S^{T_{p}M}(0,1)$. 
	Write $\widetilde{Y}(p):=\sum_{j=1}^n\beta_jE_j(p)$. Then
	\begin{align}\label{eq:1.117}
		\lim_{k\rightarrow\infty}b_{j}^{m_k}= \beta_j,\qquad j=1,\ldots,n.
	\end{align}
	Now suppose $\displaystyle{\varliminf_{k\rightarrow\infty}}\underline{K}_M (p_{m_k})<\underline{K}_M(p)$. Since
	\begin{align*}
		\displaystyle{\varliminf_{k\rightarrow\infty}}\underline{K}_M (p_{m_k})
		&=\displaystyle{\varliminf_{k\rightarrow\infty}}K\Big(\sum_{i=1}^na_{i}^{m_k}E_i(p_{m_k}),\sum_{j=1}^nb_{j}^{m_k}E_j(p_{m_k})\Big)\\
		&=K\Big(\sum_{i=1}^n \alpha_i E_i(p),\sum_{j=1}^n\beta_jE_j(p)\Big), \qquad	(\text{by  \eqref{eq:1.112} and \eqref{eq:1.117}})
			\end{align*}
 we have $K(\widetilde{X}(p),\widetilde{Y}(p))<K(X_0(p),X_0(p))$. This contradicts the definition of $\underline{K}_M(p)$. Hence
	\begin{align}\label{eq:1.113}
	\displaystyle{\varliminf_{k\rightarrow\infty}}\underline{K}_M (p_{m_k})\geq \underline{K}_M(p).
	\end{align}
	Combining \eqref{eq:1.111} and \eqref{eq:1.113}, we have $\displaystyle{\lim_{k\rightarrow\infty}}\underline{K}_M (p_{m_k})= \underline{K}_M(p)$. Hence $\underline{K}_M(p)$ is a continuous function of $p$.
	Similarly, 
	$\overline{K}_M(p)$ is a continuous function of $p$. 
	We know that for $X(p),Y(p)\in S^{T_pM}(0,1)$, $K(X(p),Y(p))$ is continuous, i.e., there exist $X_1(p),Y_1(p)$ and $X_2(p),Y_2(p)$ such that 
	$K(X_1(p),Y_1(p))\leq K(X(p),Y(p))\leq K(X_2(p),Y_2(p)).$
Hence for $p\in \overline{\Omega}$, \eqref{sec} holds
\end{proof}	
	
		The proof of Theorem \ref{lem:3.13} makes use of normal coordinate charts.  We give an outline for completeness.
	\begin{proof}[Proof of Theorem \ref{lem:3.13}]\, For $p\in \overline{\Omega}$, let $(U,\varphi)$ be a coordinate chart with $p\in U$, where $U:=B^M(p,\epsilon)$ and $\varphi$ is defined as in \eqref{eq:j3.8}. Let $\widetilde{U}:=\varphi(U)$. First,  notice that $f\in C(\overline{U})$ if and only if $f\circ \varphi^{-1}\in C(\overline{\widetilde{U}})$, and in this case,  $\|f\|_{C(\overline{U})}=\|f\circ\varphi^{-1}\|_{C(\overline{\widetilde{U}})}$. Second, {one can show} that $f\in W^{1,2}_0(U)$ if and only if $f\circ \varphi^{-1}\in W^{1,2}_0(\widetilde{U})$, and establish the equality  $\|f\|_{W^{1,2}_0(U)}=\|f\circ\varphi^{-1}\|_{W^{1,2}_0(\widetilde{U})}$. Third, combining these results with the fact that  $W^{1,2}_0(\widetilde{U})$ is compactly embedded in $C(\overline{\widetilde{U}})$, we see that $W^{1,2}_0(U)$ is compactly embedded in $C(\overline{U})$. Finally, we use partition of unity to show that $W^{1,2}_0(\Omega)$ is compactly embedded in $C(\overline{\Omega})$.
	\end{proof}

	\section{Comments and open problems} \label{S:que}
	\setcounter{equation}{0}
	
	1. Theorem \ref{thm:1.3*} proves that Hodge's theorem holds for an absolutely continuous measure with a positive density. For measures that are not absolutely continuous, it is interesting to know to what extend Hodge's theorem still holds.

	2.	Under the framework of the present paper and the assumption on Green's function in reference \cite{Ngai-Zhao_2024}, the eigenfunctions are continuous. It is interesting to	determine whether the generalization of Yau's conjecture as stated in \eqref{eq:Yau_general} hold.

	\section{Acknowledgements}\label{S:ak} The first author thanks Professor Shing-Tung Yau for conversations regarding the eigenfunctions of Kre\u{\i}n-Feller operators and for point out to him Yau's conjecture  on the nodal sets of eigenfunctions of the Laplacian. The authors thank Shuihong Zhou for some helpful discussions.

	\setcounter{equation}{0}	
		
	\end{document}